\documentclass[12pt]{article}

\usepackage[utf8]{inputenc}
\usepackage[left=2cm, right=2cm,, top=3cm, centering]{geometry}
\usepackage{hyperref}
\usepackage{amsmath}
\usepackage{amsfonts}
\usepackage{amssymb}
\usepackage{amsthm}
\usepackage{mathtools}
\usepackage{stmaryrd} \usepackage[
	scr=boondoxo,
]{mathalpha}
\usepackage{mathabx}
\usepackage{graphicx}
\usepackage[dvipsnames]{xcolor}
\usepackage{enumitem}
\usepackage{ulem}
\usepackage{aligned-overset}
\usepackage{mdframed}
\usepackage{cancel}
\usepackage{placeins}

\usepackage{tikz}

\usepackage{natbib}

\usepackage[draft]{fixme}
\fxsetup{theme=color}
\allowdisplaybreaks[4]

\usepackage{autonum}

\usepackage{hyperref}
\hypersetup{
	colorlinks=true,       linkcolor=blue,         
	citecolor=red}

\newcommand*{\feature}{\phi}

\newcommand*{\flow}{\psi}

\newcommand{\ul}[1]{\underline{#1}}
\newcommand{\ol}[1]{\overline{#1}}

\newcommand*{\transpose}{T}

\DeclarePairedDelimiter{\scp}{\langle}{\rangle}
\DeclarePairedDelimiter{\norm}{\|}{\|}
\DeclarePairedDelimiter{\holder}{[}{]}
\DeclarePairedDelimiter{\set}{\{}{\}}
\DeclarePairedDelimiter{\abs}{|}{|}
\DeclarePairedDelimiter{\floor}{\lfloor}{\rfloor}
\DeclarePairedDelimiter{\ceil}{\lceil}{\rceil}

\makeatletter
\newcounter{constants}
 \newcommand*{\defConst}[1]{\stepcounter{constants}\edef\const@temp{\arabic{constants}}\hypertarget{const:#1}{}\protected@write\@auxout{}{\string\newlabel{const: #1}{{\noexpand\ensuremath{c_{\const@temp}}}{\thepage}{\noexpand\ensuremath{c_{\const@temp}}}{const:#1}{}}}\ensuremath{c_{\const@temp}}}
\makeatother

\newcommand*{\dims}{d}
\DeclareMathOperator*{\argmin}{arg\,min}

\newcommand*{\dimIn}{{d_\mathrm{in}}}
\newcommand*{\dimOut}{{d_\mathrm{out}}}

\newcommand*{\bigO}{\mathcal{O}}

\newcommand*{\frechet}{D}

\newcommand*{\resnet}{F}

\newcommand*{\hDiscr}{h}
\newcommand*{\hInterp}{\ol{h}}
\newcommand*{\hCont}{\mathscr{h}}
\newcommand*{\wZh}{\hat{h}}
\newcommand*{\diffusion}{\sigma}

\newcommand*{\cw}{\mathscr{w}}
\newcommand*{\iw}{\ol{w}}

\newcommand*{\cv}{\mathscr{z}^{\hermRank, \hurst}}

\newcommand*{\iv}{\ol{z}}

\newcommand{\R}{\mathbb{R}}
\newcommand*{\real}{\mathbb{R}}

\newcommand{\nat}{\mathbb{N}}

\newcommand*{\banachSpace}[1][E]{\mathcal{#1}}
\newcommand*{\linOp}[2]{\mathcal{L}(#1, #2)}

\newcommand{\E}{\mathsf{E}}

\DeclarePairedDelimiterXPP{\Exp}[1]{\E}{[}{]}{}{#1}
\DeclareMathOperator{\var}{Var}

\newcommand*{\hermRank}{q}
\newcommand*{\bm}{B}

\newcommand*{\ind}{\mathbf{1}}
\newcommand*{\id}{\mathbb{I}}
\DeclarePairedDelimiterXPP{\vect}[1]{\mathrm{vec}}{(}{)}{}{#1}
\DeclarePairedDelimiterXPP{\prob}[1]{\mathsf{P}}{\{}{\}}{}{#1}

\newcommand*{\lip}{\mathrm{Lip}}
\newcommand*{\slowVar}{\ell}

\newcommand*{\layer}{l}
\newcommand*{\Layer}{L}
\newcommand*{\hurst}{H}

\newcommand*{\magenta}[1]{#1}
\newcommand*{\red}[1]{#1}
\newcommand*{\violet}[1]{#1}
\newcommand{\teal}[1]{#1}

\newlist{steps}{enumerate}{1}
\setlist[steps]{
    label=\textbf{Step \arabic*:},
    ref={Step \arabic*},
    wide=0pt,
}
\newlist{casebycase}{enumerate}{1}
\setlist[casebycase]{
    label=\textbf{Case \arabic*},
    ref={\arabic*},
    wide=0pt,
}
 
\theoremstyle{plain}
\newtheorem{theorem}{Theorem}[section]
\newtheorem*{theorem*}{Theorem}

\newtheorem*{proposition*}{Proposition}
\newtheorem{lemma}[theorem]{Lemma}
\newtheorem{corollary}[theorem]{Corollary}

\theoremstyle{definition}
\newtheorem{definition}[theorem]{Definition}
\newtheorem*{definition*}{Definition}
\newtheorem{remark}[theorem]{Remark}
\newtheorem*{remark*}{Remark}
\newtheorem{example}[theorem]{Example}
\newtheorem{assumption}[theorem]{Assumption}

\title{
Correlated initialization of deep residual networks
}
\author{
	Felix Benning
	\and Ivan Nourdin
	\and Giovanni Peccati
}
\date{\today}

\begin{document}

\maketitle

\begin{abstract}

We study the large-depth behavior of residual networks whose weights are
correlated across layers at initialization. Our results confirm and extend a
conjecture of \citet{marionScalingResNetsLargedepth2025}, according to which
correlated initializations should interpolate continuously between the
Brownian stochastic differential equation arising from independent
initialization and the ordinary differential equation arising from perfectly
correlated initialization.

    When the initialization is obtained from the application of a feature function to a
    stationary Gaussian sequence with regularly varying correlation, we prove
    that there exists a unique critical scaling such that the infinite-depth limit
    is the solution of a Young differential equation driven by a Hermite process.
    Hermite processes reduce to the fractional Brownian motion if the feature function
    generating the initialization has Hermite rank one, which is the case for
    the identity function, for example.  We show that the critical scaling and asymptotic
    limit are uniquely determined by the decay of correlations together with the
    Hermite rank of the feature function. Consequently, the correlation
    structure and Hermite rank of the initialization represent meaningful
    hyperparameters in the asymptotic regime. By contrast, under finite-variance
    iid initialization, the asymptotic driver is universally Brownian up to
    normalization regardless of the choice of distribution.

    Our proofs rely on a collection of novel results establishing a robust stability theory for Young differential equations in
    Banach spaces.
    \smallskip



\end{abstract}
 \section{Introduction}

{\it Residual connections} are one of the principal architectural ideas that made very
deep neural networks trainable. They have become a standard component of
modern architectures, including the {\it Transformer architecture}
\citep{vaswaniAttentionAllYou2017}, which underpins modern large language
models. Early work introduced general shortcut paths between
successive layers \citep{srivastavaTrainingVeryDeep2015}. \citet{heDeepResidualLearning2016}
popularized the idea of residual networks by showing that simple identity
shortcuts enable the training of networks with hundreds or even thousands of layers.
They further found that the identity map performed best among
several shortcut transformations \citep{heIdentityMappingsDeep2016}.
The resulting structure---an identity map perturbed by the residual
output---motivated the interpretation of the residual output as the
rate of change of the hidden state and the network itself
as the numerical approximation of a differential equation
\citep{eProposalMachineLearning2017,haberStableArchitecturesDeep2018,marionImplicitRegularizationDeep2023}.

In this paper, we consider residual networks (ResNets) of depth
\(\Layer\in\nat\), whose hidden states \(h_\layer\) evolve according to
\begin{equation}\label{e:updaterule}
    h_{\layer+1} = h_\layer + \lambda_\Layer \diffusion(w_\layer, h_\layer) v_\layer,
    \qquad \layer=0,\dots,\Layer-1.
\end{equation}
Here, \(\diffusion(w_\layer, h_\layer) v_\layer\) is the residual update produced by the
\(\layer\)-th layer with parameters \(w_\layer\) and \(v_\layer\).
The scaling factor \(\lambda_\Layer>0\) controls the magnitude of this update
and therefore how strongly each layer contributes to the overall transformation
of the input.

It is worth noting that the explicit depth-dependent scaling factor \(\lambda_\Layer\) was not part of
the original ResNet formulation. The residual blocks of
\citet{heDeepResidualLearning2016,heIdentityMappingsDeep2016} instead
incorporated {\it batch normalization} \citep{ioffeBatchNormalizationAccelerating2015} to control
the scale of activations and gradients throughout the network.
Although batch normalization has proved highly effective, it introduces both
theoretical and practical complications: its behavior depends on batch
statistics and therefore on the batch size, it treats training and inference
differently, and it incurs additional memory and communication costs
\citep{brockCharacterizingSignalPropagation2020}. These limitations have
motivated normalization-free architectures in which the magnitude of the
residual updates is controlled directly through depth-dependent scaling or
through an appropriate choice of initialization.

While depth-dependent scaling of order
\(\lambda_\Layer\asymp \Layer^{-1/2}\) repeatedly appears in this literature
\citep[e.g.][]{arpitHowInitializeYour2019,
deBatchNormalizationBiases2020,
shaoNormalizationIndispensableTraining2020,
zhangStabilizeDeepResNet2022}, it is important to highlight that the appropriate choice 
of the scaling factor does not only depend on the depth of the network, but also on
the dependence structure of the parameters across layers.

At one endpoint,
\citet{marionImplicitRegularizationDeep2023} studied ResNets whose \textbf{parameters}
at initialization are discretizations of paths that \textbf{vary smoothly} with the
layer index, with weight-tied initialization providing the simplest perfectly
correlated example. Under the scaling \(\lambda_\Layer\asymp \Layer^{-1}\), such networks are Euler
discretizations of ordinary differential equations (ODEs). Moreover, \citet{marionImplicitRegularizationDeep2023} show that the
smooth dependence of the parameters on the layer index is preserved during
training, so that the trained networks also admit an ODE limit.

Building on findings by \citet{zhangStabilizeDeepResNet2022,
cohenScalingPropertiesDeep2021,contAsymptoticAnalysisDeep2023},
the work by \citet{marionScalingResNetsLargedepth2025} characterizes the
other endpoint of \textbf{parameters initialized independently} across layers. They identified
\(\lambda_\Layer\asymp\Layer^{-1/2}\) as the critical scaling leading to
non-trivial large-depth dynamics and proved that the limiting hidden state is the
solution of an It\^o stochastic differential equation driven by Brownian
motion,
\[
    d\hCont_s
    =
    \diffusion(\cw_s,\hCont_s)\,d\bm_s.
\]
More precisely, scalings larger than \(\Layer^{-1/2}\) lead to explosion,
whereas scalings smaller than \(\Layer^{-1/2}\) suppress the random
fluctuations and cause the network to converge to the identity map. 

The sharp contrast between the ODE regime, obtained when the weights vary
smoothly with the layer index, and the Brownian SDE regime, obtained with
independent weights, naturally raises the question of whether intermediate
correlation structures can interpolate between these two limits.
\citet{marionScalingResNetsLargedepth2025} formulated this conjecture and
tested it experimentally by initializing the weights with increments of
fractional Brownian motion with Hurst parameter
\(\hurst\in(\frac12,1)\). Their experiments suggested that the transition
between explosion and identity occurs at the critical scaling
\[
    \lambda_\Layer\asymp\Layer^{-\hurst}.
\]
As explained in the next section, the principal aim of this paper is to provide
a rigorous and substantially more general answer to this conjecture, extending
it beyond the fractional Brownian setting.

\paragraph*{Our contributions}

\begin{enumerate}
    \item \textbf{Large-depth limits under correlated initialization.} Our
    first contribution (see Theorem \ref{thm:
    large depth limit of resnets with correlated weights, init} below) is a
    rigorous and more general resolution of the conjecture formulated by
\citet{marionScalingResNetsLargedepth2025}. Consider first the Gaussian
setting. For each coordinate \(i\), suppose that
\((v_\layer^i)_{\layer\geq 0}\) (where $v_\layer = (v^1_\layer,...,v^p_\layer)$ is given in  \eqref{e:updaterule}) is a centered stationary Gaussian sequence
with regularly varying covariance
\[
    \E[v^i_0 v^i_k]=: \rho(k)
    =
    k^{-\alpha}\slowVar(k),
    \qquad
    \alpha\in(0,1),
\]
where \(\slowVar\) is slowly varying. Setting
\[
    \hurst
    =
    1-\tfrac{\alpha}{2}
    \in
    \bigl(\tfrac12,1\bigr),
\]
we prove that, under the scaling
\[
    \lambda_\Layer
    =
    \tfrac{\Layer^{-\hurst}}{\slowVar(\Layer)^{1/2}},
\]
the interpolated hidden states converge, as the depth tends to infinity, to
the solution of the Young differential equation
\[
    d\hCont_s
    =
    \diffusion(\cw_s,\hCont_s)\,d\bm_s^\hurst,
\]
where \(\bm^\hurst\) is a fractional Brownian motion with Hurst parameter
\(\hurst\). Thus, fractional Brownian motion need not be built directly into
the initialization through its increments, as in the experiments of
\citet{marionScalingResNetsLargedepth2025}: it arises naturally as the
scaling limit of a broad class of long-range correlated Gaussian
initializations. 

Our main theorem actually covers the more general case in which each coordinate of
\(v_\layer\) is obtained by applying a centered nonlinear function of
{\it Hermite rank} \(q\) (see \eqref{eq: hermite expansion of feature function} and the subsequent discussion) to the underlying Gaussian sequence. Provided
\(\alpha q<1\), the self-similarity parameter and the corresponding scaling
become
\[
    \hurst
    =
    1-\tfrac{\alpha q}{2},
    \qquad
    \lambda_\Layer
    =
    \tfrac{\Layer^{-\hurst}}{\slowVar(\Layer)^{q/2}}.
\]
Under this scaling, the piecewise-linear interpolation of the hidden layers
converges in distribution, in H\"older topology, to the unique solution of
the Young differential equation
\[
    d\hCont_s
    =
    \diffusion(\cw_s,\hCont_s)\,d\cv_s,
    \qquad
    \hCont_0=Ax,
\]
where \(\cv\) is a {\it Hermite process} of rank \(q\) and
self-similarity parameter \(\hurst\); see
Definition~\ref{def: Hermite process}. Fractional Brownian motion is recovered
when \(q=1\), while the case \(q=2\) corresponds to the so-called {\it Rosenblatt process}, as discussed, e.g., by \cite{Tudor13, Tudor23} and \cite{pipirasLongRangeDependenceSelfSimilarity2017}.
We observe that Hermite ranks one and two are especially relevant in
applications: the identity map and many transformations used to generate the
initialization have Hermite rank one, whereas symmetry may force the first
Hermite coefficient to vanish and lead to Hermite rank two
\citep{baiSensitivityHermiteRank2019}; in general, it is not difficult to
construct functions with arbitrarily high Hermite rank.
Theorem~\ref{thm: large depth limit of resnets with correlated weights, init}
therefore identifies a whole family of large-depth limits and shows that both
the appropriate scaling and the limiting dynamics are determined jointly by
the decay of correlations and the Hermite rank of the initialization.

Our limiting theory contributes towards an
\textit{asymptotic theory for selecting initialization hyperparameters},
namely the one-layer distribution of the weights, their dependence structure
across layers, and the depth-dependent scaling factor. Consistently with the
central limit theorem, \citet{marionScalingResNetsLargedepth2025} show that
independent initialization exhibits a strong universality phenomenon: after
centering and normalization, a broad class of
finite-variance distributions leads to the same Brownian-driven SDE in the
large-depth limit, under the scaling $\lambda_L\asymp L^{-1/2}$. In this regime, many details of the one-layer distribution
are therefore asymptotically immaterial. Our main result,
Theorem~\ref{thm: large depth limit of resnets with correlated weights, init},
shows that the picture changes in the presence of long-range dependence: the
decay of correlations and the Hermite rank of the transformation used to
generate the weights jointly determine both the appropriate depth scaling and
the nature of the limiting driver. Thus, once correlations across layers are introduced, not only their
strength and decay, but also the way in which the one-dimensional weight
distribution is generated, become relevant initialization hyperparameters.
Our results do not provide a complete selection procedure, but they identify
which features of the initialization can genuinely alter the infinite-depth
dynamics.

\smallskip

Our results also substantially extend the
work of \citet{hayashiFractionalSDENetGeneration2022}, who introduced
fractional-Brownian-driven neural differential equations to model long-range
dependence in a time-series setting: whereas fractional Brownian motion is postulated there as the
driving noise, in our setting it arises naturally as a large-depth limit of
correlated ResNet initialization, and is further replaced by general Hermite
processes for nonlinear transformations of the underlying Gaussian sequence.

    \item\textbf{Stability and approximation of Young differential equations.}
To prove the ResNet convergence result, we develop a general stability and
approximation theory for parameter-dependent Young differential equations in
Banach spaces of the form
\[
    dx_t=\sigma(t,w_t,x_t)\,dg_t,
\]
where both the driving signal $g$ and the parameter path $w$ are H\"older
continuous with H\"older exponents strictly larger than $1/2$. We establish existence and uniqueness, local Lipschitz continuity
of the solution with respect to the initial condition, the driving signal, and
the parameter path, as well as convergence of Euler approximations in H\"older
topology. The ResNet convergence theorem
(Theorem~\ref{thm: large depth limit of resnets with correlated weights, init})
then follows as a direct application of these continuity and approximation
results, developed in
Section~\ref{sec: young integral equation solution theory}. This theory is of
independent interest beyond the neural-network application. 

\begin{remark}[Related work on Young differential equations] Differential equations driven by paths of finite $p$-variation, with $p<2$,
go back to the foundational work of \citet{youngInequalityHolderType1936} and \citet{lyonsDifferentialEquationsDriven1994}.
Existence, uniqueness, continuity, flow properties, and Euler approximation
for autonomous Young differential equations have been studied in several
works; see, in particular, \citep{huDifferentialEquationsDriven2007,lejayControlledDifferentialEquations2010}. Time-dependent Young differential equations, including equations driven by
fractional Brownian motion, were considered in
\citep{nualartDifferentialEquationsDriven2002,congNonautonomousYoungDifferential2018}. More broadly, Young differential
equations fit within the rough-path framework; see, for instance,
\citep{lyonsDifferentialEquationsDriven1998,bailleulRegularityItoLyonsMap2015} and the systematic presentation in
\citep[Chapter~8]{frizCourseRoughPaths2020}. Our results extend
this literature by treating Banach-space-valued equations with a separate
H\"older parameter path and by providing stability and Euler convergence
directly in H\"older topology, complementing the classical approximation
results for Young and rough differential equations
\citep{davieDifferentialEquationsDriven2008,frizEulerEstimatesRough2008,lejayControlledDifferentialEquations2010}.
    
\end{remark}
\end{enumerate}

 \section{Depth Limit with correlated weights at initialization}
\label{sec: resnet limit}

In this section, we characterize the large-depth limit of ResNets whose weights are
correlated across layers at initialization. After introducing a generalized
ResNet architecture, we specify a class of correlated initializations for
which the partial sums of residual updates converge to a Hermite process. We then
show that, under suitable regularity assumptions on the activation function
and the remaining weights, the interpolated hidden states converge to the
solution of a Young differential equation driven by this process. The proof
combines a functional limit theorem for correlated random walks \citep{benningFunctionalScalingLimits2026} with stability
of Young differential equations and convergence of their Euler
discretizations (Section \ref{sec: young integral equation solution theory}).

\begin{definition}[General ResNet]
    \label{def: generalized resnet}
    A generalized residual network \(\resnet=\resnet_{A, B,
    (v_\layer)_{\layer=0}^{\Layer-1}, (w_\layer)_{\layer=0}^{\Layer-1}}\) 
    with parameters \((A, B, (v_\layer)_{\layer=0}^{\Layer-1}, (w_\layer)_{\layer=0}^{\Layer-1})\)
    maps an input \(x\in \real^{\dimIn}\) through a series of hidden layers
    \begin{align}
        \hDiscr_0 &\coloneq A x
        \\
        \hDiscr_{\layer +1} &\coloneq \hDiscr_\layer + \lambda_\Layer \diffusion(w_\layer, \hDiscr_\layer) v_\layer 
        && \layer \in \set{0, \dots, \Layer -1}
    \end{align}
    to an output \(\resnet(x) \coloneq B\hDiscr_{\Layer} \in \real^{\dimOut}\),
    with
    \begin{itemize}[noitemsep]
        \item \(\lambda_\Layer\in [0, \infty)\), a scaling factor that depends on the depth
        \(\Layer\) of the ResNet,
        \item input and output processing matrices
        \(A\in \real^{\dims\times\dimIn}\)
        and
        \(B\in \real^{\dimOut\times\dims}\),
        \item parameters \(w_\layer\in\banachSpace[W]\), where
        \(\banachSpace[W]\) is a fixed Banach space, and
        \(v_\layer\in\real^r\), which determine the residual update through the
        continuous map
        \[
            \diffusion\colon
            \banachSpace[W]\times\real^\dims
            \longrightarrow
            \real^{\dims\times r}.
        \]
\end{itemize}
\end{definition}

\begin{example}[Classic ResNet]
    Classically, a ResNet is of the form
    \[
        \hDiscr_{\layer+1}
        = \hDiscr_\layer + \lambda_\Layer V_\layer \varphi(W_\layer \hDiscr_\layer + b_\layer),
    \]
    with activation function \(\varphi\colon \real\to \real\) applied
    component-wise and parameter matrices \(W_\layer\in \real^{m \times \dims}\)
    and \(V_\layer\in \real^{\dims \times m}\) and a bias vector \(b_\layer\).
    This is a special case of the general ResNet (Definition \ref{def: generalized resnet}) with
    \[
            V_\layer \varphi(W_\layer\hDiscr_\layer +  b_\layer)
            = \underbrace{(\varphi(W_\layer \hDiscr_\layer + b_\layer)^\transpose \otimes \id_{\dims})}_{
                \displaystyle\eqcolon \diffusion(\underbrace{(W_\layer, b_\layer)}_{\eqcolon w_\layer}, \hDiscr_\layer)
            } \underbrace{\vect{V_\layer}}_{\displaystyle\eqcolon v_\layer \mathrlap{\in \real^{m\dims}}}
            = \diffusion(w_\layer, \hDiscr_\layer) v_\layer,
    \]
    where \(\vect{A}\) stacks the columns of the matrix \(A\) into a vector and
    \(\otimes\) is the Kronecker product \citep[see
    e.g.][Prop.~2]{kschischangKroneckerProductVectorization2022}.
\end{example}

We now specify the model for correlations between the parameters
\(v_\layer\) across layers. For each coordinate, we obtain \(v_\layer\) by
applying a feature function to a stationary Gaussian sequence with regularly
varying covariance. The construction is most transparent for the identity
feature function, in which case \(v_\layer\) is itself a stationary Gaussian
sequence. Allowing more general feature functions can yield examples of ResNets whose
scaling limits are solutions of stochastic differential equations driven by
Hermite processes instead of the fractional Brownian motion.

\begin{definition}[Correlated initialization]
    \label{def: correlated initialization}
    For a feature function \(\feature \colon \real \to \real\) we define
    \[
        v_\layer^i \coloneq \feature(\xi_\layer^i), \qquad i\in \set{1, \dots, r}, \layer \in \set{0, \dots, \Layer -1},
    \]
    where \((\xi_\layer^{i})_{\layer \in \nat_0}\) are stationary Gaussian
    sequences in \(\real\), independent over \(i\), with zero mean \(\E[\xi_k^i] = 0\), unit variance
    \(\var(\xi_k^i) = 1\) and regularly-varying correlation of index \(\alpha \in (0, 1)\)
    \[
        \rho(k) \coloneq \E[\xi_\layer^i \xi_{\layer+k}^i] = k^{-\alpha}\slowVar(k), \quad k\geq 1,l\geq 0,
    \]
    where \(\slowVar\) is a slowly varying function.
\end{definition}
\begin{remark}{\rm
The centering condition
\begin{equation}
    \label{eq: centering condition init} 
    \E[v_\layer^i] = \E[\feature(\xi_\layer^{i})] = 0
\end{equation}
is essential to obtain an intermediate scaling. If the mean were nonzero, its
contribution would accumulate over the \(\Layer\) layers and require the scaling
\(\Layer^{-1}\). This scaling would average out the random fluctuations and the
deterministic mean would dominate the limit. Observe that, for the identity
feature function \(\feature(x)=x\), the centering condition
\eqref{eq: centering condition init} follows directly from the assumption
\(\E[\xi_\layer^i] = 0\).}
\end{remark}

Let us first retain the assumption \(\feature(x)=x\). The linearly interpolated
partial sums of the correlated initializations then form the interpolated
correlated Gaussian random walk
\begin{equation}
    \label{eq: random walk with correlated increments}
    \iv_s^\Layer
    \coloneq
    \lambda_\Layer
    \Bigl(
        \sum_{\layer=0}^{\floor{\Layer s}-1} v_\layer
        +
        \underbrace{
            (\Layer s-\floor{\Layer s})v_{\floor{\Layer s}}
        }_{\text{linear interpolation}}
    \Bigr), \qquad s\in [0,1].
\end{equation}
For
\[
    \hurst=1-\tfrac{\alpha}{2},
    \qquad
    \lambda_\Layer
    =
    \tfrac{\Layer^{-\hurst}}{\slowVar(\Layer)^{1/2}},
\]
this process converges in H\"older topology to a fractional Brownian motion
with Hurst parameter \(\hurst\)
\citep{benningFunctionalScalingLimits2026}. Thus, the identity case,
\[
    \feature(x)=x
    \qquad\text{leads to}\qquad
    \iv^\Layer\overset{d}\to\bm^\hurst,
\]
where $\overset{d}\to$ indicates convergence in distribution in an appropriate topology. This is the correlated analogue of the independent Gaussian setting, where
the initialization variables can be viewed as increments of a Brownian
motion, which then drives the infinite-depth limit. 

To determine what replaces fractional Brownian motion for more general feature
functions, we use the {\it Hermite rank} of \(\feature\). If
\(\feature(\xi_\layer^i)=v_\layer^i\in L^2(\Omega)\), then \(\feature\)
admits an expansion in the Hermite polynomials $(H_k)_{k\geq 0}$ (see e.g. \cite[Chapter 1]{bluebook}),
\begin{equation}
    \label{eq: hermite expansion of feature function}
    \feature(x)
    =
    \sum_{k=\hermRank}^{\infty}c_kH_k(x),
    \qquad
    c_{\hermRank}\neq0, \qquad \sum_{k=\hermRank}^\infty k! c_k^2<\infty.
\end{equation}
The Hermite rank \(\hermRank\) of $\feature$ is therefore the smallest index corresponding to a
non-zero coefficient in this expansion. Note that the centering assumption
\(\E[\feature(\xi_\layer^i)]=0\) in
\eqref{eq: centering condition init} is equivalent to
\(\hermRank\geq1\). Since \(H_1(x)=x\), the identity feature has Hermite rank
one, as do many commonly occurring functions
\citep{baiSensitivityHermiteRank2019}. 

For a feature function of generic Hermite rank \(\hermRank\), provided
\(\alpha\hermRank<1\), the correlated random
walk requires a different normalization and has, in general, a non-Gaussian
limit. More precisely, setting
\[
    \hurst
    =
    1-\tfrac{\alpha\hermRank}{2},
    \qquad
    \lambda_\Layer
    =
    \tfrac{\Layer^{-\hurst}}
         {\slowVar(\Layer)^{\hermRank/2}},
\]
the process \(\iv^\Layer\) converges in H\"older topology to a Hermite process
\((\cv_s)_{s\in[0,1]}\) of rank \(\hermRank\) and self-similarity parameter
\(\hurst\) (see Definition~\ref{def: Hermite process}, as well as \citep[Chapter 3]{Tudor13}, \citep[Chapter 2]{Tudor23} and
\citep{benningFunctionalScalingLimits2026}). In other words,
\[
    \mbox{Hermite rank of } \feature=\hermRank
    \qquad\text{leads to}\qquad
    \iv^\Layer\overset{d}\to\cv.
\]
The Hermite process of rank \(\hermRank=1\) is the fractional Brownian motion.

\smallskip

The significance of this functional limit for the ResNet is that
\(\iv^\Layer\) plays the role of the cumulative driving signal in the
residual recursion. Once the remaining parameters \(w_\layer\) are shown to
approximate a sufficiently regular path \(\cw\), the stability results
developed below will allow us to pass to the limit in this recursion and prove
that the interpolated hidden states converge to the solution of a {\it Young-type differential equation} (see Remark \ref{r:young} for details)
\[
    d\hCont_s
    =
    \diffusion(\cw_s,\hCont_s)\,d\cv_s,
    \qquad
    \hCont_0=Ax.
\]
We therefore consider the piecewise-linear interpolation of the remaining
parameters,
\[
    \iw_s^\Layer
    =
    w_{\floor{\Layer s}}
    +
    \underbrace{
        (\Layer s-\floor{\Layer s})
        \bigl(
            w_{\floor{\Layer s}+1}
            -
            w_{\floor{\Layer s}}
        \bigr)
    }_{\text{linear interpolation}}
\]
and assume that \(\iw^\Layer\) converges in distribution, in H\"older
topology, to some \(\cw\) as \(\Layer\to\infty\).
The overall mechanism
can therefore be summarized as
\[
    \iv^\Layer\overset{d}\to\cv,
    \qquad
    \iw^\Layer\overset{d}\to\cw
    \qquad\text{implies}\qquad
    \hInterp^\Layer\overset{d}\to\hCont,
\]
where the written implication follows from the stability of the Young
differential equation. We observe that two families of parameters play distinct roles:
the variables \(v_\layer\) represent increments of the limiting driver
\(\cv\), whereas the parameters \(w_\layer\) approximate the values
\(\cw_{\layer/\Layer}\) of the limiting parameter path.

\medskip

Next we state the required regularity assumptions about the activation function \(\diffusion\)
for our main result. See Remark~\ref{r:boundedsigma} for a discussion of the technical role played by the boundedness of \(\sigma\) in our proofs.

\begin{assumption}[Regularity of the activation function]
    \label{assmpt: regularity of the activation function}
    The function \(\diffusion\colon \banachSpace[W] \times \real^\dims  \to \real^{\dims \times r}\) is bounded, \(\norm{\diffusion}_\infty < \infty\), and
    locally Lipschitz continuous with locally Lipschitz continuous derivatives in the second variable. That is,
    there exists a continuous function \(c\colon \real^2 \to [0,\infty)\) such that
    \[
        \begin{aligned}
        \abs{\sigma(w, h) - \sigma(\tilde w, \tilde h)}
        &\le c(\abs{w}, \abs{\tilde w})\Bigl(\abs{h-\tilde h} + (1+\abs{h}+\abs{\tilde h})\abs{w-\tilde w}\Bigr)
        \\
        \abs{\nabla_h \sigma(w, h) - \nabla_h \sigma(\tilde w, \tilde h)}
        &\le \underbrace{c(\abs{w}, \abs{\tilde w})}_{\text{locally bounded}}\underbrace{\Bigl(\abs{h-\tilde h} + (1+\abs{h}+\abs{\tilde h})\abs{w-\tilde w}\Bigr)}_{\text{local Lipschitz control}}
        \end{aligned}
    \]
\end{assumption}

The assumption above is sufficient if \(w_{\floor{\Layer s}} = \cw_s\) for a
fixed Hölder continuous process \(\cw\).  If we want \(w_{\floor{\Layer s}} =
\iw^\Layer_s\) with \(\iw^\Layer \overset{d}\to \cw\) in Hölder space, then we
need the following additional regularity assumption.

\begin{assumption}[Additional regularity]
    \label{assmpt: additional regularity}
    The activation function is differentiable in the first variable, with a gradient that is also locally Lipschitz.
    More precisely, there exists a continuous function \(\tilde c \colon \real^4 \to [0, \infty)\) such that
    \[
        \abs{\nabla_w \sigma(w, x) - \nabla_w \sigma(\tilde w, \tilde x)}
        \le \tilde c(\abs{w}, \abs{\tilde w}, \abs{x}, \abs{\tilde x})\bigl(\abs{x-\tilde x} + \abs{w-\tilde w}\bigr).
    \]
\end{assumption}

\begin{example}[Sufficiently nice activation function]
    \label{ex: sufficiently nice function}
    If \(\psi \colon \real \to \real\) is a continuously differentiable activation function such that
    \(\psi\) and \(\psi'\) are bounded and \(\psi'\) is Lipschitz continuous (e.g.\ \(\psi''\) is bounded),
    then the generalized activation function
    \[
        \sigma(w,x) \coloneq \psi(Wx + b)^\transpose \otimes \id_\dims
        \qquad w=(W, b) \in \real^{m\times \dims} \times \real^m, x\in\real^\dims
    \]
    satisfies Assumption \ref{assmpt: regularity of the activation function} and \ref{assmpt: additional regularity}. Examples for such an activation function \(\psi\) are sigmoid functions such as
    \[
        \psi(x) \in \set[\big]{
         \tanh(x),
         \arctan(x),
         \tfrac1{1+e^{-x}},
         \mathrm{erf}(x)
        }.
    \]

\end{example}

Before stating our main result, Theorem~\ref{thm: large depth limit of resnets with correlated weights, init},
we formally introduce Hermite processes in Definition~\ref{def: Hermite process} and Remark \ref{r:propertieshermite},
and briefly recall the notion of a Young differential equation in
Remark~\ref{r:young}.
\begin{definition}[Hermite process; {see e.g. \citep[Def. 3.1]{Tudor13}}]
    \label{def: Hermite process}
    The rank \(\hermRank\) {\it Hermite process with Hurst index} \(\hurst \in (\frac12, 1)\) is defined as
    \[
        Z^{\hermRank,\hurst}_t \coloneq A_{\hermRank, \hurst} \int_{\real^\hermRank}' \int_0^t \prod_{j=1}^\hermRank (s-x_j)_+^{-(\frac12 + \frac{1-\hurst}{\hermRank})} ds\, W(dx_1)\cdots W(dx_\hermRank),\quad t\geq 0,
    \]
    where \(W\) is the Wiener Gaussian white noise measure and \(\int'_{\real^\hermRank} (...) W(dx_1)\cdots W(dx_\hermRank)\) indicates a {\it multiple Wiener-Itô integral} of order $\hermRank$ (see \cite[Section 2.7]{bluebook}). The normalizing
    constant \(A_{\hermRank, \hurst}\) is selected so that
    \(\Exp{(Z^{\hermRank,\hurst}_1)^2} = 1\) and is known explicitly \citep[Proposition 3.1]{Tudor13}.
    Equivalent representations are given e.g. in \cite[Section 3.1.2]{Tudor13} and  \citep[Cor.~4.2.11]{pipirasLongRangeDependenceSelfSimilarity2017}.
\end{definition}

\begin{remark}\label{r:propertieshermite}
    The following facts are well-known (see e.g. \cite[Section 3.1.1]{Tudor13}):
    \begin{itemize}
      \item[(i)] for every $H\in (\frac12, 1)$ and $q\geq 1$, the process $Z^{q,H}$ is $H$-{\it self-similar}, that is: for every $c>0$, $(Z_{ct}^{q,H})_{t\geq 0}$ and $(c^HZ_t^{q,H})_{t\geq 0}$ have the same law; 
      \item[(ii)] For every $q\geq 1$, $Z^{q,H}$ is centered, has stationary increments and its covariance is given by
      \begin{equation}\label{e:fbmc}
      \E[Z_t^{q,H}Z_s^{q,H}] = \frac12\left\{ t^{2H}+s^{2H}- |t-s|^{2H}\right\}, \quad s,t\geq 0;
      \end{equation}
      \item[(iii)] For every $\gamma\in (0,H)$, the process $Z^{q,H}$ admits a modification whose sample paths are locally $\gamma$-H\"older continuous with probability one.
    \end{itemize}
    It can be shown that if \(\hermRank=1\) then,
    \(Z^{1,\hurst}=\bm^\hurst\) is a standard {\it fractional Brownian motion} with Hurst index $H$ (that is, $Z^{1,\hurst}$  is a centered Gaussian process with covariance \eqref{e:fbmc}). The Hermite process with rank \(\hermRank=2\) corresponds to the so-called
    {\it Rosenblatt process} (see e.g. \cite[Section 3.2]{Tudor13} or \cite[Section 2.3.2]{Tudor23}).
\end{remark}

\begin{remark}[Young Differential Equations]\label{r:young}

    In this paper, an equation of the form
    \[
        d h_s=\diffusion(w_s,h_s)\,d z_s,
        \qquad h_0=a,
    \]
    where $z$ is typically a Hermite process, is understood pathwise in the Young sense. More precisely, a stochastic
    process \(h\) is a solution if, for almost every realization of
    \((w,z)\),
    \[
        h_t
        =
        a+\int_0^t \diffusion(w_s,h_s)\,d z_s,
        \qquad t\in[0,1].
    \]
    Here, the integral is the {\it Young integral}: if \(z\) is
    \(\beta\)-H\"older continuous and
    \(s\mapsto\diffusion(w_s,h_s)\) is \(\eta\)-H\"older continuous, with $\eta, \beta\in (0,1]$ and
    \(\eta+\beta>1\), then
    \[
        \int_0^t \diffusion(w_s,h_s)\,d z_s
        =
        \lim_{\lvert\pi\rvert\to0}
        \sum_{[u,v]\in\pi}
        \diffusion(w_u,h_u)\bigl(z_v-z_u\bigr),
    \]
    where \(\pi\) ranges over partitions of \([0,t]\) and
    \(\lvert\pi\rvert\) denotes their mesh. In particular, the integral and
    the resulting differential equation are defined pathwise, rather than in
    the It\^o sense. Since a Hermite process with self-similarity parameter \(H>\frac12\) has
sample paths that are \(\gamma\)-H\"older continuous for every \(\gamma<H\) (see Remark \ref{r:propertieshermite}),
the limiting drivers considered below fall within the Young framework. We refer to
    \citet[Chapter~8]{frizCourseRoughPaths2020} for further details.
\end{remark}

The (standard) functional spaces appearing in the following theorem are formally introduced in Definition \ref{d:functionalspaces}.

\begin{mdframed}[innertopmargin=0pt]
\begin{theorem}[Large depth limit of ResNets with correlated weights at initialization]
    \label{thm: large depth limit of resnets with correlated weights, init}
    Let \(\resnet = \resnet_{A, B, (v_\layer)_{\layer=0}^{\Layer-1},
    (w_\layer)_{\layer=0}^{\Layer-1}}\) be a ResNet as in Definition \ref{def: generalized resnet}
    with an activation function \(\diffusion\) that satisfies Assumption \ref{assmpt: regularity of the activation function}
    and correlated initialization of \(v_\layer\) as in Definition~\ref{def:
    correlated initialization}. Let \((w_\layer)_{\layer=0}^{\Layer-1}\) be initialized, independently of
    \((v_\layer)_{\layer=0}^{\Layer-1}\), with either
    \begin{enumerate}[noitemsep, label=(\roman*)]
        \item[\rm (i)] \(w_\layer = \cw_{\frac{\layer}{\Layer}}\) for some \(\cw \in C^{\beta_w}([0,1], \banachSpace[W])\) with \(\beta_w \in (\frac12,1]\), or
        \item[\rm (ii)] \(w_\layer = \iw^\Layer_{\frac{\layer}{\Layer}}\) for \(\iw^\Layer \overset{d}\to \cw\) in \(C^{\beta_w}([0,1], \banachSpace[W])\) with \(\beta_w \in (\frac12,1]\)
        and Assumption \ref{assmpt: additional regularity} is satisfied.
    \end{enumerate}
    Assume that \(v_\layer^i \in L^p(\Omega)\) for some \(p > \frac{2}{1-\alpha
    \hermRank}\) and \(\alpha \in (0, \frac1\hermRank)\), where \(\hermRank\ge 1\) is
    the Hermite rank of the feature function \(\feature\) that produces \(v_\layer^i\).
    Then for the scaling
    \[
        \lambda_\Layer = \tfrac{\Layer^{-\hurst}}{\slowVar(\Layer)^{\frac{\hermRank}2}},
    \]
    we have for all \(\beta \in (0, \min\set{\hurst - \frac1p, \beta_w})\) with \(\hurst = 1 - \frac{\alpha\hermRank}{2}\in(\frac12,1)\) that
    \[
        \iv^\Layer \overset{d}\to \cv, \text{ in } C^\beta([0,1], \R^r),
        \quad\text{and}\quad
        \hInterp^\Layer \overset{d}\to \hCont,
        \quad \text{in } C^\beta([0,1], \R^\dims),
        \quad \text{as} \quad \Layer \to \infty,
    \]
    where 
    \begin{itemize}[noitemsep]
        \item \(\iv^\Layer\) is the interpolated sum process of the parameters \(v_\layer\) defined in \eqref{eq: random walk with correlated increments},
        \item \(\cv=\gamma (Z^{\hermRank,\hurst,1}, \ldots,
        Z^{\hermRank,\hurst,r})\) are independent Hermite processes (Def.~\ref{def: Hermite process})
        \(Z^{\hermRank,\hurst,i}\) of rank \(\hermRank\) and self-similarity
        parameter \(\hurst\) scaled by \(\gamma^2 = c_\hermRank^2
        \hermRank!/(\hurst(2\hurst-1))\) with \(c_\hermRank\) from \eqref{eq:
        hermite expansion of feature function},
        \item \(\hInterp^\Layer\) is the piecewise linear interpolation of the
        hidden layers \(\hDiscr_\layer\), that is
        \[
            \hInterp_s^\Layer \coloneq \hDiscr_{\floor{\Layer s}} + \underbrace{(\Layer s - \floor{\Layer s})(\hDiscr_{\floor{\Layer s}+1} - \hDiscr_{\floor{\Layer s}})}_{
                \text{linear interpolation}
            }
        \]
        \item 
        \(\hCont\) is the unique solution of the differential equation
        \[
            d\hCont_s = \diffusion(\cw_s, \hCont_s) d\cv_s \quad \text{with} \quad \hCont_0 = Ax,
        \]
        which is a.s. contained in \(C^{\beta}([0,1], \real^\dims)\).
    \end{itemize}
\end{theorem}
\end{mdframed}

\begin{proof}[Sketch of the proof]
    The convergence of \(\iv^\Layer\) to \(\cv\) in Hölder space follows from a functional limit theorem for correlated random walks
    \citep{benningFunctionalScalingLimits2026}. To get convergence of \(\hInterp^\Layer\) to \(\hCont\) we essentially apply a triangle
    inequality in the following way: We define \(\wZh^\Layer\) as the solution of the differential equation
    \begin{equation}
        \label{eq: wong-zakai differential equation}    
        d\wZh^\Layer_s = \diffusion(w^\Layer_s, \wZh^L_s) d\iv^\Layer_s \quad \text{with} \quad \wZh^\Layer_0 = Ax,
    \end{equation}
    with \(w^\Layer \in \set{\cw, \iw^\Layer}\) depending on the initialization assumption.
    With \(\Layer \to \infty\) the parameters of this Wong-Zakai approximation of the limiting Young differential equation (see \cite[Section 9.2]{frizCourseRoughPaths2020})
    converge to the parameters of the original differential equation that define \(\hCont\). So we get
    convergence of \(\wZh^\Layer\) to \(\hCont\) by stability results about Young
    differential equations (Theorem \ref{thm: differential equation solution
    existence and uniqueness} and Corollary \ref{cor: solution is locally
    Lipschitz in inputs}). With \(\wZh^\Layer \to \hCont\) established we then essentially show that
    the difference between \(\wZh^\Layer\) and \(\hInterp^\Layer\) vanishes
    asymptotically.
    Since \(\hInterp^\Layer\) is simply the Euler discretization of the
    Wong-Zakai differential equation \eqref{eq: wong-zakai differential
    equation}, the difference between \(\wZh^\Layer\) and \(\hInterp^\Layer\)
    is controlled by our general result about the convergence of the Euler method
    for Young differential equations (Theorem \ref{thm: convergence of euler scheme}).
\end{proof}

The proof of Theorem \ref{thm: large depth limit of resnets with correlated weights, init}
consequently relies on general stability results about Young differential
equations.  These results of independent interest are the
content of Section~\ref{sec: young integral equation solution theory}.

Before moving on we highlight the natural conjecture that the scaling
\begin{equation}
    \label{eq: critical scaling for correlated resnets}
    \lambda_\Layer \asymp \frac{\Layer^{-\hurst}}{\slowVar(\Layer)^{\frac{\hermRank}2}}
\end{equation}
is necessary for a non-trivial limit of the ResNet at initialization.
Larger, super-critical scaling should lead to a blow-up of the hidden states, while
smaller sub-critical scaling should lead to a trivial limit.
The following corollary shows the latter. 

\begin{corollary}[Sub-critical scaling]
    \label{cor: trivial limit for sub-critical scaling}
    Assume the setting of Theorem \ref{thm: large depth limit of resnets with correlated weights, init}.
    In particular, let \(\lambda_\Layer \coloneq 
    \frac{\Layer^{-\hurst}}{\slowVar(\Layer)^{\hermRank/2}}\) be the
    standard scaling. Let
    \(\hInterp^{\Layer, \dagger}_s\) be the interpolated hidden states of the ResNet with
    a different scaling \(\lambda_\Layer^\dagger\) with initial condition \(h_0=Ax\).
    Then for all \(\beta \in (0, \min\set{\hurst - \frac1p, \beta_w})\)
    \begin{alignat}{3}
        \lim_{\Layer \to \infty}\frac{\lambda_\Layer^\dagger}{\lambda_\Layer} = 0
        &\qquad\implies\qquad &
        \hInterp^{\Layer, \dagger} \overset{p}\to (s\mapsto h_0), \qquad \text{in } C^\beta([0,1], \real^\dims).
        \tag{Identity limit}
    \end{alignat}
\end{corollary}

\begin{proof}
    First, we will show that instead of replacing the standard scaling \(\lambda_\Layer\)
    by \(\lambda_\Layer^\dagger\) we can equivalently replace drivers
    \(v_\layer\) by \(v_\layer^\dagger =
    \frac{\lambda_\Layer^\dagger}{\lambda_\Layer} v_\layer\) and keep the
    standard scaling \(\lambda_\Layer\) to obtain the hidden states \(\hInterp^{\Layer, \dagger}_s\).
    This is a simple consequence of linearity:
    \[
        h_{\layer+1}^\dagger
        = h_\layer + \lambda_\Layer^\dagger \diffusion(w_\layer, h_\layer^\dagger) v_\layer
        = h_\layer + \lambda_\Layer \diffusion(w_\layer, h_\layer^\dagger) v_\layer^\dagger.
    \]
    In turn, we have that the interpolated sum process of the new drivers \(v_\layer^\dagger\)
    converges to zero in distribution in Hölder space:
    \[
        \iv^{\Layer, \dagger}_s
        = \sum_{\layer=0}^{\floor{\Layer s}-1} \lambda_\Layer v_\layer^\dagger + (\Layer s - \floor{\Layer s}) \lambda_\Layer v_{\floor{\Layer s}}^\dagger
        = \frac{\lambda_\Layer^\dagger}{\lambda_\Layer} \iv^\Layer_s
        \overset{d}\to 0 \eqcolon \mathscr{z}_s^\dagger,
    \]
    The convergence follows from the convergence of \(\iv^\Layer\) to \(\cv\) in Theorem \ref{thm: large depth limit of resnets with correlated weights, init}
    combined with Slutsky's theorem \citep[e.g.][Thm.~13.18]{klenkeProbabilityTheoryComprehensive2014}
    to get the joint convergence of \((\iv^\Layer, \frac{\lambda_\Layer^\dagger}{\lambda_\Layer})\) to \((\cv, 0)\) in distribution
    and an application of the continuous mapping theorem
    \citep[e.g.][Thm.~13.25]{klenkeProbabilityTheoryComprehensive2014}.
    Since we only use convergence of \(\iv^\Layer\) against a limiting process
    \(\cv\) in the proof of Theorem \ref{thm: large depth limit of resnets with
    correlated weights, init} (proven in \ref{step: driver convergence}), the remaining steps
    of the proof yield convergence of \(\hInterp^{\Layer, \dagger}\) to the
    solution of the differential equation
    \[
        d\hCont_s = \diffusion(\cw_s, \hCont_s) d\mathscr{z}_s^\dagger, \quad \hCont_0 = Ax,
    \]
    with \(\mathscr{z}_s^\dagger = 0\) for all \(s\in [0,1]\). However the solution
    to this differential equation is simply \(\hCont_s = Ax\) for all \(s\in [0,1]\).
    Since convergence in distribution against a constant implies
    convergence in probability, we also have \(\hInterp^{\Layer, \dagger}_s \overset{p}\to Ax\)
    in Hölder space.
\end{proof}

\begin{remark}[Super-critical scaling]
    \label{rem: super-critical scaling}
    The conjectured blow-up for super-critical scaling is more difficult to prove.
    However, the same argument as in the proof above may be used to show that
    the driver of the limiting differential equation is multiplied by a
    diverging factor \(\lambda_\Layer^\dagger/\lambda_\Layer\). This alone
    does not imply blow-up of the hidden states however, because
    the diffusion term \(\diffusion(\cw_s, \hCont_s)\) may suppress this
    amplification. A trivial example is \(\diffusion=0\). A more sophisticated one is
    \(\diffusion(h) = (1-\norm{h}^2)_+^2\), which prevents \(\norm{h}\) from
    exceeding \(1\). A proof of the conjectured blow-up for super-critical scaling
    therefore requires appropriate lower bounds on the activation function \(\diffusion\) and is
    left for future work. For independent initializations, \citet{marionScalingResNetsLargedepth2025}
    introduce their Assumption \(A_2\) for this purpose.
\end{remark}

\section{Discussion and experiments}

Our analysis concerns the large-depth behavior of ResNets at initialization and
rigorously resolves a conjecture of
\citet{marionScalingResNetsLargedepth2025}. During the preparation of this work,
\citet{chizatHiddenWidthDeep2026} developed a complementary analysis of the
large-depth behavior of trained ResNets under i.i.d.\ initialization. His
results show, in particular, that a scaling which is critical at initialization
need not coincide with the scaling leading to maximal local parameter updates
during training.

In this section, we briefly review the main mechanism underlying Chizat's
analysis and formulate a conjectural extension of this analysis in the correlated setting
considered in the present paper. We emphasize that Chizat's results do not
directly apply to our model, since our initialization is correlated across
layers. Nevertheless, they suggest a natural training phase diagram in which
our critical initialization regime appears as the boundary of a locally
linearized regime. This interpretation is partially supported by the
experiments described in Section~\ref{sec: experiments}.

\subsection{Critical initialization versus trainability}
\label{sec: interpretation critical initialization vs trainability}

\subsubsection{Overview of \texorpdfstring{\citet{chizatHiddenWidthDeep2026}}{Chizat [2026]}}

A non-trivial random output at initialization is not necessarily the
right criterion for choosing the scaling of a trainable model. The main object
of interest is the output after training. This distinction is emphasized by
\citet{chizatHiddenWidthDeep2026}, who organizes the large-depth behavior of
trained ResNets in a phase diagram depending on the scaling
\(\lambda_\Layer\) (his Figure~4). In this subsection, we restrict ourselves
to the setting considered by Chizat, in which the trainable parameters are
initialized independently across layers. In this framework, Chizat shows that residual scalings less
aggressive than \(\Layer^{-1}\), while remaining below the stochastic critical
scale \(\Layer^{-1/2}\), lead to what he calls the ``lazy-ODE'' regime: the
displacement of each layer's parameters vanishes and each residual layer
becomes asymptotically linear in its parameters. At the boundary scaling \(\Layer^{-1/2}\),
\citet{marionScalingResNetsLargedepth2025} prove that the random fluctuations
at initialization remain of order one and give rise to a Brownian-driven SDE
in the large-depth limit, while Chizat expects the locally linearized training
mechanism to persist at this boundary, although with a stochastic limiting
dynamics different from the lazy-ODE regime. To understand the mechanism behind this phenomenon, define
\[
    f(h_l,z_l)\coloneq \diffusion(w_l,h_l)v_l,
    \qquad
    z_l\coloneq(w_l,v_l),
\]
where \(z_l\) denotes the trainable parameters at initialization. For centered
i.i.d.\ parameters across layers, at the critical stochastic scaling
\(\lambda_\Layer=\Layer^{-1/2}\) we have
\[
    h_\Layer
    =
    h_0
    +
    \sum_{l=0}^{\Layer-1}
    \underbrace{
        \Layer^{-1/2}f(h_l,z_l)
    }_{=\bigO(\Layer^{-1/2})}
    =
    \bigO(1).
\]
Thus, the \(\Layer\) increments of order \(\Layer^{-1/2}\) sum to a quantity
of order one by a stochastic averaging effect. Without such an averaging
effect, increments would need to be of order \(\Layer^{-1}\) in order to
accumulate to a quantity of order one.

The situation is different for the changes in the parameters induced by
training. Since the training of the parameter \(z_l\) causes a highly
structured change \(\Delta z_l\) to the parameters \(z_l\), these changes
should not be expected to benefit from the same averaging effect. Using
\(h_l(t)\) for the hidden layer at training time \(t\), we thus heuristically
obtain, by Taylor expansion of \(f(h_l(t),z_l(t))\) around the initial
parameters \(z_l\),
\begin{equation}
    \label{eq: taylor expansion of resnet update}
    h_\Layer(t)
    =
    h_0
    +
    \underbrace{
        \sum_{l=0}^{\Layer-1}
        \Layer^{-1/2}f(h_l(t),z_l)
    }_{\bigO(1)}
    +
    \sum_{l=0}^{\Layer-1}
    \Layer^{-1/2}
    \partial_z f(h_l(t),z_l)\Delta z_l(t)
    +
    \sum_{l=0}^{\Layer-1}
    \bigO\!\left(
        \Layer^{-1/2}\|\Delta z_l(t)\|^2
    \right).
\end{equation}
Since the increments
\[
    \Delta z_l(t)=z_l(t)-z_l
\]
are highly structured across layers, the second term is of order one when
\(\Layer^{-1/2}\Delta z_l(t)\) is of order \(\Layer^{-1}\). This requires
\begin{equation}
    \label{eq: trained displacement of parameters}
    \Delta z_l(t)\in\bigO(\Layer^{-1/2}).
\end{equation}
Consequently, the trained displacement \(\Delta z_l(t)\) vanishes as
\(\Layer\to\infty\), while its accumulated first-order effect across the
network remains of order one. The same scaling also implies that the
quadratic remainder in \eqref{eq: taylor expansion of resnet update} is of
order \(\bigO(\Layer^{-1/2})\), and therefore vanishes asymptotically. The
first-order Taylor expansion thus suggests a locally linearized description
of the training dynamics.

This mechanism is closely related to what
\citet{chizatHiddenWidthDeep2026} calls the {\it lazy-ODE regime}, for which he
rigorously proves the vanishing of the parameter displacements and a locally linearized
limiting dynamics. As already observed, at the critical stochastic scaling
\(\lambda_\Layer=\Layer^{-1/2}\), the random fluctuations at initialization
do not vanish, so that the limiting dynamics is different from the lazy-ODE
limit; nevertheless, the same scaling argument suggests vanishing parameter
displacements also at this boundary. However, while the displacement
of each layer's parameters vanishes, the accumulated first-order effect across
depth may still induce an order-one change of the hidden representations.
Thus, unlike in the usual Neural Tangent Kernel regime, the features need not
remain asymptotically frozen during training. We therefore use the term
{\it locally linearized}, rather than {\it lazy}, for this behavior.

We observe that the empirical evidence as to whether ResNets with independent initialization
benefit more from the scaling \(\Layer^{-1/2}\) or from \(\Layer^{-1}\) does
not yet appear fully conclusive. Earlier studies observed slightly better performance
for the \(\Layer^{-1/2}\) scaling
\citep[Table~1]{shaoNormalizationIndispensableTraining2020}, whereas the study
motivating Chizat's work suggests that the \(\Layer^{-1}\) scaling may be more
beneficial \citep{deyDonBeLazy2025}.

\FloatBarrier

\subsubsection{A conjectural phase diagram}

We now put ourselves in the framework of
Theorem~\ref{thm: large depth limit of resnets with correlated weights, init}.
Fix \(\alpha\in(0,1/\hermRank)\) and Hermite rank \(\hermRank\), and set
\[
    \hurst=1-\frac{\alpha\hermRank}{2}\in(1/2,1).
\]
Thus, \(\hurst\) is the critical exponent associated with the correlation
structure by our theorem. To distinguish this critical exponent from the
scaling actually used in the network, we write
\[
    \lambda_\Layer=\Layer^{-\gamma}.
\]

Motivated by the phase diagram of
\citet[Fig.~4]{chizatHiddenWidthDeep2026} and by the discussion in the
previous subsection, we conjecture that our results fit into the broader phase
diagram represented in Figure~\ref{fig: phase diagram}. We emphasize that
the behavior of the ResNet after training and the Blow-up regime
are conjectures based on \citeauthor{chizatHiddenWidthDeep2026}'s analysis of the independent
initialization setting and Remark \ref{rem: super-critical scaling}.

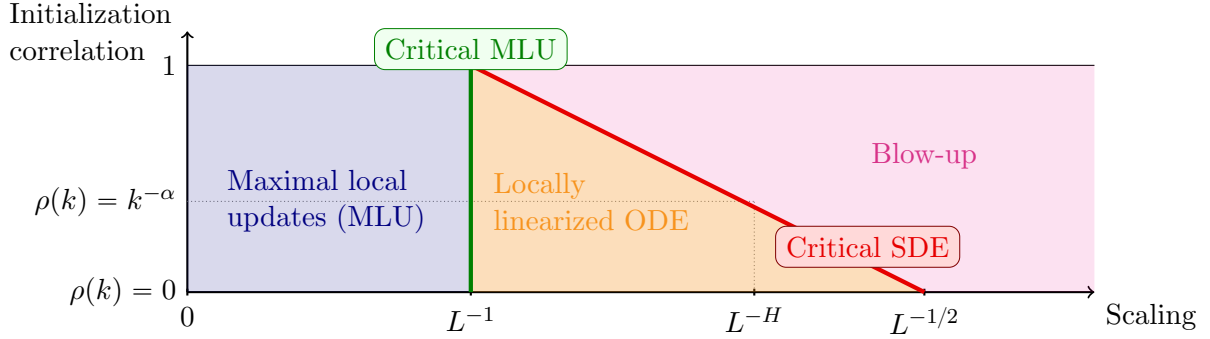
\begin{figure}[t]
    \centering
    \begin{tikzpicture}[
    scale=3,
    axis/.style={->, thick},
    guide/.style={densely dotted, gray},
    ode/.style={green!50!black, ultra thick},
    sde/.style={red!90!black, ultra thick},
    lazyode/.style={yellow!50!red, very thick},
    subcritical/.style={blue!50!black, very thick},
    every node/.style={font=\small}
]

\def\xL{1.25}
\def\xK{2.5}
\def\xSqrtL{3.25}
\def\yk{0.4}
\def\xmax{4}

\draw[axis] (0,0) -- (\xmax,0) node[below right] {Scaling};
\draw[axis] (0,0) -- (0,1.15) node[left, text width=2.2cm] {Initialization correlation};

\draw[guide] (0,\yk) -- (\xK,\yk);
\draw[] (0,1) -- (\xmax,1);
\draw[guide] (\xK,0) -- (\xK,\yk);

\foreach \x/\label in {
    0/{0},
    \xL/{L^{-1}},
    \xK/{L^{-\hurst}},
    \xSqrtL/{L^{-1/2}}
}{
    \draw[thick] (\x,0.015) -- (\x,-0.015)
        node[below] {$\label$};
}

\node[left] at (0,0) {$\rho(k)=0$};
\node[left] at (0,\yk) {$\rho(k)=k^{-\alpha}$};
\node[left] at (0,1) {$1$};

\fill[magenta, opacity=0.12]
    (\xL,1) -- (\xmax,1) -- (\xmax,0) -- (\xSqrtL,0) -- cycle;

\node[magenta!90!black] at (\xSqrtL,0.6) {Blow-up};

\fill[lazyode, opacity=0.3]
    (\xL,1) -- (\xSqrtL,0) -- (\xL, 0) -- cycle;

\fill[subcritical, opacity=0.15]
    (0,1) -- (\xL,1) -- (\xL,0) -- (0,0) -- cycle;

\draw[ode] (\xL,0) -- (\xL,1);
\draw[sde] (\xL,1) -- (\xSqrtL,0);

\node[ode, above, fill=green!6!white, rounded corners, draw=green!50!black, very thin] at (\xL,0.98) {Critical MLU};
\node[sde, fill=red!15!white, rounded corners,draw=red!50!black,very thin] at (\xSqrtL-0.25,0.20) {Critical SDE};
\node[lazyode, text width=2.7cm] at (\xK-0.7,\yk) {Locally\\linearized ODE};
\node[subcritical, text width=2.7cm] at (\xL/2,\yk) {Maximal local updates (MLU)};

\end{tikzpicture}     \caption{\small Conjectural phase diagram for the ResNet model in the framework of
    Theorem~\ref{thm: large depth limit of resnets with correlated weights, init},
    with \(\hurst = 1-\frac{\alpha\hermRank}{2}\) and
    \(\alpha\in(0,1/\hermRank)\). The behavior during training and the Blow-up
    regime are conjectures.}
    \label{fig: phase diagram}
\end{figure}

\begin{itemize}[noitemsep]
    \item \textbf{Blow up.\footnote{Conjecture, see Rem.~\ref{rem: super-critical scaling}}}
    If \(\gamma<\hurst\), the scaling is larger than the critical scale
    associated with the prescribed correlation structure, and we expect the
    ResNet to blow up at initialization subject to suitable
    non-degeneracy assumptions about \(\diffusion\) (Rem.~\ref{rem: super-critical scaling}). This agrees with the independent
    case, where the corresponding threshold is \(1/2\).

    \item \textbf{Critical SDE (Non-trivial initialization).\footnote{The behavior after training is conjecture based on Chizat's analysis for independent initialization.\label{fn: training is conjecture}}}
    If
    \[
        \gamma=\hurst=1-\frac{\alpha\hermRank}{2},
    \]
    then
    Theorem~\ref{thm: large depth limit of resnets with correlated weights, init}
    proves that the ResNet converges at initialization to the non-trivial
    limiting differential equation driven by the corresponding Hermite
    process. Based on Chizat's analysis of the training we conjecture that the
    parameter changes should be of order \(\Layer^{\hurst-1}\), and hence
    vanish, while their accumulated first-order effect remains of order one. We
    therefore expect the training dynamics to be locally linearized. This would
    extend the ``Lazy SDE'' regime of \citet[Fig.~4]{chizatHiddenWidthDeep2026}
    to the correlated setting.

    \item \textbf{Locally linearized ODE.}\textsuperscript{\ref{fn: training is conjecture}}
    If \(\hurst<\gamma<1\), the residual scaling is smaller than the critical
    initialization scale, and we prove that the ResNet converges to the identity
    at initialization (Corollary~\ref{cor: trivial limit for sub-critical scaling}). At the same time, extrapolating Chizat's argument
    suggests parameter changes of order \(\Layer^{\gamma-1}\), which vanish
    as the depth diverges. To see this, combine the scaling \(\lambda_\Layer=\Layer^{-\gamma}\)
    with the Taylor expansion \eqref{eq: taylor expansion of resnet update}.
    We therefore conjecture a locally linearized ODE
    regime, analogous to the ``Lazy ODE'' regime of
    \citet[Fig.~4, Thm.~2]{chizatHiddenWidthDeep2026}.

    \item \textbf{Critical maximal local updates.}\textsuperscript{\ref{fn: training is conjecture}}
    At \(\gamma=1\), the preceding scaling argument predicts parameter changes
    of order one. This corresponds to the ``Maximum local update'' regime of
    \citet[Thm.~1]{chizatHiddenWidthDeep2026}. Here, maximal local updates
    means that the local features generated by an individual residual block
    may change by order one during training while the overall dynamics remains
    stable. We conjecture that an analogous regime persists under our
    correlated initialization. The ResNet is generally scaled to be the
    identity function at initialization.

    \item \textbf{Maximal local updates (Subcritical ODE).}\textsuperscript{\ref{fn: training is conjecture}}
    For \(\gamma>1\), the residual scaling is even smaller. In Chizat's
    setting, the corresponding subcritical MLU regime asymptotically
    coincides with the behavior of a network initialized with zero output
    weights \citep[Remark~4.2]{chizatHiddenWidthDeep2026}. By analogy, we
    conjecture a similar subcritical ODE behavior in our setting.

\end{itemize}

In summary, at initialization Theorem~\ref{thm: large depth limit of resnets with correlated weights, init}
and Corollary~\ref{cor: trivial limit for sub-critical scaling} rigorously characterize
the critical SDE curve 
\[
    \gamma=\hurst=1-\frac{\alpha\hermRank}{2}
\]
and its subcritical side in Figure~\ref{fig: phase diagram}.
The blow up-region, and the subdivision of the sub-critical region
into the Locally linearized ODE, Critical MLU and MLU regimes
represent a conjectural extension of Chizat's training phase diagram to
correlated initialization. The resulting picture suggests, in particular, that
criticality at initialization need not coincide with maximal local updates
during training.

\subsection{Experiments}
\label{sec: experiments}

To connect our main results to a training setting, we
modify the experiments of
\citet[Fig.~9]{marionScalingResNetsLargedepth2025} and reproduce their
experiment on trained ResNets with correlated initialization.\footnote{
    Our fork of their code is available at \url{https://github.com/FelixBenning/scaling-resnets}.
}
We keep the same architecture of a ResNet of width \(\dims=30\) and depth \(\Layer=1000\)
with ReLU activation function and without inner weights \(w_\layer\) trained on the MNIST dataset.  While they
initialized the inner weights \(v_\layer\) directly with increments of a
fractional Brownian motion, we use the Cauchy correlation function
\[
    \rho(k)=(1+k)^{-\alpha},
\]
which has the regularly varying behavior considered in our theoretical
framework. Besides the identity feature function
\(\feature(x)=x\), of Hermite rank \(1\), we also use the second Hermite
polynomial \(\feature(x)=x^2-1\), of Hermite rank \(2\). The results are shown
in Figure~\ref{fig: training experiments}.

\begin{figure}
    \centering
    \includegraphics*[width=0.45\textwidth]{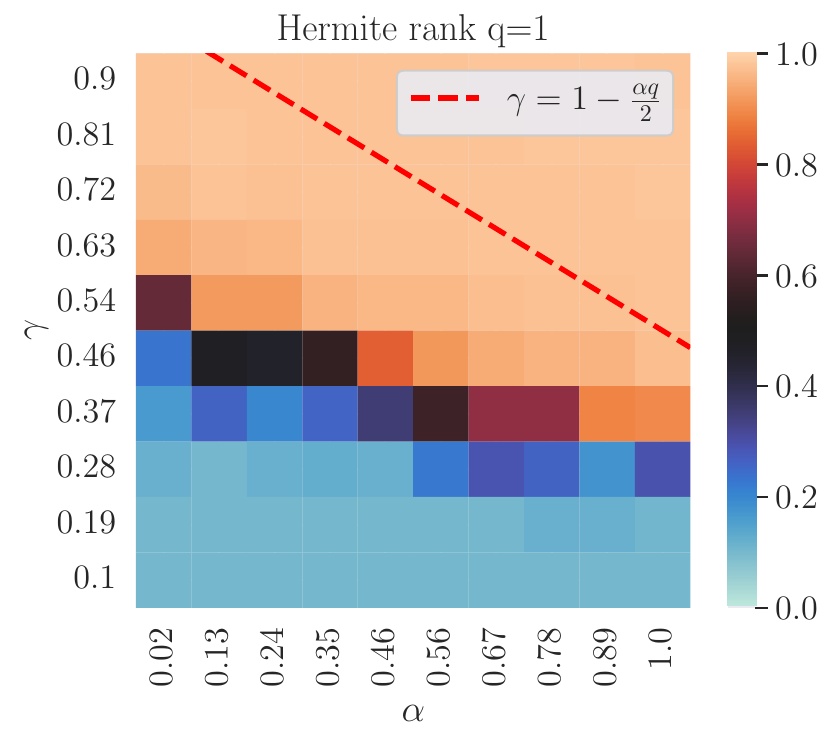}
    \includegraphics*[width=0.45\textwidth]{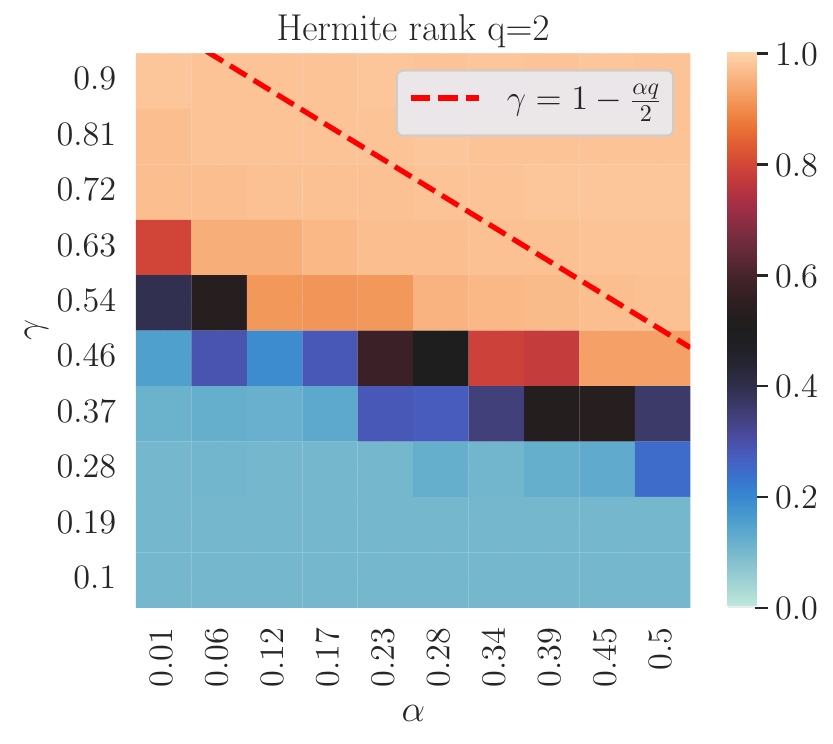}
    \caption{
    \label{fig: training experiments}
    \small The plots show the accuracy after ten epochs of
    training of a ResNet on the MNIST dataset. The parameter \(\gamma\) is the
    exponent of the scaling
    \(\lambda_\Layer=\Layer^{-\gamma}\), while \(\alpha\) is the index of the
    regularly varying correlation of the initialization process. The dashed
    red line corresponds to the critical exponent
    \(\gamma=1-\frac{\alpha\hermRank}{2}\) identified by
    Theorem~\ref{thm: large depth limit of resnets with correlated weights, init}.
    The left plot uses the identity \(x\mapsto x\) as feature function, with
    Hermite rank \(\hermRank=1\), whereas the right plot uses the second Hermite
    polynomial \(x^2-1\), with Hermite rank \(\hermRank=2\).
    }
\end{figure}

\smallskip 

For both Hermite rank one and Hermite rank two, the low-accuracy region in blue
lies predominantly below the critical curve
\[
    \gamma=1-\frac{\alpha\hermRank}{2}
\]
identified by
Theorem~\ref{thm: large depth limit of resnets with correlated weights, init}.
This region below the curve represents the conjectured blow-up region.
Larger exponents resulting in sub-critical scaling often still yield good performance. 
Curiously, the slope of the transition between poor and successful training 
neither matches the slope of the critical curve for non-trivial initialization,
nor is it horizontal. A horizontal border may be expected if maximal
local updates fully determined training behavior. In that case the outcome
should improve as the scaling approaches \(\Layer^{-1}\).
Our experiments therefore cannot establish whether the critical initialization
scaling or the maximal-local-update scaling \(\lambda_\Layer=\Layer^{-1}\) is
the more favorable choice for training.
\FloatBarrier

\section{Young integral equation solution theory}
\label{sec: young integral equation solution theory}

The goal of this section is to develop a general solution theory for
differential equations of the form
\[
    dx_t = \sigma(t,w_t, x_t) \, dg_t
    \quad\text{with initial condition}\quad x_{\ul{t}} = a \in \banachSpace[X],
\]
where \(g_t\) and \(w_t\) are \(\beta\) and \(\alpha\)-Hölder continuous functions with
exponent \(\beta\in (\frac12,1]\) and \(\alpha\in (\frac12, \beta)\). This means that classical ODE theory does
not apply as the driving signal \(g\) is not necessarily differentiable or of
bounded variation. However it is still smooth enough for Young integration
theory to be applicable. For \(\beta \le \frac12\) it would become necessary to
use rough path theory instead of Young integration \citep[see e.g.][]{frizCourseRoughPaths2020,frizEulerEstimatesRough2008}. Differential
equations of the form
\[
    dx_t = \sigma(x_t)\, dg_t
\]
with Lipschitz continuous \(\sigma\) are already well understood both
in the Young regime as well as in the rough path regime
\citep[e.g.][Chapter~8]{frizCourseRoughPaths2020}. Our
contribution is to extend this theory to non-homogeneous \(\sigma\) and prove
stability results with respect to \(g\), \(w\) and initial conditions \(a\).
An application of this theory is the infinite depth limit
of residual neural networks as described in Section \ref{sec: resnet limit}.

To distinguish between functional norms and norms on vectors, we use
\(\abs{\cdot}\) as notation for the norm of vectors and reserve \(\norm{\cdot}\)
for functional norms. Of course a vector in a general Banach space
may be a function. Since we will however not use its properties as a function
this distinction still helps make the concepts clearer.
Moreover we typically write \(f_s\) for function
evaluation at \(s\) to reserve
\(f(x)\) for functions that map a function \(x\) to another function.

\begin{definition}
    Let \((\banachSpace, \abs{\cdot})\) be a Banach space, and let $\alpha\in (0, 1]$.
    For a function \(f\colon [\ul{t}, \ol{t}] \to \banachSpace\)
    with \(\ul{t}, \ol{t} \in \real\) and  \(\ul{s}, \ol{s} \in [\ul{t},
    \ol{t}]\) we define the Hölder seminorm
    \[
        \holder{f}_{\alpha, [\ul{s},\ol{s}]} \coloneq \sup_{s\neq t \in [\ul{s}, \ol{s}]} \frac{\abs{f_t - f_s}}{\abs{t-s}^\alpha}
        \qquad\text{and}\qquad
        \holder{f}_{\alpha} \coloneq \holder{f}_{\alpha, [\ul{t}, \ol{t}]}.
    \]
    We further define the Hölder norm
    \[
        \norm{f}_{\alpha, [\ul{s}, \ol{s}]} \coloneq \norm{f}_{\infty, [\ul{s}, \ol{s}]} + \holder{f}_{\alpha, [\ul{s}, \ol{s}]}
        \qquad\text{and}\qquad
        \norm{f}_\alpha \coloneq \norm{f}_{\alpha, [\ul{t}, \ol{t}]}
    \]
    where \(\norm{f}_{\infty, [\ul{s}, \ol{s}]} = \sup_{s \in [\ul{s}, \ol{s}]} \abs{f_s}\) is the supremum norm of \(f\) on \([\ul{s}, \ol{s}]\).
\end{definition}

\begin{definition}[Functional spaces]\label{d:functionalspaces}
    Let \(\banachSpace[W]\) be a Banach space, let
    \(T = [\ul{t}, \ol{t}]\subseteq [0, \infty)\) be a compact interval, and let
    \(\beta\in(0,1]\). We denote by
    \[
        C^\beta(T,\banachSpace[W])
        \coloneq
        \left\{
            w\colon T\to\banachSpace[W]
            \,:\,
            \norm{w}_{\beta}<\infty
        \right\}
    \]
    the space of \(\beta\)-Hölder continuous functions from \(T\) to
    \(\banachSpace[W]\), equipped with the Hölder norm
    \(\norm{\cdot}_{\beta}\) defined above.
\end{definition}

For the differential equation \(dx_t = \sigma(t,w_t, x_t) \, dg_t\)
to have a unique solution, we require the following assumption, that is a
generalization of Assumptions \ref{assmpt: regularity of the activation function}
and \ref{assmpt: additional regularity} to the non-homogeneous case.

\begin{assumption}[Sufficiently nice function]
    \label{assmpt: sufficiently nice function}
    Let \(\banachSpace[W]\), \(\banachSpace[V]\) and \(\banachSpace[X]\) be Banach spaces
    and let \(\alpha\in (0,1]\). Then, the mapping
    \[
        \sigma\colon \real \times \banachSpace[W] \times \banachSpace[X]
        \to \linOp{\banachSpace[V]}{\banachSpace[X]}
    \]
    is an {\it \(\alpha\)-nice function} if the following properties are verified:
    \begin{enumerate}[label=(\alph*)]
        \item\label{it: (nice function) bounded}
        \(\sigma\) is \textbf{bounded}, that is \(\norm{\sigma}_\infty < \infty\).

        \item\label{it: (nice function) locally Lipschitz}
        \textbf{Local Lipschitz continuity:}
        \(\sigma(t, w, x)\) and the Fréchet derivative
        \[
            \frechet_x \sigma(t, w, x)
            \in
            \linOp{\banachSpace[X]}
            {\linOp{\banachSpace[V]}{\banachSpace[X]}}
        \]
        are locally Lipschitz continuous in \(w\) and \(x\),
        with local Lipschitz coefficients controlled by a continuous function
        \(c\colon \real^3 \to [0,\infty)\):
        \[
            \begin{aligned}
            \abs{\sigma(t, w, x) - \sigma(t, \tilde w, \tilde x)}
            &\le c(t, \abs{w}, \abs{\tilde w})
            \Bigl(
                \abs{x-\tilde x}
                + (1+\abs{x}+\abs{\tilde x})\abs{w-\tilde w}
            \Bigr)
            \\
            \abs{\frechet_x \sigma(t, w, x)
            - \frechet_x \sigma(t, \tilde w, \tilde x)}
            &\le
            \underbrace{
                c(t, \abs{w}, \abs{\tilde w})
            }_{\text{locally bounded}}
            \underbrace{
                \Bigl(
                    \abs{x-\tilde x}
                    + (1+\abs{x}+\abs{\tilde x})\abs{w-\tilde w}
                \Bigr)
            }_{\text{local Lipschitz control}} .
            \end{aligned}
        \]

        \item\label{it: (nice function) locally Hölder continuous}
        \textbf{Hölder continuity in \(t\):}
        \(\sigma(\cdot, w, x)\) and \(\frechet_x \sigma(\cdot, w, x)\)
        are locally \(\alpha\)-Hölder continuous in \(t\),
        with local Hölder coefficients controlled by a continuous function
        \(\mathsf c\colon\real \to [0,\infty)\):
        \[
            \begin{aligned}
            \abs{\sigma(t, w, x) - \sigma(s, w, x)}
            &\le
            \mathsf{c}(\abs{w})(1+\abs{x})\abs{t-s}^{\alpha}
            \\
            \abs{\frechet_x \sigma(t, w, x)
            - \frechet_x \sigma(s, w, x)}
            &\le
            \underbrace{
                \mathsf{c}(\abs{w})(1+\abs{x})
            }_{\text{locally bounded}}
            \abs{t-s}^{\alpha}.
            \end{aligned}
        \]
    \end{enumerate}

    Finally, an \textbf{optional} assumption on the Fréchet derivative of \(\sigma\)
    with respect to \(w\) instead of \(x\) is given by
    \begin{enumerate}[label=(\alph*), resume]
        \item\label{it: extra assumption}
        \textbf{Differentiability in \(w\):}
        there exists a continuous function
        \(\tilde c \colon \real^5 \to [0,\infty)\) and
        \(\tilde{\mathsf c} \colon \real^2 \to [0,\infty)\) such that
        the Fréchet derivative
        \[
            \frechet_w \sigma(t, w, x)
            \in
            \linOp{\banachSpace[W]}
            {\linOp{\banachSpace[V]}{\banachSpace[X]}}
        \]
        is locally Lipschitz continuous in the sense
        \begin{align}
            \abs{\frechet_w \sigma(t, w, x)
            - \frechet_w \sigma(t, \tilde w, \tilde x)}
            &\le
            \tilde c(t, \abs{w}, \abs{\tilde w}, \abs{x}, \abs{\tilde x})
            \Bigl(
                \abs{x-\tilde x} + \abs{w-\tilde w}
            \Bigr)
            \\
            \abs{\frechet_w \sigma(t, w, x)
            - \frechet_w \sigma(s, w, x)}
            &\le
            \tilde{\mathsf c}(\abs{w}, \abs{x}) \abs{t-s}^\alpha .
        \end{align}
    \end{enumerate}
\end{assumption}

\begin{remark}[Merging \(t\) into \(w\)]
    While we assume local Lipschitz continuity in \(w\), we only assume 
    local Hölder continuity in \(t\). For this reason, it is not trivial to
    merge \(t\) into \(w\). If one were to embed \(t\) into
    the function space \(L^p([0,T], \real)\) with \(p=\frac1\alpha\) via
    \(t\mapsto \ind_{[0,t]}\), then Lipschitz continuity translates to 
    \(\alpha\)-Hölder continuity in \(t\) as \(\norm{\ind_{[0,t]}-\ind_{[0,s]}}_{L^p} = \abs{t-s}^\alpha\), such that in place of \((t,w)\) one
    may consider the parameter \(\tilde w = (\ind_{[0,t]}, w)\). This may allow for
    an alternative proof where \(t\) is merged into \(w\), but the translation is a bit
    awkward and we choose the more direct approach in the following.
\end{remark}

Theorem~\ref{thm: differential equation solution existence and uniqueness} below
is one of the main theoretical contributions of the paper.

\begin{remark}
For the reader's convenience, we formally clarify the meaning of the differential equation
\eqref{e:sdeb} below; see e.g. \cite{frizEulerEstimatesRough2008,frizCourseRoughPaths2020} for further details. A path
\(
    x\in C^\alpha([\ul{t},\ol{t}],\banachSpace[X])
\)
is a solution of \eqref{e:sdeb} if
\[
    x_t
    =
    a+\int_{\ul{t}}^t \sigma(s,w_s,x_s)\,dg_s,
    \qquad t\in[\ul{t},\ol{t}],
\]
where the integral is understood in the Young sense.
More precisely,
\[
    \int_{\ul{t}}^t \sigma(s,w_s,x_s)\,dg_s
    \coloneq
    \lim_{\abs{\pi}\to0}
    \sum_{[u,v]\in\pi}
    \sigma(u,w_u,x_u)(g_v-g_u),
\]
where \(\pi\) ranges over partitions of \([\ul{t},t]\) and
\(\abs{\pi}\) denotes the mesh of the partition.
Under the assumptions
of Theorem~\ref{thm: differential equation solution existence and uniqueness},
the map
\(
    s\mapsto \sigma(s,w_s,x_s)
\)
is \(\alpha\)-Hölder continuous, while \(g\) is \(\beta\)-Hölder
continuous, and \(\alpha+\beta>1\); hence the Young integral above is
well defined. Whenever \eqref{e:sdeb} admits a unique solution for every initial
condition, one defines the \emph{solution flow} 
\[
    \flow(\,\cdot\,; s,t)\colon \banachSpace[X]\to\banachSpace[X],
    \qquad 0\le s\le t\le T,
\]
by \(\flow(a; s,t) \coloneq x_t^{s,a}\),
where \(x^{s,a}\) denotes the unique solution of \eqref{e:sdeb} starting
from \(a\in\banachSpace[X]\) at time \(s\), that is,
\[
    x_r^{s,a}
    =
    a+\int_s^r \sigma(u,w_u,x_u^{s,a})\,dg_u,
    \qquad r\in[s,T].
\]
By uniqueness the flow satisfies
\[
    \flow(a; s,s) = a,
    \qquad
    \flow(\,\cdot\,; s,t)
    = \flow(\,\cdot\,; u,t) \circ \flow(\,\cdot\,; s,u),
    \qquad 0\le s\le u\le t\le T.
\]
\end{remark}

\begin{mdframed}[innertopmargin=0pt]
\begin{theorem}[Differential equation solution]
    \label{thm: differential equation solution existence and uniqueness}
    For \(\beta\in (\frac12, 1]\), let \(\alpha \in (\frac12, \beta)\) and assume \(g\in C^\beta([0,T],
    \banachSpace[V])\), \(w \in C^\alpha([0,T], \banachSpace[W])\) and let
    \(\sigma\colon \real \times \banachSpace[W] \times \banachSpace[X] \to \linOp{\banachSpace[V]}{\banachSpace[X]}\)
    be an \(\alpha\)-nice function (Assumption~\ref{assmpt: sufficiently nice function}).
    Then, for any \(0\le \ul{t}<\ol{t}\le T\) the differential equation 
    \begin{equation}\label{e:sdeb}
        dx_t = \sigma(t,w_t, x_t) \, dg_t
        \quad\text{with initial condition}\quad x_{\ul{t}} = a \in \banachSpace[X] 
    \end{equation}
    \begin{enumerate}[label=(\roman*), font=\upshape]
        \item\label{it: existence and uniqueness} has a \ul{unique} solution
        \(x\) with \(x \in C^\alpha([\ul{t},\ol{t}], \banachSpace[X])\).
    \end{enumerate}
For any \(R>0\)
    there exist constants
    \[
        C_{\text{flow}}^R, C_{\text{flow,loc}}^R, C_{\text{init}}^R, C_{\text{driver}}^R, C_{\text{param}}^R
        > 0
    \]
    such that for all initial conditions
    \(a,b \in B(0,R)\), all driving signals
    \(g, \tilde g \in C^\beta([0,T], \banachSpace[V])\)
    with \(\holder{g}_\beta, \holder{\tilde g}_\beta \le R\),
    all \(w, \tilde w\in C^\alpha([0,T], \banachSpace[W])\) with
    \(\norm{w}_\alpha, \norm{\tilde w}_\alpha \le R\) and all \(0\le \ul{t}< \ol{t} \le T\) we have
    \begin{enumerate}[label=(\roman*), font=\upshape, resume]
        \item\label{it: flow bound}
        \textbf{Local flow bound:}
        Let \(\flow\) be the flow starting at \(a\in \banachSpace[X]\) in time \(\ul{t}\) (the solution path), then we have
        \begin{align}
            \label{eq: flow bound}            
            \norm{\flow(a; \ul{t}, \cdot)}_{\alpha} &\le C_{\text{flow}}^R
            \\
            \holder{\flow(a; \ul{t}, \cdot)}_{\alpha, [\ul{t}, \ol{t}]}
            &\le C_{\text{flow,loc}}^R \abs{\ol{t} - \ul{t}}^{\beta - \alpha}
            &\forall a \in B(0,R), 0 \le \ul{t}< \ol{t} \le T.
        \end{align}

        \item\label{it: lipschitz in initial condition}
        \textbf{Local Lipschitz continuity in the initial condition:} Let \(x\) be the
        solution to the differential equation \eqref{e:sdeb} with initial condition \(a\), and
        \(y\) be the solution to \eqref{e:sdeb} with initial
        condition \(b\) and the same driving signal \(g\) with \(\holder{g}_\beta < R\), then
        \begin{equation}
            \label{eq: init cond lipschitz bound}
            \norm{x-y}_{\alpha, [\ul{t}, \ol{t}]} \le C_{\text{init}}^R \abs{a-b}
            \qquad \forall a,b \in B(0,R), 0 \le \ul{t}< \ol{t} \le T.
        \end{equation}

        \item\label{it: lipschitz in driving signal}
        \textbf{Local Lipschitz continuity in the driving signal:} Let \(x\) be the solution
        to the differential equation \eqref{e:sdeb} with driving signal \(g\), and \(y\) be the
        solution to \eqref{e:sdeb} with driving signal \(\tilde g \in
        C^\beta([0,T], \banachSpace[V])\)
        and the same initial condition \(a\in B(0,R)\), then
        \[
            \norm{x-y}_{\alpha, [\ul{t}, \ol{t}]} \le C_{\text{driver}}^{R} \holder{g-\tilde g}_\beta
            \qquad \forall g, \tilde g \in C^\beta([0,T], \banachSpace[V]): \holder{g}_\beta, \holder{\tilde g}_\beta \le R.
        \]
    \end{enumerate}
    And with the optional Assumption \ref{assmpt: sufficiently nice function} \ref{it: extra assumption}
    \begin{enumerate}[resume, label=(\roman*)]
        \item\label{it: lipschitz in parameters}
        \textbf{Local Lipschitz continuity in the parameters:}
        Let \(x\) be the solution to the differential equation \eqref{e:sdeb} with parameters
        \(w\), and \(y\) be the solution to \eqref{e:sdeb} with
        parameters \(\tilde w\) and the same initial condition \(a\in B(0,R)\) and the
        same driving signal \(g\) with \(\holder{g}_\beta < R\), then
        \[
            \norm{x-y}_{\alpha, [\ul{t}, \ol{t}]} \le C_{\text{param}}^R \norm{w-\tilde w}_\alpha
            \qquad \forall w, \tilde w \in C^\alpha([0,T], \banachSpace[W]): \norm{w}_\alpha, \norm{\tilde w}_\alpha \le R.
        \]
    \end{enumerate}
\end{theorem}
\end{mdframed}

A direct consequence of the previous statement is that the solution of the differential equation
is locally Lipschitz continuous in all input arguments.

\begin{mdframed}[innertopmargin=0pt]
\begin{corollary}[Solution is locally Lipschitz in inputs]
    \label{cor: solution is locally Lipschitz in inputs}
    Assume that Assumption \ref{assmpt: sufficiently nice function},
    including the optional condition \ref{it: extra assumption}, holds.
    For \(\beta\in (\frac12, 1]\) and \(\alpha \in (\frac12, \beta)\) let
    \[
        \Psi_\alpha \colon 
        \begin{cases}
            C^\beta([0,T], \banachSpace[V]) \times C^\alpha([0,T], \banachSpace[W]) \times \banachSpace[X]
            &\to C^\alpha([0,T], \banachSpace[X])
            \\
            (g,w,a) &\mapsto  \Psi_\alpha(g, w, a)
        \end{cases}
    \]
    be the map that maps the driving signal \(g\), the parameters \(w\) and the initial condition \(a\) to the unique solution \(\Psi_\alpha(g, w, a)\) of the differential equation \eqref{e:sdeb}.
    Then \(\Psi_\alpha\) is locally Lipschitz continuous, where
    the space \(C^\beta([0,T], \banachSpace[V]) \times C^\alpha([0,T], \banachSpace[W]) \times \banachSpace[X]\)
    is equipped with the norm
    \[
        \norm{(g,w,a)} = \norm{g}_\beta + \norm{w}_\alpha + \abs{a}.
    \]
    Without the optional assumption \ref{it: extra assumption} local Lipschitz continuity in \(g\) and \(a\) remains.
\end{corollary}
\end{mdframed}
\begin{proof}[Proof of Corollary \ref{cor: solution is locally Lipschitz in inputs}]
    Let \(\theta = (g, w, a)\) and define \(R \coloneq \max\set{2\norm{\theta},1}\).
    Then for all \(\tilde \theta = (\tilde g, \tilde w, \tilde a)\) with \(\norm{\tilde\theta - \theta} < \frac{R}2\)
    we have
    \[
        \holder{\tilde g}_\beta \le \norm{\tilde g}_\beta
        \le \norm{\tilde g - g}_\beta + \norm{g}_\beta
        \le \tfrac{R}2 + \norm{g}_\beta
        \le R
    \]
    and similarly, \(\norm{\tilde w}_\alpha \le R\) and \(\abs{\tilde a} \le R\). Thus, we can apply Theorem \ref{thm: differential equation solution existence and uniqueness} to obtain
    \begin{align}
        \norm{\Psi_\alpha(\tilde \theta) - \Psi_\alpha(\theta)}_\alpha
        \overset{\Delta}&\le \norm{\Psi_\alpha(\tilde g, \tilde w, \tilde a) - \Psi_\alpha(\tilde g, \tilde w, a)}_\alpha
        + \ldots
        + \norm{\Psi_\alpha(\tilde g, w, a) - \Psi_\alpha(g, w, a)}_\alpha
        \\
        \overset{\text{Thm.~\ref{thm: differential equation solution existence and uniqueness}}}&\le C_{\text{init}}^R \abs{a-\tilde a} + C_{\text{driver}}^R \holder{g-\tilde g}_\beta + C_{\text{param}}^R \norm{w-\tilde w}_\alpha
        \\
        &\le \underbrace{(C_{\text{init}}^R + C_{\text{driver}}^R + C_{\text{param}}^R)}_{\text{loc. Lipschitz constant}} \norm{\tilde\theta - \theta}
    \end{align}
    The arguments for local Lipschitz continuity in \(g\) and \(a\) without the optional condition \ref{it: extra assumption} in Assumption \ref{assmpt: sufficiently nice function} are analogous.
\end{proof}

Having established existence, uniqueness, and stability of the solution, we
now turn to its approximation by finite discretizations. The following result
shows that the piecewise-linear interpolation of the Euler scheme converges to
the solution in Hölder topology, and in particular in the supremum norm. This
approximation result is a key ingredient in the proof of the large-depth
convergence of ResNets in Section~\ref{sec: resnet limit}.

\begin{mdframed}[innertopmargin=0pt]
\begin{theorem}[Euler method convergence]
    \label{thm: convergence of euler scheme}
    For \(\beta \in (\frac12, 1]\) let \(\alpha \in (\frac12, \beta)\),
    \(g\in C^\beta([0,T], \banachSpace[V])\), \(w\in C^\alpha([0,T],
    \banachSpace[W])\) and assume \(\sigma(t,w, x)\) is an \(\alpha\)-nice
    function (see Assumption \ref{assmpt: sufficiently nice
    function}). Let \(x\) be the unique solution of the differential equation
    \[
        dx_t = \sigma(t, w_t, x_t) \, dg_t
        \qquad \text{with initial condition}\quad x_0 = a.
    \]
    Let \(\pi = \set{t_0, \dots, t_n}\) be a discretization of \([0,T]\) with
    \(0=t_0 < \dots < t_n = T\), yielding the Euler discretization of \(x\) given by
    \[
        x^\pi_{k+1} = x^\pi_k + \sigma(t_k, w_{t_k}, x^\pi_k) (g_{t_{k+1}} - g_{t_k})
        \qquad \text{with initial condition}\quad x^\pi_0 = a.
    \]
    Define \(\abs{\pi}\coloneq \max_{k} \abs{t_{k+1} - t_k}\).
    Then for any \(\alpha' \in (0, \alpha)\) we have
    \begin{equation}
        \lim_{\abs{\pi} \to 0} \norm{\bar x^\pi - x}_{\alpha'} = 0
        \qquad\text{with}\qquad
        \bar x^\pi_t \coloneq x^\pi_k + \underbrace{\tfrac{t-t_k}{t_{k+1}-t_k}(x^\pi_{k+1} - x^\pi_k)}_{\text{linear interpolation}} \text{ for } t\in [t_k, t_{k+1})
    \end{equation}
    and \(\bar x^\pi_T \coloneq x^\pi_n\).
    Specifically, for any \(\alpha \in (\frac12, \beta)\) as above, and any \(R>0\), there exist \(C_{\mathrm{Euler}}^{R, \alpha}
    > 0\) and \(\tau = \tau(R, \alpha) > 0\) such that for \emph{all}
    \(\alpha'\in (0, \alpha)\), \emph{all} initial conditions \(a\) with
    \(\abs{a}\le R\), \emph{all} drivers \(g\) with
    \(\holder{g}_\beta \le R\) and \emph{all} parameters \(w\) with \(\norm{w}_\alpha
    \le R\) we have for \emph{all} discretizations \(\pi\) with maximal gap \(\abs{\pi} \le \tau\)
    \[
        \norm{\bar x^\pi - x}_{\alpha'} 
        \le C_{\mathrm{Euler}}^{R, \alpha} \abs{\pi}^{(1-\frac{\alpha'}{\alpha})(\alpha+\beta - 1)}.
    \]

\end{theorem}
\end{mdframed}
\begin{remark}[Sup-norm]
    Observe that the case \(\alpha'=0\) may be viewed as corresponding to the
    supremum norm, for which the estimate becomes
    \[
        \norm{\bar x^\pi-x}_{\infty}
        \le
        \liminf_{\alpha'\to 0}
        \norm{\bar x^\pi-x}_{\alpha'}
        \le
        \limsup_{\alpha'\to 0}
        \norm{\bar x^\pi-x}_{\alpha'}
        \le
        C_{\mathrm{Euler}}^{R,\alpha}
        \abs{\pi}^{\alpha+\beta-1}.
    \]
    In fact, the proof proceeds by first establishing this sup-norm estimate and
    then using it to deduce convergence in Hölder topology.
\end{remark}
\begin{remark}[Discrete convergence]
    If one is not interested in the interpolation \(\bar{x}\) of \(x\), one can use the fact that, 
    at the discretization points, the interpolation coincides with \(x_k\), in such a way that
    \[
        \sup_{k} \abs{x_k^\pi - x_{t_k}}
        = \sup_{k} \abs{\bar{x}_{t_k}^\pi - x_{t_k}}
        \le \norm{\bar{x}^\pi - x}_\infty.
    \]
    Analogously, a discrete Hölder bound holds.
\end{remark}

\section{Proofs}

\subsection{Proof of Theorem \ref{thm: large depth limit of resnets with correlated weights, init}}

Choose a target exponent \(\beta \in (\tfrac12, \min\set{\hurst-\frac1p, \beta_w})\) and the 
auxiliary exponents \(\gamma_1, \gamma_2\) such that
\begin{equation}
    \label{eq: auxiliary exponents}
    \tfrac12 < \beta < \gamma_2 < \min\set{\gamma_1, \beta_w}, \qquad \text{with} \qquad \gamma_1 < \hurst - \tfrac1p.
\end{equation}
Note that we may choose \(\beta > \frac12\) without loss of generality, even though we only assume \(\beta>0\)
in the theorem statement, due to the embedding of
Hölder spaces and \(\min\set{\hurst-\frac1p, \beta_w} > \frac12\) by assumption.

\begin{steps}
    \item\label{step: driver convergence}
    \textbf{\(\iv^\Layer \overset{d}\to \cv\) in \(\gamma_1\)-Hölder space:}
    This follows directly from
    the functional limit theorems for correlated random walks
    \citep{benningFunctionalScalingLimits2026}. More specifically, Theorem 2.4
    from \citet{benningFunctionalScalingLimits2026} implies component-wise
    convergence of \(\iv^\Layer\) in Hölder space
    \[
        (\iv^\Layer)^i \overset{d}\to (\cv)^i.
    \]
    Since these processes are independent in \(i\) we obtain convergence
    of the entire processes \(\iv^\Layer\) in the product Hölder space \citep[Thm.~2.8]{billingsleyConvergenceProbabilityMeasures1999}.
    Note that by independence of \(\iv^\Layer\) and \(\iw^\Layer\) we also have
    joint convergence \((\iv^\Layer, \iw^\Layer) \overset{d}\to (\cv, \cw)\) in
    \(C^{\gamma_1} \times C^{\gamma_2}\) using \(\gamma_2 < \beta_w\).

    \item \textbf{Convergence of Wong-Zakai approximation:}
    For \(w^\Layer \in \set{\iw^\Layer, \cw}\), depending on the initialization
    assumption on \(w_\layer\), we have, by the continuous mapping Theorem \citep[e.g.][Thm.~13.25]{klenkeProbabilityTheoryComprehensive2014},
    \begin{equation}
        \label{eq: convergence wong zakai}    
        \wZh^\Layer \coloneq \Psi_{\beta}(\iv^\Layer, w^\Layer, Ax) \overset{d}\to \Psi_{\beta}(\cv, \cw, Ax) = \hCont,
        \quad \text{in} \quad C^{\beta}([0,1], \real^\dims),
    \end{equation}
    where \(\Psi_{\beta}(z, w, a)\) is the continuous solution map of the differential equation
    \[
        dh_t = \diffusion(w_t,h_t) dz_t \quad \text{with} \quad h_0 = a.
    \]
    The continuity of the solution map follows directly from Corollary \ref{cor: solution is locally Lipschitz
    in inputs} using
    Assumption \ref{assmpt: regularity of the activation function} to get
    continuity in \((z,a)\) and the additional regularity Assumption
    \ref{assmpt: additional regularity} for continuity in \((z,w,a)\).

    \item \textbf{Convergence of Euler discretization:}
    Since \(\hDiscr^\Layer\) is the Euler discretization of the differential equation
    \begin{equation}
        \label{eq: wong zakai approximation}    
        d\wZh^\Layer_s = \diffusion(w^\Layer_s, \wZh^\Layer_s) d\iv^\Layer_s \quad \text{with} \quad \wZh^\Layer_0 = Ax
    \end{equation}
    and \(\hInterp^\Layer\) is the piecewise linear interpolation of the Euler discretization,
    the convergence proof of the Euler method (Theorem \ref{thm: convergence of euler scheme}) will
    allow us to finish the proof. Observe that for this result to be applicable we need
    \(w^\Layer\) and \(\iv^\Layer\) to be bounded by some \(R > \max\set{\abs{Ax},0}\). As a consequence, we condition on this
    event to get for all bounded, Lipschitz continuous functions \(f\colon (C^\beta([0,1], \real^\dims), \norm{\cdot}_{\beta}) \to \real\)
    \begin{align}
        \E\Bigl[
            \abs[\big]{f(\hInterp^\Layer) - f(\wZh^\Layer)}\ind_{\norm{\iv^\Layer}_{\gamma_1} \le R, \norm{w^\Layer}_{\gamma_2} \le R}
        \Bigr]
        &\le  \mathrm{Lip}(f)\E[
            \norm{\hInterp^\Layer - \wZh^\Layer}_\beta
            \ind_{\norm{\iv^\Layer}_{\gamma_1} \le R, \norm{w^\Layer}_{\gamma_2} \le R}
        ]
        \\
        \label{eq: euler scheme approximation}
        \overset{\text{Thm.~\ref{thm: convergence of euler scheme}}}
        &\le  \mathrm{Lip}(f)C_{\text{Euler}}^{R, \gamma_2} \abs{\tfrac1\Layer}^{\bigl(1-\frac{\beta}{\gamma_2}\bigr)(\gamma_2 + \gamma_1 -1)}
        \to 0 \quad (\Layer \to \infty),
    \end{align}
    since \(\frac1\Layer\) is the size of the discretization intervals, and the exponent is
    positive by the choice of auxiliary exponents in \eqref{eq: auxiliary exponents}.

    \item \textbf{Conclusion.}
    For all bounded, Lipschitz continuous functions \(f \colon
    (C^\beta([0,1], \real^\dims), \norm{\cdot}_{\beta}) \to \real\)
    we have for all \(R > \max\set{\abs{Ax},0} \)
    \begin{align}
        \abs[\big]{\E[f(\hCont)] - \E[f(\hInterp^\Layer)]}
        &\le
        \begin{aligned}[t]
        \underbrace{\abs[\big]{\E[f(\hCont)] - \E[f(\wZh^\Layer)]}}_{\to 0 \quad \text{contin.\ solution }\mathrlap{\text{map \eqref{eq: convergence wong zakai}}}}
        &+ \underbrace{\E\bigl[\abs[\big]{f(\wZh^\Layer) - f(\hInterp^\Layer)}\ind_{\norm{\iv^\Layer}_{\gamma_1} \le R, \norm{w^\Layer}_{\gamma_2} \le R}
        \bigr]}_{\to 0 \quad \text{Euler method conv. \eqref{eq: euler scheme approximation}}}
        \\
        &+ 2\norm{f}_\infty \prob[\big]{\set{\norm{\iv^\Layer}_{\gamma_1} > R}\cup \set{\norm{w^\Layer}_{\gamma_2} > R}}.
        \end{aligned}
    \end{align}
    Consequently, we have
    \begin{align}
        \label{eq: bound on limit superior}    
        \limsup_{\Layer \to \infty} \abs[\big]{\E[f(\hCont)] - \E[f(\hInterp^\Layer)]}
        &\le \limsup_{\Layer \to \infty}2\norm{f}_\infty \bigl(\prob[\big]{\norm{\iv^\Layer}_{\gamma_1} > R} + \prob[\big]{\norm{w^\Layer}_{\gamma_2} > R}\bigr)
        \\
        &\to 0 \qquad (R \to \infty).
    \end{align}
    The convergence follows from tightness. Indeed, since \(\iv^\Layer
    \overset{d}\to \cv\) in \(\gamma_1\)-Hölder space (\ref{step: driver convergence}), we have that (by the continuous mapping Theorem)
    \(\norm{\iv^\Layer}_{\gamma_1} \overset{d}\to \norm{\cv}_{\gamma_1}\). This sequence is
    consequently tight in \(\real\) by Prokhorov's theorem
    \citep[e.g.][Thm.~13.29]{klenkeProbabilityTheoryComprehensive2014},
    which implies that
    \[
        \lim_{R \to \infty} \sup_{\Layer \in \nat} \prob{\norm{\iv^\Layer}_{\gamma_1} > R} = 0.
    \]
    The proof for \(\norm{w^\Layer}_{\gamma_2}\) is analogous and follows from the definition
    of \(w^\Layer\) due to \(\gamma_2 < \beta_w\).
\end{steps}

\subsection{Proof of Example \ref{ex: sufficiently nice function}}

The boundedness of \(\diffusion\) follows from the boundedness of \(\psi\).
Since \(\diffusion(w, x) = \psi(Wx + b)^\transpose \otimes \id_\dims\) we have
\[
   \abs{\sigma(w,x) - \sigma(\tilde w, \tilde x)} 
   \precsim \abs[\big]{\psi(Wx + b) - \psi(\tilde W \tilde x + \tilde b)},
\]
where \(\precsim\) means that the left-hand side is upper bounded by a constant
multiple of the right-hand side.  This constant multiple depends on the choice
of matrix norm, but always exists since all norms in finite dimension are
equivalent. Let \(W_i\) be the \(i\)-th row vector of \(W\), then,
since any vector norm is equivalent to the \(1\)-norm up to a constant, we have
\[
    \abs[\big]{\psi(Wx + b) - \psi(\tilde W \tilde x + \tilde b)}
    \precsim \sum_{i=1}^m\abs[\big]{\psi(\scp{W_i, x} + b_i) - \psi(\scp{\tilde W_i,\tilde x} + \tilde b_i)}.
\]
For every \(i\) we have
\begin{align}
    \abs{\psi(\scp{W_i, x} + b_i) - \psi(\scp{\tilde W_i,\tilde x} + \tilde b_i)}
    &\le \begin{aligned}[t]
    &\abs{\psi(\scp{W_i, x} + b_i) - \psi(\scp{W_i,\tilde x} + b_i)}
    \\
    &+ \abs{\psi(\scp{W_i, \tilde x} + b_i) - \psi(\scp{\tilde W_i,\tilde x} + b_i)}
    \\
    &+ \abs{\psi(\scp{\tilde W_i, \tilde x} + b_i) - \psi(\scp{\tilde W_i,\tilde x} + \tilde b_i)}.
    \end{aligned}
    \\
    \overset{\text{Cauchy-Schwarz}}&\precsim \lip(\psi)(\abs{W_i}\abs{x-\tilde x} + \abs{\tilde x}\abs{W_i - \tilde W_i} + \abs{b_i - \tilde b_i})
    \\
    &\le c^0(\abs{W_i})\Bigl(\abs{x-\tilde x} + (1+\abs{x}+\abs{\tilde x})\underbrace{\bigl(\abs{W_i - \tilde W_i} + \abs{b_i - \tilde b_i}\bigr)}_{\precsim \abs{w-\tilde w}}\Bigr)
\end{align}
with \(c^0(\abs{W_i}) \coloneq \lip(\psi)(1+\abs{W_i})\). Collecting
the absolute constants from the conversion of norms we have shown that
there exists a continuous function \(c^1(\abs{w})\) in \(\abs{w}\) such that
\[
    \abs{\sigma(w,x) - \sigma(\tilde w, \tilde x)} 
    \le c^1(\abs{w})\bigl(\abs{x-\tilde x} + (1+\abs{x}+\abs{\tilde x})\abs{w-\tilde w}\bigr).
\]
For the Fréchet derivative we proceed similarly:
\begin{align}
    \abs{\frechet_x \sigma(w,x) - \frechet_x \sigma(\tilde w, \tilde x)}
    &\precsim \abs[\big]{\psi'(Wx + b)W - \psi'(\tilde W \tilde x + \tilde b)\tilde W}
    \\
    &\precsim \sum_{i=1}^m\abs[\big]{\psi'(\scp{W_i, x} + b_i)W_i - \psi'(\scp{\tilde W_i, \tilde x} + \tilde b_i)\tilde W_i}.
\end{align}
This additional \(W_i\) factor does not pose a problem, as we can reduce it to the previous case with
\(\psi'\) instead of \(\psi\) using the triangle inequality
\begin{align}
    &\abs[\big]{\psi'(\scp{W_i, x} + b_i)W_i - \psi'(\scp{\tilde W_i, \tilde x} + \tilde b_i)\tilde W_i}
    \\
    &\le 
    \underbrace{\abs[\big]{\psi'(\scp{W_i, x} + b_i) - \psi'(\scp{\tilde W_i, \tilde x} + \tilde b_i)}}_{
        \precsim \lip(\psi')(\abs{W_i}\abs{x-\tilde x} + \abs{\tilde x}\abs{W_i - \tilde W_i} + \abs{b_i - \tilde b_i})
    }\abs{W_i}
    + \underbrace{\abs[\big]{\psi'(\scp{\tilde W_i, \tilde x} + \tilde b_i)}}_{\le \lip(\psi)}\abs{W_i - \tilde W_i}.
\end{align}
This yields for some constant \(c^2(\abs{w})\) that depends continuously on \(\abs{w}\) that
\[
    \abs{\frechet_x \sigma(w,x) - \frechet_x \sigma(\tilde w, \tilde x)}
    \le c^2(\abs{w})\bigl(\abs{x-\tilde x} + (1+\abs{x}+\abs{\tilde x})\abs{w-\tilde w}\bigr).
\]
Putting everything together, we infer that Assumption \ref{assmpt: regularity of the activation function} is
satisfied, with the function \(c(\abs{w}) \coloneq \max\set{c^1(\abs{w}),
c^2(\abs{w})}\). The proof of the additional regularity Assumption \ref{assmpt:
additional regularity} is analogous and left to the reader.

\subsection{Proof of Theorem \ref{thm: differential equation solution existence and uniqueness}}

For existence and uniqueness of the solution we will use the Banach fixed point theorem.
Specifically, we will construct time intervals
\([t_k, t_{k+1}]\) on which we show that the operator \(F_k\) with
\[
    F_k(x)_t = a_k + \int_{t_k}^t f(x)_s \, dg_s
    \qquad \text{where} \qquad
    f(x)_s \coloneq \sigma(s, w_s, x_s)
\]
has a unique fixed point. Key ingredients for this are the continuity
properties of the Young integral, which are summarized in Lemma \ref{lem: continuity of young integrals}
and the continuity properties of \(f\) summarized in Lemma \ref{lem: f Lipschitz with respect to alpha norm}.
Using these results we prove that \(F_k\) maps a suitable ball to itself (Lemma \ref{lem: F_k maps the ball to itself})
and is a contraction on this ball (Lemma \ref{lem: F_k is a contraction}).
The last step to prove existence and uniqueness will then be to glue the
solutions on the small intervals together and prove Hölder continuity.
The key ingredient for this is Lemma \ref{lem: holder glue}.

\begin{remark}\label{r:boundedsigma}
Interestingly, the boundedness of \(\sigma\) is used only to show that \(F_k\)
maps a suitable ball into itself. More precisely, it enters the proof only
through the four estimates
\eqref{eq: bound on f_infty}, \eqref{e:estisup},
\eqref{e:estisup2}, and \eqref{eq: uniform bound on f(x) at s}.
By contrast, the contraction property of \(F_k\) does not rely on the
boundedness of \(\sigma\).
\end{remark}

The proofs of the stability properties are relatively short, reusing some
of the machinery developed for existence and uniqueness.

\subsubsection{Proof of \ref{it: existence and uniqueness}: Existence and uniqueness}

We prove that \(x_t\) is the unique solution on finitely many small intervals.
    Define
    \[
        \tau_k\coloneq \frac{\tau_0}{k+1},
        \qquad
        K \coloneq \min\set[\big]{n\in\nat:
        \sum_{k=0}^{n-1}\tau_k\ge \ol{t}-\ul{t}},
    \]
    where \(\tau_0>0\) is to be chosen later. The number \(K\) is finite since
    \(\sum_{k=0}^\infty\tau_k=\infty\). We set
    \[
        t_k\coloneq \ul{t}+\sum_{l=0}^{k-1}\tau_l
        \quad\text{for }0\le k<K,
        \qquad t_K\coloneq \ol{t}.
    \]
    Thus \(t_0=\ul{t}\), \(t_K=\ol{t}\), and
    \(t_{k+1}-t_k\le\tau_k\) for every \(0\le k<K\); only the final
    interval may be shorter than \(\tau_{K-1}\). We set \(a_0=a\) and, once
    the fixed point on \([t_k,t_{k+1}]\) has been constructed, use its end
    point as the next initial condition, \(a_{k+1}\coloneq x_{t_{k+1}}\).
    Using the Banach fixed point
    theorem we will construct \(\alpha\)-Hölder continuous solutions on each
    interval. These are then glued together to obtain a solution on \([\ul{t},\ol{t}]\)
    and we show uniqueness and Hölder continuity of the solution on the entire
    interval \([\ul{t},\ol{t}]\).

\begin{steps}
    \item \textbf{The Banach fixed point theorem on small intervals.}
    \label{step: banach fixed point}
    On each interval \([t_k, t_{k+1}]\) we want to show that the operator \(F_k\) with
    \[
        F_k(x)_t = a_k + \int_{t_k}^t f(x)_s \, dg_s
        \qquad \text{where} \qquad
        f(x)_s = \sigma(s, w_s, x_s)
    \]
    has a unique fixed point.
    Define \(\norm{\cdot}_{\infty, k} \coloneq \norm{\cdot}_{\infty, [t_k, t_{k+1}]}\)
    and \(\holder{\cdot}_{\alpha, k} \coloneq \holder{\cdot}_{\alpha, [t_k, t_{k+1}]}\),
and the helper function
    \[
        h\colon \begin{cases}
            \nat \to \real
            \\
            n \mapsto n^{1-\alpha}\log(n)^\alpha.
        \end{cases}
    \]
    Then
    \[
        n^{-\alpha}(1+h(n)) = \Bigl(n^{-\alpha}+\frac{\log(n)^\alpha}{n^{2\alpha -1}}\Bigr) \to 0
        \implies K_h \coloneq \sup_{n\in \nat}n^{-\alpha}(1+h(n)) < \infty.
    \]
    With the constant above and the constant \(C_{\alpha, \beta}\) from Lemma \ref{lem: continuity of young integrals} we may now choose the initial interval size \(\tau_0\)
    \[
        \tau_0 = \min\set[\Big]{
            1, \bigl(\frac{\eta}{(\norm{\sigma}_\infty + C_{\alpha, \beta} \ref{const: f_alpha bound} K_h T^\alpha)R}\bigr)^{\frac1{\beta-\alpha}}, (2 K_h \ref{const: contraction})^{-\frac1\beta}
        },
    \]
    where \(\eta\) and \ref{const: f_alpha bound} are defined in Lemma \ref{lem: F_k maps the ball to itself} below and \ref{const: contraction} is defined in Lemma \ref{lem: F_k is a contraction} below.
    This choice of \(\tau_0\) ensures that the Lemmas \ref{lem: F_k maps the
    ball to itself} and \ref{lem: F_k is a contraction} are in force.

    \begin{remark}[Constants]
        \label{rem: constants}
        The constants are chosen so that they do not depend on the initial point \(a\),
        the driving signal \(g\) or the parameters \(w\) as long as they are bounded by \(R\). 
        Finally, \(\eta\) is only a variable for the uniqueness argument in \ref{step: gluing solutions together}.
        After it is established that the solution is unique we can choose
        \(\eta=1\) without loss of generality to obtain constants independent of \(\eta\).
    \end{remark}

    \begin{mdframed}[innertopmargin=0pt]
    \begin{lemma}[\(F_k\) maps the ball to itself]
        \label{lem: F_k maps the ball to itself}
        For all \(\eta > 0\) selected independently of \(k\)
        \[
            x\in B_k \coloneq B_k(\eta) \coloneq
            \set[\big]{x\colon [t_k,t_{k+1}]\to \banachSpace[X] \mid x_{t_k} = a_k,
            \holder{x}_{\alpha,k} \le \eta}
        \]
        we have for all starting points with \(\abs{a} \le R\) and all \(g\) with \(\holder{g}_\beta \le R\)
        \begin{align}
            \label{eq: bound on x_infty}
            \abs{a_{k+1}} \le \norm{x}_{\infty, k} &\le \ref{const: x sup norm bound}(1+h(k+1))
            & \defConst{x sup norm bound} &\coloneq R + \eta\bigl(1+\tfrac1{\log(2)}\bigr)^\alpha
            \\
            \label{eq: bound on f_alpha}
            \holder{f(x)}_{\alpha,k} &\le \ref{const: f_alpha bound}(1+h(k+1))
            & \defConst{f_alpha bound} &
            \coloneq \ref{const: f_bound}(1+\eta +\ref{const: x sup norm bound})
            \\
            \label{eq: bound on f_infty}
            {\norm{f(x)}_{\infty, k}} &{{}\le \norm{\sigma}_\infty}.
        \end{align}
        As a consequence, for any \(\tau_0 \le \min\set{1, \bigl(\frac{\eta}{(\norm{\sigma}_\infty + C_{\alpha, \beta} \ref{const: f_alpha bound} K_h T^\alpha)R}\bigr)^{\frac1{\beta-\alpha}}}\) we have
        \[
            \holder{F_k(x)}_{\alpha, k} \le \eta
            \qquad\text{and thereby}\qquad
            F_k(B_k) \subseteq B_k.
        \]
    \end{lemma}
    \end{mdframed}
    \begin{proof}
        For the first claim \eqref{eq: bound on x_infty}, observe that by \(\tau_0 \le 1\)
        \[
            \sum_{l=0}^{k-1} \tau_l^\alpha
            \le k\Bigl(\frac1k\sum_{l=0}^{k-1} \tau_l^\alpha\Bigr)
            \overset{\text{Jensen}}\le k\Bigl(\frac1k\sum_{l=0}^{k-1} \tau_l\Bigr)^\alpha
            \le k^{1-\alpha} \Bigl(\sum_{l=0}^{k-1} \frac{\tau_0}{l+1}\Bigr)^\alpha
            \overset{\tau_0 \le 1}\le k^{1-\alpha} \Bigl(1+\log(k)\Bigr)^\alpha.
        \]
        Consequently, for all \(k\ge 1\) and \(x \in B_{k-1}\) we have,
    \[\begin{aligned}
        \norm{x}_{\infty, k-1} = \sup_{t\in[t_{k-1},t_{k}]} \abs{x_t}
        &\le \abs{a_{k-1}} + \holder{x}_{\alpha,k-1} (t_{k} - t_{k-1})^\alpha
        \le \abs{a_{k-1}} + \eta\tau_{k-1}^\alpha
        \\
        \overset{\text{induction}}&\le \abs{a_0} + \eta\sum_{l=0}^{k-1} \tau_l^\alpha
        \le R + \eta k^{1-\alpha} \Bigl(1+\log(k)\Bigr)^\alpha
        \le \ref{const: x sup norm bound} (1+ h(k)).
    \end{aligned}
    \]
    The last inequality follows from \(R + \eta\le \ref{const: x sup norm bound}\) for \(k=1\) and for \(k\ge 2\) we use 
    \[
        k^{1-\alpha} \Bigl(1+\log(k)\Bigr)^\alpha
        \le \bigl(1+\tfrac1{\log(2)}\bigr)^\alpha
        \underbrace{k^{1-\alpha}\log(k)^\alpha}_{=h(k)}.
    \]
    For the second claim \eqref{eq: bound on f_alpha} we use local boundedness of \(f\) from Lemma \ref{lem: f Lipschitz with respect to alpha norm},
    specifically,
    \[
        \holder{f(x)}_{\alpha, k} \overset{\text{Lem.~\ref{lem: f Lipschitz with respect to alpha norm}}}\le \ref{const: f_bound}(1+\norm{x}_{\alpha, k})
        = \ref{const: f_bound}(1+ \underbrace{\holder{x}_{\alpha, k}}_{\le \eta} + \underbrace{\norm{x}_{\infty, k}}_{
            \le \ref{const: x sup norm bound}(1+ h(k))
        })
        \le \ref{const: f_alpha bound}(1+h(k))
    \]
    with \(\ref{const: f_alpha bound} = \ref{const: f_bound}(1+\eta +\ref{const: x sup norm bound})\).
    The third claim \eqref{eq: bound on f_infty} follows immediately from the boundedness of \(\sigma\): by the definition of \(f\)
    \[
        \abs{f(x)_s}
        = \abs{\sigma(s, w_s, x_s)}
        \le \norm{\sigma}_\infty.
    \]
    This bound finally implies that for any \(x\in B_k\) and \(s,t \in [t_k, t_{k+1}]\) using the continuity of Young integrals from Lemma \ref{lem: continuity of young integrals} 
    \[\begin{aligned}
        \frac{\abs{F_k(x)_t - F_k(x)_s}}{\abs{t-s}^\alpha}
        = \frac{\abs[\Big]{\int_s^t f(x)_u \, dg_u}}{\abs{t-s}^\alpha}
        \overset{\text{Lem.~\ref{lem: continuity of young integrals}}}&\le (\underbrace{\abs{f(x)_s}}_{\le {\norm{\sigma}_{\infty}}}
        + C_{\alpha, \beta} \underbrace{\holder{f(x)}_{\alpha, k} \abs{t-s}^\alpha}_{\le \ref{const: f_alpha bound}(1+h(k+1)) \tau_k^\alpha}
        ) \holder{g}_\beta \underbrace{\abs{t-s}^{\beta-\alpha}}_{\le \tau_k^{\beta-\alpha}}
        \\
        \label{e:estisup}\overset{\tau_k=\frac{\tau_0}{k+1}}&\le (\norm{\sigma}_\infty + C_{\alpha, \beta} \ref{const: f_alpha bound} K_h \tau_0^\alpha)\holder{g}_\beta 
        \bigl(\tfrac{\tau_0}{k+1}\bigr)^{\beta-\alpha}
        \\
        &\le \eta.
    \end{aligned}
    \]
    For the last equation we use \(\holder{g}_\beta \le R\) and the choice
    \(\tau_0 \le \min\set{1, \bigl(\frac{\eta}{(\norm{\sigma}_\infty + C_{\alpha, \beta} \ref{const: f_alpha bound} K_h T^\alpha)R}\bigr)^{\frac1{\beta-\alpha}}}\).
    This proves \(\holder{F_k(x)}_{\alpha,k} \le \eta\) and thus \(F_k(x) \in B_k\), that is: \(F_k\) maps \(B_k\) to itself.
    \end{proof}

    \begin{mdframed}[innertopmargin=0pt]
    \begin{lemma}[\(F_k\) is a contraction]
        \label{lem: F_k is a contraction}
        Assume \(\holder{g}_\beta \le R\) and let \(\tau_0 \le \min\set{1, (2
        K_h \ref{const: contraction})^{-\frac1\beta}}\) with
        \[
            \defConst{contraction}
            \coloneq 2\ref{const: f holder bound}(1 + \eta + \ref{const: x sup norm bound})(1+ C_{\alpha, \beta})R.
        \]
        Let \(x,y \in B_k(\eta)\) with \(B_k\) as in Lemma \ref{lem: F_k maps the ball to itself},
        then \(F_k\) is a contraction on \(B_k\), that is
        \[
            \holder{F_k(y) - F_k(x)}_{\alpha, k}
            \le \tfrac12 \holder{y-x}_{\alpha,k}.
        \]
    \end{lemma}
    \end{mdframed}

    \begin{proof}
    Let \(x,y \in B_k\). 
    Since \(x_{t_k} = y_{t_k}\) we also have \(f(x)_{t_k} = f(y)_{t_k}\) 
    and therefore for all \(s\in [t_k, t_{k+1}]\)
    \begin{equation}
        \label{eq: bound on f(y)-f(x) at s}
        \begin{aligned}
        \abs{f(y)_s - f(x)_s}
        = \abs{f(y)_s - f(x)_s - f(y)_{t_k} + f(x)_{t_k}}
        &\le \holder{f(y) - f(x)}_{\alpha, k} \abs{s-t_k}^\alpha
        \\
        &\le \holder{f(y) - f(x)}_{\alpha, k}\tau_k^\alpha.
    \end{aligned}
    \end{equation}
    Thus for any \(s,t \in [t_k, t_{k+1}]\) we have
    by the continuity of Young integrals (Lemma \ref{lem: continuity of young integrals})
    \begin{align}
        &\abs{F_k(y)_t - F_k(y)_s - F_k(x)_t + F_k(x)_s}
        = \abs[\Big]{\int_s^t f(y)_u - f(x)_u \, dg_u}
        \\
        \overset{\text{Lem.~\ref{lem: continuity of young integrals}}}&\le (\abs{f(y)_s - f(x)_s} + C_{\alpha, \beta}\holder{f(y) - f(x)}_{\alpha,k}\abs{t-s}^\alpha ) \holder{g}_\beta \abs{t-s}^{\beta}
        \\
        \overset{\eqref{eq: bound on f(y)-f(x) at s}}&\le (1 + C_{\alpha, \beta})\holder{f(y) - f(x)}_{\alpha,k} \holder{g}_\beta \tau_k^\alpha \abs{t-s}^{\beta}
    \end{align}
    Dividing both sides by \(\abs{t-s}^\alpha\) we observe that \(\abs{t-s}^{\beta-\alpha} \le \tau_k^{\beta-\alpha}\), where we use
    \(t_{k+1}-t_k\le \tau_k\), and therefore
    \[\begin{aligned}
        \holder{F_k(y)-F_k(x)}_{\alpha,k}
        &= \sup_{s\neq t \in[t_k, t_{k+1}]}\frac{\abs{F_k(y)_t - F_k(y)_s - F_k(x)_t + F_k(x)_s}}{\abs{t-s}^\alpha}
        \\
        &\le \holder{f(y) - f(x)}_{\alpha, k}(1 + C_{\alpha, \beta}) \holder{g}_\beta \tau_k^{\beta}.
    \end{aligned}
    \]
    Using \(x_{t_k} = y_{t_k}\) again in Lemma \ref{lem: f Lipschitz with respect to alpha norm}
    we get a bound on \(\holder{f(y) - f(x)}_{\alpha, k}\) of the form
    \[
        \holder{f(y) - f(x)}_{\alpha, k}
        \le \ref{const: f holder bound} (1+ \norm{x}_{\alpha,k} + \norm{y}_{\alpha,k}) \holder{y-x}_{\alpha,k}.
    \]
    Recall that by Lemma \ref{lem: F_k maps the ball to itself} we have for \(x \in B_k\)
    \[
        \norm{x}_{\alpha, k} = \holder{x}_{\alpha, k} + \norm{x}_{\infty, k} \le \eta + \ref{const: x sup norm bound}(1+h(k+1))
    \]
    and therefore the same for \(y\in B_k\).
    Using the constant
    \[
        \ref{const: contraction}
        = 2\ref{const: f holder bound}(1 + \eta + \ref{const: x sup norm bound})(1+ C_{\alpha, \beta})R
    \]
    and \(\holder{g}_\beta \le R\) we thus have
    \[
        \holder{F_k(y)-F_k(x)}_{\alpha,k}
       \le \holder{y-x}_{\alpha,k} \ref{const: contraction} (1+h(k+1)) \Bigl(\frac{\tau_0}{k+1}\Bigr)^{\beta}
       \le \tfrac12 \holder{y-x}_{\alpha,k}.
    \]
    In the last inequality, we used \((1+h(k+1))(k+1)^{-\beta} \le K_h\) (due to
    \(\alpha < \beta\)) and the choice of \(\tau_0 \le (2K_h\ref{const:
    contraction})^{-\frac1\beta}\).
    Consequently \(F_k\) is a contraction on \(B_k\).
    \end{proof}

    \item \textbf{Gluing the solutions together, uniqueness and Hölder continuity.}
    \label{step: gluing solutions together}
    Now we simply apply Lemma \ref{lem: F_k maps the ball to itself} and Lemma \ref{lem: F_k is a contraction}
    to obtain by the Banach fixed point theorem that \(F_k\) has a unique fixed point
    \(x\) in \(B_k = B_k(\eta)\) for every \(k\in\set{0,\dots,K-1}\). By concatenating the solutions on the intervals \([t_k,
    t_{k+1}]\)
    we obtain a solution on \([\ul{t}, \ol{t}]\). So far we only know that this
    solution is \(\alpha\)-Hölder continuous on each interval \([t_k, t_{k+1}]\) with constant \(\eta\).
    
    \begin{lemma}[Hölder glue]
        \label{lem: holder glue}
        Let \(0\le\ul{t}<\ol{t}\le T\) and
        \(\pi = \set{t_0, \dots, t_K}\) be a discretization of
        \([\ul{t},\ol{t}]\), so that \(t_0=\ul{t}\) and \(t_K=\ol{t}\).
        Let \(x\colon [\ul{t}, \ol{t}] \to \banachSpace[X]\)
        be a function such that for all \(k\in \set{0, \dots, K-1}\) we have
        \[
            \holder{x}_{\alpha, [t_k, t_{k+1}]} \le \eta
        \]
        for some \(\eta > 0\). Then \(x\) is \(\alpha\)-Hölder continuous
        on \([\ul{t}, \ol{t}]\), specifically
        \[
            \holder{x}_{\alpha, [\ul{t}, \ol{t}]}
            \le \eta K^{1-\alpha}.
        \]
    \end{lemma}
    \begin{proof}
    To get \(\alpha\)-Hölder continuity on \([\ul{t},\ol{t}]\), let \(t, s\in[\ul{t},\ol{t}]\) (without loss of
    generality \(t > s\)). Then there exist \(k,m \in \set{0,\dots,K-1}\) such that \(t \in
    [t_k, t_{k+1}]\) and \(s \in [t_m, t_{m+1}]\). The case \(k=m\) is trivial so we assume
    \(m<k\) without loss of generality. Then we have
    \[\begin{aligned}
        \abs{x_t - x_s}
        &\le \abs{x_t - x_{t_k}} + \sum_{l=m+1}^{k-1} \abs{x_{t_{l+1}} - x_{t_l}} + \abs{x_{t_{m+1}} - x_s}
        \\
        &\le \eta (t-t_k)^\alpha + \sum_{l=m+1}^{k-1} \eta (t_{l+1}-t_l)^\alpha + \eta (t_{m+1}-s)^\alpha
        & (\holder{x}_{\alpha, [t_l, t_{l+1}]} \le \eta)
        \\
        &\le \eta(k-m + 1)^{1-\alpha} (t-s)^\alpha
        & \hspace{-4em}\Bigl(n\sum_{i=1}^n \tfrac1ny_i^\alpha \overset{\text{concave}}\le n \Bigl(\frac1n\sum_{i=1}^n y_i\Bigr)^\alpha\Bigr)
        \\
        &\le \eta K^{1-\alpha}\abs{t-s}^\alpha
    \end{aligned}\]
    and consequently \(x\) is \(\alpha\)-Hölder continuous on \([\ul{t},\ol{t}]\) with constant \(\eta K^{1-\alpha}\).
    \end{proof}

    Since \(\ol{t}-\ul{t}\le T\), the number \(K\) of intervals in our
    construction is bounded by the constant
    \begin{equation}
        \label{eq: definition of K_T}
        K\le K_T\coloneq
        \min\set[\Big]{n\in\nat:\sum_{k=0}^{n-1}\tau_k\ge T}.
    \end{equation}
    The sequence \((\tau_k)_k\) is independent of \(\ul{t}\), and so is \(K_T\). Applying Lemma \ref{lem: holder glue} therefore gives the
    uniform bound
    \[
        \holder{x}_{\alpha,[\ul{t},\ol{t}]}
        \le \defConst{x alpha Hölder}
        \coloneq \eta(K_T+1)^{1-\alpha}.
    \]

    Uniqueness of the solution follows from the Banach fixed point theorem on
    the ball \(B_k\). For the general case, pick two \(\alpha\)-Hölder
    continuous solutions \(x\) and \(y\) of the differential equation
    and select \(\eta = \max\set{\holder{x}_\alpha, \holder{y}_\alpha}\). Then by induction
    over \(k\) we have \(x,y \in B_k(\eta)\) for all \(k\) and thus \(x=y\) as argued
    above. This finishes the proof of \ref{it: existence and uniqueness}.
\end{steps}

\subsubsection{Proof of \ref{it: flow bound}: Local flow bound}

    First observe that none of the constants depend on the exact initial condition (cf.~Remark~\ref{rem: constants}).
    We only used \(\abs{a} \le R\) in Lemma \ref{lem: F_k maps the ball to itself} to define
    the constant \(\ref{const: x sup norm bound}\) and thereby the following constants that use it.
    In particular, for any initial condition \(a\in B(0,R)\) we obtain
    uniform bounds on the solution \(x\)
    \begin{equation}
        \label{eq: uniform bounds on x}    
        \begin{alignedat}{2}
        \norm{x}_\alpha
        &= \norm{x}_\infty
        &&+ \holder{x}_\alpha
        \\
        &\le \ref{const: x sup norm bound}(1+h(K_T +1))
        &&+ \ref{const: x alpha Hölder}
        \eqcolon C_{\text{flow}}^R 
        \end{alignedat}
    \end{equation}
    where \ref{const: x alpha Hölder} and \(K_T\) are defined in \eqref{eq: definition of K_T} and
    are independent of \(\ul{t}\) by the discussion preceding Lemma \ref{lem: holder glue}. Since this constant
    is independent of the initial condition \(a\in B(0,R)\) and the initial time \(\ul{t}\)
    we moreover have these uniform bounds on the flow \(\flow(a; \ul{t}, \cdot)\)
    for all \(a\in B(0,R)\) and \(\ul{t} \in [0,T]\), that is
    \begin{equation}
        \label{eq: uniform bounds on flow}    
        \norm{\flow(a; \ul{t}, \cdot)}_\alpha
        \le C_{\text{flow}}^R
        \qquad\forall a\in B(0,R), \ul{t} \in [0,T].
    \end{equation}
    This proves the bound \eqref{eq: flow bound}. For the local bound
    observe that we have by continuity of the Young integral (Lemma \ref{lem: continuity of young integrals})
    and local boundedness of \(f\) (Lemma \ref{lem: f Lipschitz with respect to alpha norm})
    \begin{align}
        \abs{x_t - x_s}
        &= \abs[\Big]{\int_s^t f(x)_u \, dg_u}
        \\
        \overset{\text{Lem.~\ref{lem: continuity of young integrals}}}&\le (\underbrace{\abs{f(x)_s}}_{\le \red{\norm{\sigma}_\infty}}
        + C_{\alpha, \beta} \underbrace{
            \holder{f(x)}_\alpha
        }_{\le \ref{const: f_bound}\mathrlap{(1+ \norm{x}_{\alpha}) \quad (\text{Lem.~\ref{lem: f Lipschitz with respect to alpha norm}})}} \abs{t-s}^\alpha) \holder{g}_\beta \abs{t-s}^\beta \label{e:estisup2}
        \\
        &\le \underbrace{(\red{\norm{\sigma}_\infty} + C_{\alpha, \beta} \ref{const: f_bound}(1+ C_{\text{flow}}^R) T^\alpha) R}_{\eqcolon C_{\text{flow,loc}}^R} \abs{t-s}^\beta.
    \end{align}
    Since this is a uniform bound, we consequently have
    \[
        \holder{\flow(a; \ul{t}, \cdot)}_{\alpha, [\ul{t}, \ol{t}]}
        \le C_{\text{flow,loc}}^R (\ol{t}-\ul{t})^{\beta-\alpha}.
    \]

\subsubsection{Proof of \ref{it: lipschitz in initial condition}: Local Lipschitz continuity in the initial condition}

    With \(f(x)_s = \sigma(s, w_s, x_s)\) consider two solutions to the ODE \(x\) and \(y\)
    starting in \(a\) and \(b\) respectively, that is
    \[
        x_t = a + \int_{\ul{t}}^t f(x)_s \, dg_s \qquad\text{and}\qquad y_t = b + \int_{\ul{t}}^t f(y)_s \, dg_s.
    \]
    We will again prove Lipschitz continuity on small intervals \([t_k,
    t_{k+1}]\) and then glue the bounds together to obtain Lipschitz continuity
    on \([\ul{t},\ol{t}]\). However this time it is sufficient to choose a
    partition \(\ul{t}=t_0<\dots<t_K=\ol{t}\) whose interval lengths satisfy
    \(t_{k+1}-t_k\le\tau\). We may choose it such that
    \(K\le\ceil{T/\tau}+1\).
    Using the continuity of \(f\) (Lemma \ref{lem: f Lipschitz with respect to alpha norm}) and the
    uniform bounds on \(x\) and \(y\) from \eqref{eq: uniform bounds on x} we get
    that for all \(k\)
    \begin{equation}
        \label{eq: bound on f(x) - f(y) alpha norm}    
        \begin{aligned}
            \holder{f(x) - f(y)}_{\alpha,k}
            &\le \ref{const: f_bound}(1 + \norm{x}_{\alpha, k} + \norm{y}_{\alpha, k}) (\holder{x-y}_{\alpha,k} + \abs{x_{t_k}-y_{t_k}})
            \\
            &\le \underbrace{\ref{const: f_bound}(1+ 2C_{\text{flow}}^R)}_{\eqcolon \defConst{uniform f alpha bound}}(\holder{x-y}_{\alpha,k}+ \abs{x_{t_k}-y_{t_k}}).
        \end{aligned}
    \end{equation}
    We will use the constant \ref{const: x-y hölder bound} defined in \eqref{eq: definition of x-y holder bound} to select \(\tau \le (2\ref{const: x-y hölder bound})^{-\frac1{\beta-\alpha}}\).
    Similarly to the proof of the contraction property (Lemma \ref{lem: F_k is a contraction}) we deduce,
    using Lemma \ref{lem: continuity of young integrals},
    \begin{equation}
    \label{eq: bound on x-y difference, initial condition}    
    \begin{aligned}
        \abs{x_t - y_t - x_s + y_s}
        &= \abs[\Big]{\int_s^t f(x)_u - f(y)_u \, dg_u}
        \\
        &\le (\abs{f(x)_s - f(y)_s} + C_{\alpha, \beta} \holder{f(x) - f(y)}_{\alpha,k}\abs{t-s}^\alpha) \holder{g}_\beta \abs{t-s}^{\beta}.
    \end{aligned}
    \end{equation}
    Since we do not have the same starting location, the bound on the difference at \(s\)
    is less tight. However, one has the estimates
    \begin{equation}
        \label{eq: bound on f(x) - f(y) at s}    
        \begin{aligned}
        \abs{f(x)_s - f(y)_s}
        &\le \abs{f(x)_{t_k} - f(y)_{t_k}} + \holder{f(x) - f(y)}_{\alpha,k}\tau^\alpha
        \\
        &= \abs{\sigma(t_k, w_{t_k}, x_{t_k}) - \sigma(t_k, w_{t_k}, y_{t_k})} + \holder{f(x) - f(y)}_{\alpha,k}\tau^\alpha
        \\
        &\le \underbrace{\sup_{t\in[0,T]} c(t, \abs{w_t}, \abs{w_t})}_{\le K_c^R}\abs{x_{t_k}-y_{t_k}} + \holder{f(x) - f(y)}_{\alpha,k}T^\alpha.
        \\
        &\le \bigl(K_c^R + T^\alpha\ref{const: uniform f alpha bound}\bigr)
        \bigl(\holder{x-y}_{\alpha,k} + \abs{x_{t_k}-y_{t_k}}\bigr),
        \end{aligned}
    \end{equation}
    with \(K_c^R\) as defined in Lemma \ref{lem: f Lipschitz with respect to alpha norm}.
    Using \eqref{eq: bound on f(x) - f(y) at s}, \eqref{eq: bound on f(x) -
    f(y) alpha norm} and \(\holder{g}_\beta\le R\) in \eqref{eq: bound on x-y difference, initial condition},
    we finally get the bound
    \[
        \holder{x-y}_{\alpha, k}
        \le \ref{const: x-y hölder bound}(\holder{x-y}_{\alpha,k} + \abs{x_{t_k}-y_{t_k}})\tau^{\beta-\alpha}
    \]
    with
    \begin{equation}
        \label{eq: definition of x-y holder bound}    
        \defConst{x-y hölder bound}\coloneq (K_c^R + \ref{const: uniform f alpha bound}(T^\alpha  + C_{\alpha, \beta} T^\alpha))R.
    \end{equation}
    Due to the choice of \(\tau\le (2\ref{const: x-y hölder bound})^{-\frac1{\beta-\alpha}}\) we then get
    \[\begin{aligned}
        \holder{x-y}_{\alpha, k}
        &\le \abs{x_{t_k}-y_{t_k}}
        \le (\abs{x_{t_{k-1}}-y_{t_{k-1}}} + \holder{x-y}_{\alpha, k-1}\tau^\alpha)
        \\
        &\le (1+\tau^\alpha)\abs{x_{t_{k-1}}-y_{t_{k-1}}}
        \overset{\text{induction}}\le
        (1+ \tau^\alpha)^k \abs{a-b}
        \\
        &\le \underbrace{(1+ \tau^\alpha)^{\ceil{\frac{T}\tau}}}_{\eqcolon \defConst{local Lipschitz in init}} \abs{a-b}.
    \end{aligned}
    \]
    Recall that the number of intervals \(K\) is bounded by \(\ceil{T/\tau}+1\) and therefore
    we can glue the local bounds on the Hölder seminorm together using Lemma
    \ref{lem: holder glue} to obtain
    \[
        \holder{x-y}_\alpha \le 2\ref{const: local Lipschitz in init}(\ceil{\tfrac{T}\tau}+1) \abs{a-b}.
    \]
    Consequently, we have
    \[\begin{aligned}
        \norm{x-y}_\alpha
        &= \norm{x-y}_\infty + \holder{x-y}_\alpha
        \\
        &\le \abs{a-b} + \holder{x-y}_\alpha T^\alpha + \holder{x-y}_\alpha
        \le \underbrace{(1+ 2(1+T^\alpha)\ref{const: local Lipschitz in init}(\ceil{\tfrac{T}\tau}+1))}_{\eqcolon C_{\text{init}}^R} \abs{a-b}.
    \end{aligned}
    \]
    This is Lipschitz continuity in the initial condition with constant
    \(C_{\text{init}}^R\).

\subsubsection{Proof of \ref{it: lipschitz in driving signal}: Local Lipschitz continuity in the driving signal}

    Again, we have carefully chosen the constants to be independent of \(g\) and only
    depending on the uniform bound \(R\). We will similarly prove Lipschitz continuity
    on small intervals \([t_k, t_{k+1}]\) first and then glue the bounds
    together to obtain Lipschitz continuity. Let \(x\) and \(y\) be two
    solutions to the ODE with
    \[
        x_t = a + \int_{\ul{t}}^t f(x)_s \, dg_s
        \qquad\text{and}\qquad
        y_t = a + \int_{\ul{t}}^t f(y)_s \, d\tilde g_s.
    \]
    First observe that we already have obtained some bounds for \(x,y \in B(0,R)\) and \(\norm{g}_\beta \le R\), namely
    \begin{align}
        \label{eq: uniform bound on f(x) at s}
        \abs{f(x)_s} \le \norm{f(x)}_{\infty, k} &\le {\norm{\sigma}_\infty}
        \\
        \label{eq: uniform bound on f(x) alpha norm}
        \holder{f(x)}_{\alpha, k}
        \overset{\text{Lemma \ref{lem: f Lipschitz with respect to alpha norm}}}&\le \ref{const: f_bound}(1+ \norm{x}_{\alpha})
        \overset{\eqref{eq: uniform bounds on x}}\le \ref{const: f_bound}(1+ C_{\text{flow}}^R)
    \end{align}
    While \(x\) and \(y\) are defined differently, the same arguments as
    in the previous section yield
    \begin{align}
        \label{eq: bound on f(x) - f(y) at s (repeat)}
        \abs{f(x)_s - f(y)_s}
        \overset{\eqref{eq: bound on f(x) - f(y) at s}}&\le
        \bigl(K_c^R + T^\alpha\ref{const: uniform f alpha bound}\bigr)
        \bigl(\holder{x-y}_{\alpha,k} + \abs{x_{t_k}-y_{t_k}}\bigr)
        \\
        \label{eq: bound holder of f(x) - f(y)}
        \holder{f(x) - f(y)}_{\alpha, k}
        \overset{\eqref{eq: bound on f(x) - f(y) alpha norm}}&\le
        \ref{const: uniform f alpha bound}
        \bigl(\holder{x-y}_{\alpha, k} + \abs{x_{t_k} - y_{t_k}}\bigr).
    \end{align}
    Using that the Young integral is bilinear in \((f, g)\), we obtain for \(s,t \in [t_k, t_{k+1}]\)
    \[
        x_t-y_t - x_s + y_s =  \int_s^t f(x)_u - f(y)_u \, dg_u - \int_s^t f(y)_u \, d(\tilde g_u - g_u).
    \]
    We will now bound each term individually using Lemma \ref{lem: continuity of young integrals}.
    We have for \(s,t \in [t_k, t_{k+1}]\)
    \begin{align}
        \abs[\Big]{\int_s^t f(x)_u - f(y)_u \, dg_u}
        &\le \bigl(
            \abs{f(x)_s - f(y)_s}
            + C_{\alpha, \beta} \holder{f(x) - f(y)}_{\alpha,k}\abs{t-s}^\alpha
        \bigr)
        \holder{g}_\beta \abs{t-s}^{\beta}
        \\
        &\overset{\eqref{eq: bound on f(x) - f(y) at s (repeat)},\eqref{eq: bound holder of f(x) - f(y)}}\le 
        \underbrace{
            (K_c^R + T^\alpha\ref{const: uniform f alpha bound} + C_{\alpha, \beta}\ref{const: uniform f alpha bound} T^\alpha)R
        }_{=\ref{const: x-y hölder bound} \quad \eqref{eq: definition of x-y holder bound}}
        \bigl(
            \holder{x-y}_{\alpha, k} + \abs{x_{t_k} - y_{t_k}}
        \bigr)
        \abs{t-s}^{\beta}
    \end{align}
    The bound on the second integral is simply
    \begin{align}
        \abs[\Big]{\int_s^t f(y)_u \, d(\tilde g_u - g_u)}
        &\le \bigl(
            \abs{f(y)_s}
            + C_{\alpha, \beta} \holder{f(y)}_{\alpha,k}\abs{t-s}^\alpha
        \bigr)
        \holder{\tilde g - g}_\beta \abs{t-s}^{\beta}
        \\
        \overset{\eqref{eq: uniform bound on f(x) at s}, \eqref{eq: uniform bound on f(x) alpha norm}}&\le
        \underbrace{(\norm{\sigma}_\infty + C_{\alpha, \beta} \ref{const: f_bound}(1+ C_{\text{flow}}^R) T^\alpha)}_{\eqcolon \defConst{integral f(y) bound}}
        \holder{\tilde g - g}_\beta \abs{t-s}^{\beta}.
    \end{align}
    Putting everything together, we thus have
    \[
        \holder{x-y}_{\alpha, k}
        \le \max\set{\ref{const: x-y hölder bound}, \ref{const: integral f(y) bound}}(\abs{x_{t_k} - y_{t_k}} + \holder{x-y}_{\alpha, k} + \holder{\tilde g - g}_\beta) \tau^{\beta-\alpha}
    \]
    and, for \(\tau \le (2\max\set{\ref{const: x-y hölder bound}, \ref{const: integral f(y) bound}})^{-\frac1{\beta-\alpha}}\), we thus obtain
    \[\begin{aligned}
        \holder{x-y}_{\alpha,k}
        &\le \abs{x_{t_k} - y_{t_k}} + \holder{\tilde g - g}_\beta
        \\
        &\le \abs{x_{t_{k-1}} - y_{t_{k-1}}}  + \holder{x-y}_{\alpha, k-1}\tau^\alpha + \holder{\tilde g - g}_\beta
        \\
        &\le (1+\tau^\alpha)\Bigl(\abs{x_{t_{k-1}} - y_{t_{k-1}}} + \holder{\tilde g - g}_\beta\Bigr)
        \\
        &\le (1+\tau^\alpha)^k \Bigl(\underbrace{\abs{x_{t_0} - y_{t_0}}}_{=0} + \holder{\tilde g - g}_\beta\Bigr)
        \le \underbrace{(1+\tau^\alpha)^{\ceil{\frac{T}\tau}}}_{\eqcolon \defConst{local Hölder bound}} \holder{\tilde g - g}_\beta.
    \end{aligned}
    \]
    With the same arguments as before we can glue the local bounds on the Hölder seminorm
    together to obtain
    \[
        \holder{x-y}_{\alpha}
        \le 2(\ceil{\tfrac{T}\tau}+1)\ref{const: local Hölder bound}\holder{\tilde g - g}_\beta,
    \]
    and therefore
    \[
        \norm{x-y}_\alpha
        = \norm{x-y}_\infty + \holder{x-y}_\alpha
        \le \underbrace{2(1+T^\alpha)(\ceil{\tfrac{T}\tau}+1)\ref{const: local Hölder bound}}_{\eqcolon C_{\text{driver}}^R} \holder{\tilde g - g}_\beta,
    \]
    which is Lipschitz continuity in the driving signal.

    \subsubsection{Proof of \ref{it: lipschitz in parameters}: local Lipschitz continuity in the parameters}

    With \(f(w,x)_s = \sigma(s, w_s, x_s)\) consider two solutions to the ODE \(x\) and \(y\)
    with the same initial condition \(a\) and driving signal \(g\) but different
    parameters \(w\) and \(\tilde w\) respectively, that is
    \[
        x_t = a + \int_{\ul{t}}^t f(w,x)_s \, dg_s \qquad\text{and}\qquad y_t = a + \int_{\ul{t}}^t f(\tilde w,y)_s \, dg_s.
    \]
    The proof is now very similar to that of \ref{it: lipschitz in initial condition} and \ref{it: lipschitz in driving signal}.
    Using \(\tilde R \coloneq C_{\text{flow}}^R\) as a uniform bound on the solutions \(x\) and \(y\), we have
    by Lemma \ref{lem: f Lipschitz with respect to alpha norm}:
    \begin{align}
        \holder{f(w,x) - f(\tilde w, y)}_{\alpha, k}
        &\le \holder{f(w,x) - f(w,y)}_{\alpha, k} + \holder{f(w,y) - f(\tilde w, y)}_{\alpha, k}
        \\
        &\le \ref{const: f holder bound}(1+2\tilde R)\bigl(\holder{x-y}_{\alpha, k} + \abs{x_{t_k}-y_{t_k}}\bigr)
        + \ref{const: f w diff holder bound} \norm{w-\tilde w}_{\alpha}.
    \end{align}
    And we have
    \begin{align}
        \abs{f(w,x)_s - f(\tilde w, y)_s}
        &= \abs{\sigma(s, w_s, x_s) - \sigma(s, \tilde w_s, y_s)}
        \\
        \overset{\text{Assmpt.~\ref{assmpt: sufficiently nice function}}}&\le K_c^R \bigl(\abs{x_s - y_s} + (1+2\tilde R)\abs{w_s - \tilde w_s}\bigr)
        \\
        &\le K_c^R \bigl(\holder{x-y}_{\alpha, k} T^\alpha + \abs{x_{t_k} - y_{t_k}} + (1+2\tilde R)\norm{w-\tilde w}_{\alpha}\bigr)
        \\
        &\le \underbrace{
            K_c^R\max\set{1, T^\alpha, (1+2\tilde R)}
        }_{\eqcolon \defConst{f(w,x) diff bound}}\bigl(
            \holder{x-y}_{\alpha, k} + \abs{x_{t_k} - y_{t_k}} + \norm{w-\tilde w}_{\alpha}
        \bigr)
    \end{align}
    Thus, we have by Lemma \ref{lem: continuity of young integrals}:
    \begin{align}
        \abs{x_t - y_t - x_s + y_s}
        &= \abs[\Big]{\int_s^t f(w,x)_u - f(\tilde w, y)_u \, dg_u}
        \\
        &\le \bigl(
            \abs{f(w,x)_s - f(\tilde w, y)_s} + C_{\alpha, \beta} \holder{f(w,x) - f(\tilde w, y)}_{\alpha, k}\abs{t-s}^\alpha
        \bigr) \holder{g}_\beta \abs{t-s}^{\beta}.
        \\
        &\le \ref{const: x-y w diff holder bound}
        \Bigl(\holder{x-y}_{\alpha, k} + \abs{x_{t_k} - y_{t_k}} + \norm{w-\tilde w}_{\alpha}\Bigr)\abs{t-s}^{\beta}.
    \end{align}
    with the constant
    \[
        \defConst{x-y w diff holder bound} \coloneq (\ref{const: f(w,x) diff bound} + C_{\alpha, \beta}\max\set{\ref{const: f holder bound}(1+2\tilde R), \ref{const: f w diff holder bound}}T^\alpha)R.
    \]
    With \(t_{k+1} - t_k \le \tau\), we thus get
    \[
        \holder{x-y}_{\alpha, k}
        \le \ref{const: x-y w diff holder bound}\Bigl(\holder{x-y}_{\alpha, k} + \abs{x_{t_k} - y_{t_k}} + \norm{w-\tilde w}_{\alpha}\Bigr)\tau^{\beta-\alpha}.
    \]
    We now finish with the usual arguments. We pick \(\tau \le (2\ref{const: x-y w diff holder bound})^{-\frac1{\beta-\alpha}}\) to get
    \begin{align}
        \label{eq: bound on x-y difference, parameters}
        \holder{x-y}_{\alpha, k}
        &\le \abs{x_{t_k} - y_{t_k}} + \norm{w-\tilde w}_{\alpha}
        \\
        &\le \abs{x_{t_{k-1}} - y_{t_{k-1}}} + \holder{x-y}_{\alpha, k-1}\tau^\alpha + \norm{w-\tilde w}_{\alpha}
        \\
        \overset{\eqref{eq: bound on x-y difference, parameters}}&\le (1+\tau^\alpha)\Bigl(\abs{x_{t_{k-1}} - y_{t_{k-1}}} + \norm{w-\tilde w}_{\alpha}\Bigr)
        \\
        \overset{\text{ind.}}&\le
        (1+\tau^\alpha)^k \Bigl(\underbrace{\abs{x_{t_0} - y_{t_0}}}_{=0} + \norm{w-\tilde w}_{\alpha}\Bigr)
        \le (1+\tau^\alpha)^{\ceil{\frac{T}\tau}} \norm{w-\tilde w}_{\alpha}.
    \end{align}
    Again, we recall that the number of intervals \(K\) is bounded by \(\ceil{T/\tau}+1\) and therefore
    Lemma \ref{lem: holder glue} allows us to glue the local bounds on the Hölder seminorm together and
    to get
\[
    \holder{x-y}_\alpha
    \le (1+\tau^\alpha)^{\ceil{\frac{T}\tau}}(\ceil{\tfrac{T}\tau}+1) \norm{w-\tilde w}_{\alpha}
\]
and therefore
\begin{align}
    \norm{x-y}_\alpha
    &= \norm{x-y}_\infty + \holder{x-y}_\alpha
    \\
    &\le \underbrace{\abs{x_0 - y_0}}_{=0} + \holder{x-y}_\alpha T^\alpha + \holder{x-y}_\alpha
    \\
    &\le \underbrace{(1+T^\alpha)(1+\tau^\alpha)^{\ceil{\frac{T}\tau}}(\ceil{\tfrac{T}\tau}+1)}_{\eqcolon C_{\text{param}}^R} \norm{w-\tilde w}_{\alpha}.
    \qedhere
\end{align}

\subsubsection{Technical Lemmas}

\begin{lemma}[Continuity of Young integrals]
    \label{lem: continuity of young integrals}
    Let \(\banachSpace[V]\) and \(\banachSpace[W]\) be Banach spaces and let
    \(\linOp{\banachSpace[V]}{\banachSpace[W]}\) be the space of bounded linear
    operators from \(\banachSpace[V]\) to \(\banachSpace[W]\) equipped with the
    operator norm. Let \(f \in C^\alpha([\ul{t}, \ol{t}], \linOp{\banachSpace[V]}{\banachSpace[W]})\) and \(g \in C^\beta([\ul{t},\ol{t}], \banachSpace[V])\)
    with $\alpha,\beta\in (0,1]$ and \(\alpha + \beta > 1\).
    Then, there exists a constant \(C_{\alpha,\beta}\) such that for all \(s,t\in [\ul{t}, \ol{t}]\)
    \[
        \abs[\Big]{
            \int_s^t f_u \, dg_u
            - f_s (g_t - g_s)
        }
        \le C_{\alpha, \beta} \holder{f}_{\alpha} \holder{g}_{\beta} \abs{t-s}^{\alpha + \beta}
    \]
    in particular
    \[
        \abs[\Big]{
            \int_s^t f_u \, dg_u
        }
        \le (\abs{f_s} + C_{\alpha, \beta} \holder{f}_{\alpha}\abs{t-s}^\alpha) \holder{g}_{\beta} \abs{t-s}^{\beta}
    \]
\end{lemma}
\begin{proof}
    See e.g.\ \citep[Theorem 6.8]{frizMultidimensionalStochasticProcesses2010}
    or \citep[Equation (4.3)]{frizCourseRoughPaths2020}.
\end{proof}

For \(T>0\) let \(w\in C^\alpha([0,T], \banachSpace[W])\) and for \(\ul{t}, \ol{t} \in [0,T]\) let
\(x, y\in C^\alpha([\ul{t},\ol{t}], \banachSpace[X])\). Define the map
\[
    f(x)_t \coloneq f(w,x)_t \coloneq \sigma(t, w_t, x_t).
\]

\begin{mdframed}[innertopmargin=0pt]
\begin{lemma}[Lipschitz continuity and boundedness of \(f\)]
    \label{lem: f Lipschitz with respect to alpha norm}
    Let \(\sigma\colon \real\times \banachSpace[W] \times \banachSpace[X] \to \linOp{\banachSpace[V]}{\banachSpace[X]}\) satisfy Assumption \ref{assmpt: sufficiently nice function},
    Then for every \(R>0\) there exist \(\ref{const: f holder bound}, \ref{const: f_bound} > 0\)
    such that for all \(w\in C^\alpha([0,T], \banachSpace[W])\) with \(\norm{w}_\alpha \le R\)
    and all \(x,y \in C^\alpha([\ul{t},\ol{t}], \banachSpace[X])\)
    \begin{align}
        \label{eq: local Lipschitz continuity of f}
        \holder{f(y) - f(x)}_{\alpha, [\ul{t}, \ol{t}]}
        &\le \ref{const: f holder bound}
        (1+\norm{x}_{\alpha, [\ul{t}, \ol{t}]}
        + \norm{y}_{\alpha, [\ul{t}, \ol{t}]})
        \Bigl(
            \holder{y-x}_{\alpha,[\ul{t},\ol{t}]}
            +\min_{t\in [\ul{t}, \ol{t}]}\abs{x_t-y_t}
        \Bigr)
        \\
        \label{eq: local boundedness of f}
        \holder{f(x)}_{\alpha, [\ul{t}, \ol{t}]}
        &\le \ref{const: f_bound}(1+\norm{x}_{\alpha, [\ul{t}, \ol{t}]}).
    \intertext{And with the additional assumption \ref{it: extra assumption} it holds that
        for all \(R, \tilde R > 0\) there exists \(\ref{const: f w diff holder bound} > 0\) such that
        for all \(w, \tilde w \in C^\alpha([0,T], \banachSpace[W])\)
        with \(\norm{w}_\alpha, \norm{\tilde w}_\alpha \le R\)
        and all \(x\in C^\alpha([\ul{t},\ol{t}], \banachSpace[X])\) with \(\norm{x}_\alpha \le \tilde R\)
    }
        \label{eq: f Lipschitz in w}
        \holder{f(w,x) - f(\tilde w, x)}_{\alpha, [\ul{t}, \ol{t}]}
        &\le \ref{const: f w diff holder bound}\norm{w-\tilde w}_\alpha .
    \end{align}
    Moreover the constants may be chosen as
    \begin{align}
        \defConst{f holder bound}
        &\coloneq \abs{\frechet_x\sigma(0,0,0)} + 2 K_c^R(1+R)(1+T^\alpha)
        & K_c^R \coloneq \smash{\max_{\substack{t\in [0,T]\\ r,s \in [0, R]}}}
        \max\set[\big]{c(t, r, s), \mathsf c(r) }< \infty
        \\
        \defConst{f_bound}
        &\coloneq K_c^R(1+2R)
        \\
        \defConst{f w diff holder bound}
        &\coloneq \abs{\frechet_w \sigma(0, 0, 0)} + K_c^{R, \tilde R}(1+2\tilde R + 3R + T^\alpha)
    \end{align}
    with
    \[
        K_{c}^{R, \tilde{R}} \coloneq \max_{\substack{t\in [0,T]\\ r_w,s_w \in [0, R]\\ r_x,s_x \in [0, \tilde R]}}
        \max\set[\big]{\tilde c(t, r_w, s_w, r_x, s_x), \tilde{\mathsf c}(r_w, r_x)}
        < \infty.
    \]
\end{lemma}
\end{mdframed}

\begin{proof}
    We first prove the \textbf{local Lipschitz continuity of \(f\)} \eqref{eq: local Lipschitz continuity of f}.
    Using \(v_t^\lambda \coloneq \lambda y_t + (1-\lambda) x_t\) we have
\begin{align}
    &\abs[\big]{f(y)_t - f(x)_t - f(y)_s + f(x)_s}
    \\
    &= \abs[\big]{\sigma(t, w_t, y_t) - \sigma(t, w_t, x_t) -  \sigma(s, w_s, y_s) + \sigma(s, w_s, x_s)}
    \\
    &= \abs[\bigg]{\int_0^1 \frechet_x \sigma(t, w_t, v_t^\lambda)d\lambda\, (y_t-x_t) - \int_0^1 \frechet_x \sigma(s, w_s, v_s^\lambda) d\lambda\, (y_s-x_s)}
    \\
    \label{eq: bounds on f diffs}
    &\le \begin{aligned}[t]
        \abs[\big]{y_t-x_t &- (y_s - x_s)}\int_0^1 \abs[\big]{\frechet_x \sigma(t, w_t, v_t^\lambda)} d\lambda
        \\
        &+ \abs[\big]{y_s - x_s}\int_0^1 \abs[\Big]{\frechet_x \sigma(t, w_t, v_t^\lambda) - \frechet_x \sigma(s, w_s, v_s^\lambda)} d\lambda.
    \end{aligned}
\end{align}
\begin{steps}
    \item \textbf{Bound on first summand.} 
    The factor in front is bounded by
    \begin{equation}
        \label{eq: y-x difference, nice case}
        \abs[\big]{y_t-x_t - (y_s - x_s)}
        \le \holder{y - x}_\alpha \abs{t-s}^\alpha.
    \end{equation}
    To bound the integral we use
    \begin{equation}
        \abs{v_t^\lambda}
        \le \lambda\abs{y_t} + (1-\lambda)\abs{x_t}
        \le \norm{y}_{\infty} + \norm{x}_{\infty}
    \end{equation}
    to get the following bound on the Fréchet derivative
    \begin{align}
        \abs[\big]{\frechet_x \sigma(t, w_t, v_t^\lambda)}
        &\le 
        \underbrace{
            \abs[\big]{\frechet_x \sigma(t, w_t, v_t^\lambda) -\frechet_x \sigma(0, w_t, v_t^\lambda)}
        }_{\le \violet{\mathsf c(\abs{w_t})} (1+\abs{v_t^\lambda}) \teal{\abs{t-0}^\alpha}}
        + \underbrace{
            \abs{\frechet_x\sigma(0, w_t, v_t^\lambda) - \frechet_x \sigma(0, 0, 0)}
        }_{
            \le \violet{c(0, \abs{w_t}, 0)}
            \bigl(\abs{v_t^\lambda - 0} + (1+\abs{v_t^\lambda}+0)\magenta{\abs{w_t - 0}}\bigr)
        }
        + \abs{\frechet_x \sigma(0, 0, 0)}
        \\
        &\le \underbrace{\Bigl(\violet{K_c^R}(1+ \magenta{R} + \teal{T^\alpha}) + \abs{\frechet_x \sigma(0, 0, 0)}\Bigr)}_{
            \eqcolon \defConst{sigma frechet derivative sup bound}
        }(1+ \norm{x}_\infty + \norm{y}_\infty)
        \label{eq: bound on frechet derivative norm}
    \end{align}
    using \(\norm{w}_\infty \le \norm{w}_\alpha \le R\). Combining \eqref{eq: bound on frechet derivative norm} with \eqref{eq: y-x difference, nice case} we get
    the following bound on the first summand in \eqref{eq: bounds on f diffs}
    \begin{equation}
        \label{eq: first summand bound}        
        \abs[\big]{y_t-x_t - (y_s - x_s)}\int_0^1 \abs[\big]{\frechet_x \sigma(t, w_t, v_t^\lambda)} d\lambda
        \le \ref{const: sigma frechet derivative sup bound}(1+ \norm{x}_\infty + \norm{y}_\infty)\holder{y-x}_\alpha \abs{t-s}^\alpha.
    \end{equation}

    \item \textbf{Bound on second summand.}
    By the triangle inequality and \eqref{eq: y-x difference, nice case}
    \begin{equation}
        \abs[\big]{y_s - x_s}
        \overset{\Delta}\le \min_{t\in [\ul{t}, \ol{t}]} \abs{y_t-x_t} + \abs{y_s - y_t - (x_s - x_t)}
        \label{eq: y-x difference, bad case}
        \overset{\eqref{eq: y-x difference, nice case}}\le \min_{t\in [\ul{t}, \ol{t}]} \abs{y_t-x_t} + \holder{y - x}_\alpha T^\alpha
    \end{equation}
    As \eqref{eq: y-x difference, nice case} is smaller for \(s\) close to \(t\), we need tighter bounds
    on the derivative difference than on the derivative itself. For this we use
    \begin{align}
        \abs{v_t^\lambda - v_s^\lambda}
        &\le \lambda\abs{y_t - y_s} + (1-\lambda)\abs{x_t - x_s}
        \\
        &\le (\holder{y}_\alpha + \holder{x}_\alpha) \abs{t-s}^\alpha.
    \end{align}
    together with \(\abs{w_t - w_s} \le \holder{w}_\alpha \abs{t-s}^\alpha\)
    and the previous bound \(\abs{v_t^\lambda} \le \norm{y}_\infty + \norm{x}_\infty\) we get
    \begin{align}
        &\abs[\big]{\frechet_x \sigma(t, w_t, v_t^\lambda) - \frechet_x \sigma(s, w_s, v_s^\lambda)}
        \\
        \overset{\Delta}&\le \underbrace{
            \abs[\big]{\frechet_x \sigma(t, w_t, v_t^\lambda) - \frechet_x \sigma(s, w_t, v_t^\lambda)}
        }_{
           \le \violet{\mathsf c(\abs{w_t})} (1+\abs{v_t^\lambda}) \abs{t-s}^\alpha
        }
        + \underbrace{\abs[\big]{\frechet_x \sigma(s, w_t, v_t^\lambda) - \frechet_x \sigma(s, w_s, v_s^\lambda)}
        }_{
            \le \violet{c(s, \abs{w_t}, \abs{w_s})} \bigl(\teal{\abs{v_t^\lambda - v_s^\lambda}}
            + (1+\abs{v_t^\lambda}+ \abs{v_s^\lambda})\magenta{\abs{w_t - w_s}}
            \bigr)
            \mathrlap{\quad 
            \text{(Assmpt.~\ref{assmpt: sufficiently nice function})}
            }
        }
        \\
        &\le \violet{K_c^R}(1+\norm{y}_\infty + \norm{x}_\infty)\abs{t-s}^\alpha
        + \violet{K_c^R} \bigl(\teal{\holder{y}_\alpha + \holder{x}_\alpha} + (1+2\norm{x}_\infty + 2\norm{y}_\infty)\magenta{\holder{w}_\alpha}\bigr)\abs{t-s}^\alpha
        \\
        \label{eq: bound on frechet derivative difference}
        &\le \underbrace{K_c^R(1+2R)}_{
            = \ref{const: f_bound}
        }
        (1+ \norm{y}_\alpha + \norm{x}_\alpha) \abs{t-s}^\alpha
    \end{align}
    using \(\holder{w}_\alpha \le \norm{w}_\alpha \le R\) and \(\norm{x}_\alpha = \norm{x}_\infty + \holder{x}_\alpha\).
    Combining \eqref{eq: y-x difference, bad case} with \eqref{eq: bound on
    frechet derivative difference} we get the following bound on the second summand in \eqref{eq: bounds on f diffs}
    \begin{align}
        &\abs[\big]{y_s - x_s}\int_0^1 \abs[\Big]{\frechet_x \sigma(t, w_t, v_t^\lambda) - \frechet_x \sigma(s, w_s, v_s^\lambda)} d\lambda
        \\
        \label{eq: second summand bound}
        &\le
        \Bigl(\min_{t\in [\ul{t}, \ol{t}]} \abs{y_t-x_t} + \holder{y - x}_\alpha T^\alpha\Bigr)
        \ref{const: f_bound}
        (1+ \norm{y}_\alpha + \norm{x}_\alpha) \abs{t-s}^\alpha.
    \end{align}
\end{steps}
Combining the bound on the first summand \eqref{eq: first summand bound} and second summand \eqref{eq: second summand bound}
in \eqref{eq: bounds on f diffs} we get \eqref{eq: local Lipschitz continuity of f}, that is
\begin{align}
    \holder{f(y)-f(x)}_\alpha
    &= \sup_{s\neq t\in [\ul{t}, \ol{t}]}\frac{\abs{f(y)_t - f(x)_t - f(y)_s + f(x)_s}}{\abs{t-s}^\alpha}
    \\
    &\le 
    \underbrace{(\ref{const: sigma frechet derivative sup bound} + \ref{const: f_bound}T^\alpha)}_{\le \ref{const: f holder bound}}
    (1+ \norm{y}_\alpha + \norm{x}_\alpha) \Bigl(\min_{t\in [\ul{t}, \ol{t}]}\abs{y_t-x_t} + \holder{y-x}_\alpha\Bigr)
\end{align}
with the constant
\begin{align}
    \ref{const: sigma frechet derivative sup bound} + \ref{const: f_bound}T^\alpha
    &= \abs{\frechet_x\sigma(0,0,0)} + K_c^R(1+ R + T^\alpha) + K_c^R(1+2R) T^\alpha
    \\
    &\le
    \abs{\frechet_x\sigma(0,0,0)} + 2 K_c^R(1+R)(1+T^\alpha)
    \\
    \overset{\text{def.}}&= \ref{const: f holder bound}.
\end{align}
For the \textbf{local boundedness of \(f\)} \eqref{eq: local boundedness of f} we simply
use Assumption \ref{assmpt: sufficiently nice function} to get
\begin{align}
    \holder{f(x)}_{\alpha, [\ul{t}, \ol{t}]}
    &= \sup_{s\neq t\in [\ul{t}, \ol{t}]}\frac{\abs{\sigma(t,w_t,x_t) - \sigma(s, w_s, x_s)}}{\abs{t-s}^\alpha}
    \\
    \overset{\Delta}&\le 
    \sup_{s\neq t\in [\ul{t}, \ol{t}]}\frac{\abs{\sigma(t,w_t,x_t) - \sigma(s, w_t, x_t)}}{\abs{t-s}^\alpha}
    + \frac{\abs{\sigma(s,w_t,x_t) - \sigma(s, w_s, x_s)}}{\abs{t-s}^\alpha}
    \\
    \overset{\text{Assmpt.~\ref{assmpt: sufficiently nice function}}}&\le
    K_c^R(1+\norm{x}_{\infty,[\ul{t}, \ol{t}]}) + K_c^R
    \Bigl(\holder{x}_{\alpha,[\ul{t}, \ol{t}]}
    + (1+2\norm{x}_{\infty, [\ul{t}, \ol{t}]})\underbrace{\holder{w}_\alpha}_{\le R}\Bigr)
    \\
    &\le \underbrace{K_c^R(1+2R)}_{=\ref{const: f_bound}}
    (1+ \norm{x}_{\alpha, [\ul{t}, \ol{t}]}).
\end{align}
Finally, we prove the \textbf{local Lipschitz continuity of \(f\) in \(w\)} \eqref{eq: f Lipschitz in w}.
The proof is similar to that of local Lipschitz continuity of \(f\) in \(x\).
We begin by defining the convex combination \(w_t^\lambda \coloneq (1-\lambda)w_t + \lambda\tilde w_t\) with \(\lambda \in [0,1]\)
such that
\begin{align}
    \abs{w_t^\lambda - w_s^\lambda}
    &\le (1-\lambda)\abs{w_t - w_s} + \lambda\abs{\tilde w_t - \tilde w_s}
    \\
    &\le (\holder{w}_\alpha + \holder{\tilde w}_\alpha) \abs{t-s}^\alpha
    \\
    &\le 2R \abs{t-s}^\alpha
    \\[1ex]
    \abs{w_t^\lambda}
    &\le (1-\lambda)\abs{w_t} + \lambda\abs{\tilde w_t}
    \\
    &\le \max\set{\norm{w}_\infty, \norm{\tilde w}_\infty}
    \\
    &\le R.
\end{align}
With the following bound on the Fréchet derivative using Assumption~\ref{assmpt: sufficiently nice function} \ref{it: extra assumption}
and the definition of \(K_c^{R, \tilde R}\)
\begin{align}
    &\abs{\frechet_w \sigma(t, w_t^\lambda, x_t)}
    \\
    &\le \underbrace{
        \abs{\frechet_w \sigma(t, w_t^\lambda, x_t) - \frechet_w \sigma(0, w_t^\lambda, x_t)}
    }_{
        \le K_c^{R, \tilde R} \abs{t-0}^\alpha
        \le K_c^{R, \tilde R} T^\alpha
    }
    + \underbrace{
        \abs{\frechet_w \sigma(0, w_t^\lambda, x_t) - \frechet_w \sigma(0, 0, 0)}
    }_{
        \le K_c^{R, \tilde R} (\abs{x_t} + \abs{w_t^\lambda})
        \le K_c^{R, \tilde R}(\tilde R + R)
    }
    + \abs{\frechet_w \sigma(0, 0, 0)}
    \\
    &\le \defConst{frechet w sigma sup bound} \coloneq K_c^{R, \tilde R}(\tilde R + R + T^\alpha) + \abs{\frechet_w \sigma(0, 0, 0)}.
\end{align}
we get 
\begin{align}
    &\abs{f(w,x)_t - f(w,x)_s - f(\tilde w, x)_t + f(\tilde w, x)_s}
    \\
    &= \abs{\sigma(t, w_t, x_t) - \sigma(t, \tilde w_t, x_t) - \sigma(s, w_s, x_s) + \sigma(s, \tilde w_s, x_s)}
    \\
    &= \abs[\bigg]{\int_0^1 \frechet_w \sigma(t, w_t^\lambda, x_t)d\lambda\,(w_t - \tilde w_t) - \int_0^1 \frechet_w \sigma(s, w_s^\lambda, x_s) d\lambda\,(w_s - \tilde w_s)}
    \\
    &\le \begin{aligned}[t]
        &\underbrace{
            \abs{w_t - \tilde w_t - (w_s - \tilde w_s)}
        }_{
            \le \holder{w-\tilde w}_\alpha \abs{t-s}^\alpha
        }
        \int_0^1 \underbrace{\abs{\frechet_w \sigma(t, w_t^\lambda, x_t)}}_{\le \ref{const: frechet w sigma sup bound}} d\lambda
        \\
        &\quad + \underbrace{\abs{w_s - \tilde w_s}}_{
            \le \norm{w -\tilde w}_\infty
        }\int_0^1
        \underbrace{
            \abs{\frechet_w \sigma(t, w_t^\lambda, x_t) - \frechet_w \sigma(s, w_s^\lambda, x_s)}
        }_{
            \begin{aligned}
                & \scriptstyle
                \le K_c^{R, \tilde R}(\abs{t-s}^\alpha + \abs{x_t - x_s} + \abs{w_t^\lambda - w_s^\lambda})
                \mathrlap{\quad\text{(Assmpt.~\ref{assmpt: sufficiently nice function} \ref{it: extra assumption} + \ref{it: (nice function) locally Hölder continuous})}}
                \\
                & \scriptstyle
                \le K_c^{R, \tilde R}(1+ \holder{x}_\alpha + \holder{w}_\alpha + \holder{\tilde w}_\alpha)\abs{t-s}^\alpha
                \\
                & \scriptstyle
                \le K_c^{R, \tilde R}(1+\tilde R + 2R)\abs{t-s}^\alpha
            \end{aligned}
        } d\lambda
    \end{aligned}
    \\
    &\le \norm{w-\tilde w}_\alpha \bigl(\ref{const: frechet w sigma sup bound} + K_c^{R, \tilde R}(1+\tilde R + 2R)\bigr)\abs{t-s}^\alpha.
\end{align}
And consequently
\[
    \holder{f(w,x) - f(\tilde w, x)}_\alpha
    \le \norm{w-\tilde w}_\alpha \underbrace{\bigl(\ref{const: frechet w sigma sup bound} + K_c^{R, \tilde R}(1+\tilde R + 2R)\bigr)}_{= \ref{const: f w diff holder bound}}.
\]
This proves the final claim.
\end{proof}

\subsection{Proof of Theorem \ref{thm: convergence of euler scheme}}

The heart of the proof in sup-norm convergence is an incremental restart of the flow
at the Euler method points and the fact that the flow is locally Lipschitz in the
initial condition. But since we only have \emph{local} Lipschitz continuity,
we need to carefully construct a sufficiently large ball to encompass both
the ODE solution and the Euler discretization.
For a fixed \(R>0\) that bounds the driver \(g\), the parameter \(w\) and
initial condition \(a\) we define 
\begin{equation}
    \label{eq: r(R) definition}    
    r(R) \coloneq 4 \max\set{C_{\text{flow}}^R, R} \ge 4 \norm{x}_\alpha.
\end{equation}
Then for all \(s,t \in [0,T]\) with \(\abs{t-s} \le \abs{\pi}\) we have by the
local Hölder bound on the flow from Theorem \ref{thm: differential equation solution existence and uniqueness} \ref{it: flow bound}
\begin{align}
    \abs{\flow(b; s,t)}
    &\le \abs{b} + \abs{\flow(b; s, t) - \flow(b; s, s)}
    \\
    &\le \frac{r(R)}2 + C_{\text{flow,loc}}^{r(R)} \abs{\pi}^{\beta-\alpha}
    && \text{for}\quad \abs{b} \le \tfrac{r(R)}2
    \\
    \label{eq: bound on flow}
    &\le r(R) 
    && \text{for}\quad \abs{\pi} \le \bigl(\tfrac{r(R)}{2C_{\text{flow,loc}}^{r(R)}}\bigr)^{\frac1{\beta-\alpha}}
\end{align}
For \(b\) and \(\abs{\pi}\) selected to satisfy \eqref{eq: bound on flow} we
moreover have by Lemma \ref{lem: f Lipschitz with respect to alpha norm}
that for \(\abs{t-s} \le \abs{\pi}\)
\begin{align}
    \label{eq: bound on f at flow}    
    \holder{f(\flow(b; s, \cdot))}_{\alpha, [s,t]}
    &\le \ref{const: f_bound}(1+\norm{\flow(b; s, \cdot)}_{\alpha, [s,t]})
    \\
    &\le \ref{const: f_bound}(1+\underbrace{\norm{\flow(b; s, \cdot)}_{\infty, [s,t]}}_{\le r(R) \quad \eqref{eq: bound on flow}} + \underbrace{\holder{\flow(b;s, \cdot)}_{\alpha}}_{\le C_{\text{flow}}^{r(R)}})
    \\[-2ex]
    &\le\defConst{uniform bound on f_alpha}.
\end{align}
with \(\ref{const: uniform bound on f_alpha} \coloneq \ref{const: f_bound}(1+r(R) + C_{\text{flow}}^{r(R)})\).
To ensure that \eqref{eq: bound on flow} is
in force we select
\begin{equation}
    \label{eq: tau definition}
    \tau
    \coloneq \min\set[\big]{
        1,
        \underbrace{
            (\tfrac{r(R)}{2C_{\text{flow,loc}}^{r(R)}})^{\frac1{\beta-\alpha}},
            (\tfrac{r(R)}{4C_{\text{Euler}}^{R,\alpha, 1}})^{\frac1{\alpha+\beta -1}}
        }_{\text{sup-norm bound (\ref{step: discrete sup-norm bound})}},
        \underbrace{
            \tfrac12(2K_{\mathrm{sew}}^R)^{-1/\beta}
        }_{\text{Hölder bound} \mathrlap{\text{ (\ref{step: discrete Hölder bound})}}}
    },
\end{equation}
with \(
    C_{\text{Euler}}^{R,\alpha, 1} \coloneq C_{\text{init}}^{r(R)} C_{\alpha, \beta} \ref{const: uniform bound on f_alpha} R T
\), where \(C_{\text{init}}^{r(R)}\) and  \(2C_{\text{flow,loc}}^{r(R)}\) are the constants
from Theorem \ref{thm: differential equation solution existence and uniqueness} and
\(K_{\mathrm{sew}}^R\) is the constant from the discrete sewing lemma (Lemma \ref{lem: discrete sewing}).
If we can keep the Euler discretization within the ball of radius \(\frac{r(R)}2\), then 
for \(\abs{\pi}\le \tau\) we can apply \eqref{eq: bound on flow}. This turns
out to be possible. Indeed we will prove in the first
step that a constant  \(C_{\text{Euler}}^{R, \alpha,1}>0\) exists such that
for all \(\abs{\pi} \le \tau\) we have the uniform bound
\begin{equation}
    \label{eq: uniform bound on euler scheme}
    \norm{x^\pi - x}_{\infty, \pi} 
    \le C_{\text{Euler}}^{R, \alpha,1}
    \abs{\pi}^{\alpha+\beta -1}
    \qquad\text{and}\qquad
    \norm{x^\pi}_{\infty, \pi} \le \tfrac{r(R)}2.
\end{equation}
with \(\norm{x^\pi - x}_{\infty, \pi} \coloneq \sup_{k} \abs{x^\pi_k - x_{t_k}}\)
and \(\norm{x^\pi}_{\infty, \pi} \coloneq \sup_k \abs{x^\pi_k}\).
\begin{steps}
\item\label{step: discrete sup-norm bound} \textbf{Bound in discrete sup-norm.} 
We prove \eqref{eq: uniform bound on euler scheme} by induction.
That is, for all \(k\in \set{0,\dots,n}\) we show
\[
    \abs{x^\pi_k - x_{t_k}} = \abs{x^\pi_k - \flow(a; 0, t_k)}
    \le C_{\text{Euler}}^{R,\alpha, 1} \abs{\pi}^{\alpha+\beta -1}
    \quad\text{and} \quad \abs{x^\pi_k} \le \tfrac{r(R)}2.
\]
The second claim is needed to ensure we can apply \eqref{eq: bound on flow}
and \eqref{eq: bound on f at flow} in the induction step and
we will also require this uniform bound on the Euler discretization in
later proof steps.

The induction start \(k=0\) is trivial, since \(x^\pi_0 = a = \flow(a; 0, 0)\). For the induction step we have 
\[\begin{aligned}
    \abs{x^\pi_k - \flow(a; 0, t_k)}
    &\le \sum_{l=1}^{k} \abs[\big]{\flow(x^\pi_{l}; t_{l}, t_k) - \flow(x^\pi_{l-1}; t_{l-1}, t_k)}
    \\
    &\le \sum_{l=1}^k \abs[\big]{\flow(x^\pi_{l}; t_{l}, t_k) - \flow(\flow(x^\pi_{l-1}; t_{l-1}, t_l); t_{l}, t_k)}
    \\
    &= \abs{x^\pi_k - \flow(x^\pi_{k-1}; t_{k-1}, t_k)}
    + \sum_{l=1}^{k-1} \underbrace{
        \abs[\big]{\flow(x^\pi_{l}; t_{l}, t_k) - \flow(\flow(x^\pi_{l-1}; t_{l-1}, t_l); t_{l}, t_k)}
    }_{
        \le C_{\text{init}}^{r(R)} \abs{x^\pi_l - \flow(x^\pi_{l-1}; t_{l-1}, t_l)}
        \qquad \mathrlap{\text{Induct. + \eqref{eq: bound on flow}} }
    }
    \\[-1ex]
    &\le C_{\text{init}}^{r(R)} \sum_{l=1}^k  \abs{x^\pi_l - \flow(x^\pi_{l-1}; t_{l-1}, t_l)}.
\end{aligned}\]
Here \(C_{\text{init}}^{r(R)}\) is the constant from the local Lipschitz continuity in the initial condition
(see \ref{it: lipschitz in initial condition} of Theorem \ref{thm: differential equation solution existence and uniqueness}),
which is applicable since \(x^\pi_l, \flow(x^\pi_{l-1}; t_{l-1}, t_l) \in B(0, r(R))\) due to \eqref{eq: bound on flow} and the induction hypothesis.
Now we may bound the individual terms
\[\begin{aligned}
    \abs{x^\pi_l - \flow(x^\pi_{l-1}; t_{l-1}, t_l)}
    &\le \abs[\Big]{\sigma(t_{l-1}, w_{t_{l-1}}, x^\pi_{l-1}) (g_{t_l} - g_{t_{l-1}}) - \int_{t_{l-1}}^{t_l} \sigma(s, w_s, \flow(x^\pi_{l-1}; t_{l-1}, s)) \, dg_s}
    \\
    \overset{\text{Lem.~\ref{lem: continuity of young integrals}}}&\le
    C_{\alpha, \beta} \underbrace{\holder[\big]{f(\flow(x^\pi_{l-1}; t_{l-1}, \cdot))}_{\alpha, [t_{l-1}, t_l]}}_{\le \ref{const: uniform bound on f_alpha}}
    \underbrace{\holder{g}_\beta}_{\le R} \abs{t_l - t_{l-1}}^{\alpha + \beta},
\end{aligned}\]
using \eqref{eq: bound on f at flow} in the last step. This bound on
the individual terms together with
\[\begin{aligned}
    \abs{t_l - t_{l-1}}^{\alpha+\beta}
    &= \abs{t_l - t_{l-1}}\abs{t_l - t_{l-1}}^{\alpha+\beta-1}
    \\
    &\le \abs{t_l - t_{l-1}}\abs{\pi}^{\alpha+\beta-1}
\end{aligned}
\]
results in
\begin{equation}
    \label{eq: Euler scheme convergence, induction step finish}    
    \abs{x^\pi_k - \flow(a; 0, t_k)}
    \le C_{\text{init}}^{r(R)} C_{\alpha, \beta} \ref{const: uniform bound on f_alpha} R \underbrace{\sum_{l=1}^k \abs{t_l - t_{l-1}}}_{=t_k \le T}\abs{\pi}^{\alpha+\beta -1}
    \le C_{\text{Euler}}^{R,\alpha,1} \abs{\pi}^{\alpha+\beta -1}
\end{equation}
using the constant \(C_{\text{Euler}}^{R,\alpha,1} = C_{\text{init}}^{r(R)} C_{\alpha, \beta} \ref{const: uniform bound on f_alpha} R T\).
This proves the first claim. Using \(\norm{x}_\infty \le \frac{r(R)}{4}\) by
definition \eqref{eq: r(R) definition} we also have
the second claim. Indeed we have
\[
    \abs{x^\pi_k}
    \le \abs{x^\pi_k - \flow(a; 0, t_k)} + \abs{\flow(a; 0, t_k)}
    \le \tfrac{r(R)}{4} + \norm{x}_\infty \le \tfrac{r(R)}{2},
\]
due to \(\abs{\pi} \le (\frac{r(R)}{4C_{\text{Euler}}^{R,\alpha,1}})^{\frac1{\alpha+\beta
-1}}\) combined with \eqref{eq: Euler scheme convergence, induction step finish} for the difference. This completes the induction and we thus have proved \eqref{eq: uniform bound on euler scheme}.

\item \textbf{Bound in sup-norm.}
Next we bound the piecewise linear interpolation
\(\bar x^\pi\). For \(t\in [t_k, t_{k+1})\) let \(\lambda_t =
\frac{t-t_k}{t_{k+1}-t_k}\). Then we have
\[
    \bar x^\pi_t = (1-\lambda_t)x^\pi_{k} + \lambda_t x^\pi_{k+1} \qquad t\in [t_k, t_{k+1}).
\]
Define an interpolated version of the ODE solution \(\bar x_t \coloneq (1-\lambda_t)x_{t_k} + \lambda_t x_{t_{k+1}}\). Then we have
\begin{alignat}{3}
    \abs{\bar x^\pi_t - x_t}
    &\le \abs{\bar x_t^\pi - \bar x_t}
    + \abs{\bar x_t - x_t}
    \\
    &\le \lambda_t \underbrace{\abs{x^\pi_{k+1} - x_{t_{k+1}}}}_{\le C_{\text{Euler}}^{R,\alpha,1} \abs{\pi}^{\alpha+\beta-1}}
    + (1-\lambda_t) \underbrace{\abs{x^\pi_k - x_{t_k}}}_{\le C_{\text{Euler}}^{R,\alpha,1} \mathrlap{\abs{\pi}^{\alpha+\beta-1}}}
    + \lambda_t \underbrace{\abs{x_{t_{k+1}} - x_t}}_{\le \norm{x}_\alpha \abs{\pi}^\alpha}
    + (1-\lambda_t) \underbrace{\abs{x_{t_k} - x_t}}_{\le \norm{x}_\alpha \abs{\pi}^\alpha}
    \\
    &\le C_{\text{Euler}}^{R,\alpha,1} \abs{\pi}^{\alpha+\beta -1} + \norm{x}_\alpha \abs{\pi}^\alpha
    \\
    &\le \underbrace{(C_{\text{Euler}}^{R,\alpha,1} + C_{\text{flow}}^R \tau^{1-\beta})}_{\eqcolon C_{\mathrm{Euler}}^{R, \alpha, 2}}\abs{\pi}^{\alpha+\beta -1} .
\end{alignat}
using \(\norm{x}_\alpha \le C_{\text{flow}}^R\) and \(\abs{\pi} \le \tau\) in the last inequality.
This proves the bound of the sup-norm 
\begin{equation}
    \label{eq: Euler scheme convergence, sup-norm}
    \norm{\bar x^\pi - x}_\infty \le C_{\mathrm{Euler}}^{R, \alpha, 2} \abs{\pi}^{\alpha+\beta -1}
    \qquad \forall \abs{\pi} \le \tau.
\end{equation}

\item\label{step: discrete Hölder bound} \textbf{Bound of discrete Hölder semi-norm.} For the bound in the Hölder norm we want to apply
Lemma \ref{lem: uniform to holder bound} with \(\epsilon = \alpha -
\alpha'\). This requires a uniform bound on the Hölder semi-norms of \(x\) and
\(\bar x^\pi\). For \(x\) we already have \(\holder{x}_\alpha \le
C_{\text{flow}}^R\) by Theorem \ref{thm: differential equation solution
existence and uniqueness} \ref{it: flow bound}. For 
the bound on \(\bar x^\pi\) we have to work.

From the discretization \(\pi=\set{t_0, \dots, t_n}\) select a subset
of anchors \(\pi'= \set{T_0, \dots, T_m}\) with \(T_0 = 0\),
\(T_m = T\) and
\[
    \tau \le T_i - T_{i-1} \le 2\tau, \quad\forall i \in \set{1, \dots, m-1} \qquad \text{and}\qquad T_m - T_{m-1} \le 2\tau.
\]
This is possible since \(\abs{\pi} \le \tau\) implies there exists \(t_k\)
between \(T_i+\tau\) and \(T_i + 2\tau\) which may be selected as
\(T_{i+1}\). This selection of anchors ensures that their number \(m\) is uniformly
bounded independent of \(\pi\). Indeed, we have
\begin{equation}
    \label{eq: bound on number of anchors}
    T \ge T_{m-1} = \sum_{i=1}^{m-1} (T_i - T_{i-1}) \ge (m-1) \tau
    \quad\implies\quad
    m \le \tfrac{T}{\tau}+1.
\end{equation}

Using \(z_i \coloneq \sigma(t_i, w_{t_i}, x^\pi_i)\) and
\(A_{i,j} \coloneq z_i (g_{t_j} - g_{t_i})\)
we may express the increments of \(x^\pi\) as
\[
    x_j^\pi - x_i^\pi
    = \sum_{l=i}^{j-1} z_l (g_{t_{l+1}} - g_{t_l})
    = \sum_{l=i}^{j-1} A_{l, l+1}.
\]
Our first step is to obtain a uniform Hölder bound on \(x^\pi\) on the intervals spanned
by the anchors. For this let \(i,j\) be such that \([t_i, t_j) \subseteq
[T_{k-1}, T_k) = [t_l, t_{l'})\) for some \(k\), \(l\) and \(l'\). Then we
have by the discrete sewing Lemma (Lemma \ref{lem: discrete sewing})
\begin{align}
    \abs{x_j^\pi - x_i^\pi}
    &\le \abs{x_j^\pi - x_i^\pi - A_{i,j}}
    + \abs{A_{i,j}}
    \\
    &= \abs[\Big]{
        \sum_{l=i}^{j-1} A_{l, l+1}
        - A_{i,j}
    }
    + \abs{z_i (g_{t_j} - g_{t_i})}
    \\
    \overset{\text{Lem.~\ref{lem: discrete sewing}}}&\le
    K_{\mathrm{sew}}^R (1+ \norm{x^\pi}_{\alpha, [i:j]}) \abs{t_j - t_i}^{\alpha+\beta}
    + \norm{\sigma}_\infty \holder{g}_\beta \abs{t_j - t_i}^\beta
    \\
    &\le
    K_{\mathrm{sew}}^R (1+ \norm{x^\pi}_{\alpha, [l:l']}) \abs{t_j - t_i}^{\alpha+\beta}
    + \norm{\sigma}_\infty R \abs{t_j - t_i}^\beta,
\end{align}
where
\(\norm{x^\pi}_{\alpha, [l:l']}\coloneq \norm{x^\pi}_{\infty, [l:l']} + \holder{x^\pi}_{\alpha, [l:l']}\)
with
\[
    \norm{x^\pi}_{\infty, [l:l']}
    \coloneq \sup_{l\le k \le l'} \abs{x^\pi_k}
    \qquad \text{and}\qquad
    \holder{x^\pi}_{\alpha, [l:l']}
    \coloneq \sup_{l\le i <j \le l'} \frac{\abs{x_j^\pi - x_i^\pi}}{\abs{t_j - t_i}^\alpha}.
\]
Due to \(\abs{t_j - t_i} \le \abs{T_k - T_{k-1}} \le 2\tau\) we thus have
\begin{align}
    \holder{x^\pi}_{\alpha, [l:l']}
    &= \sup_{i\neq j \in [l:l']} \frac{\abs{x_j^\pi - x_i^\pi}}{\abs{t_j - t_i}^\alpha}
    \\
    &\le K_{\mathrm{sew}}^R (1+ \norm{x^\pi}_{\infty, [l:l']} + \holder{x^\pi}_{\alpha, [l:l']}) \abs{t_j - t_i}^{\beta}
    + \norm{\sigma}_\infty R \abs{t_j - t_i}^{\beta-\alpha}
    \\
    &\le K_{\mathrm{sew}}^R (1+ r(R) + \holder{x^\pi}_{\alpha, [l:l']}) (2\tau)^{\beta}
    + \norm{\sigma}_\infty R (2\tau)^{\beta-\alpha}
\end{align}
Due to the selection of \(\tau \le \frac12(2K_{\mathrm{sew}}^R)^{-1/\beta}\) in \eqref{eq: tau definition} we have
\(K_{\mathrm{sew}}^R (2\tau)^{\beta} \le \frac12\) and therefore
\[
    \holder{x^\pi}_{\alpha, [l:l']}
    \le \frac{K_{\mathrm{sew}}^R (1+ r(R)) (2\tau)^{\beta} + \norm{\sigma}_\infty R (2\tau)^{\beta-\alpha}}{1 - K_{\mathrm{sew}}^R (2\tau)^{\beta}}
    \le 1+ r(R) + 2\norm{\sigma}_\infty R (2\tau)^{\beta-\alpha}
    \eqcolon M_0.
\]
Since this bound does not depend on \([T_{k-1}, T_k)\) we thereby have a uniform
Hölder bound on \(x^\pi\) on each of these intervals. Since the number of these intervals
is uniformly bounded independent of \(\pi\), we only need to glue these bounds together to get a uniform
Hölder bound on \(x^\pi\) on the entire interval \([0,T]\). For this we
use
\begin{equation}
    \label{eq: concave bound}
    \sum_{i=1}^n y_i^{\alpha}
    = n\sum_{i=1}^n \tfrac1n y_i^{\alpha}
    \overset{\substack{\text{concave}\\\text{Jensen}}}\le
    n\Bigl(\sum_{i=1}^n \tfrac1n y_i\Bigr)^{\alpha}
    = n^{1-\alpha} \Bigl(\sum_{i=1}^n y_i\Bigr)^{\alpha}.
\end{equation}
In the following we will write \(x_{T_k}^\pi\coloneq x_{l}^\pi\) for \(T_k=t_l\) to avoid cumbersome notation.
Then we have for \(i,j\) with \(t_i \in [T_{k-1}, T_k)\) and \(t_j \in [T_{k'}, T_{k'+1})\) such that \(k \le k'\)
\begin{align}
    \abs{x_j^\pi - x_i^\pi}
    &\le \abs{x_j^\pi - x_{T_{k'}}^\pi} + \sum_{l=k}^{k'-1} \abs{x_{T_{l+1}}^\pi - x_{T_l}^\pi} + \abs{x_{T_{k}}^\pi - x_i^\pi}
    \\
    &\le M_0(t_j - T_{k'})^{\alpha} + M_0 \sum_{l=k}^{k'-1} (T_{l+1} - T_l)^{\alpha} + M_0 (T_k - t_i)^{\alpha}
    \\
    \overset{\eqref{eq: concave bound}}&\le M_0 \underbrace{(k'-k+2)^{1-\alpha}}_{\le m^{1-\alpha}} (t_j - t_i)^{\alpha}
    \\
    \overset{\eqref{eq: bound on number of anchors}}&\le \underbrace{M_0 \bigl(\tfrac{T}{\tau}+1\bigr)^{1-\alpha}}_{\eqcolon M_1} (t_j - t_i)^{\alpha}.
\end{align}
Due to \(M_0 \le M_1\) the constant \(M_1\) may also be used in the case of
\(t_i, t_j \in [T_{k-1}, T_k)\) and we obtain a discrete uniform Hölder
bound on \(x^\pi\)
\[
    \holder{x^\pi}_{\alpha,\pi} \coloneq  \sup_{0\le i < j \le n} \frac{\abs{x_j^\pi - x_i^\pi}}{\abs{t_j - t_i}^{\alpha}} \le M_1.
\]

\item\textbf{Bound of the Hölder norm.}
For the interpolation \(\bar x^\pi\) we have for \(t,s\in [t_k, t_{k+1})\)
\[
    \abs{\bar x^\pi_t - \bar x^\pi_s}
    = \frac{\abs{t-s}}{t_{k+1}-t_k} \abs{x^\pi_{k+1} - x^\pi_k}
    \le \abs{t-s} M_1 \abs{t_{k+1} - t_k}^{\alpha-1} \le M_1\abs{t-s}^{\alpha}.
\]
And for \(s\in [t_{j-1}, t_j)\) and \(t\in [t_i, t_{i+1})\) with \(j\le i\) we therefore get
\begin{align}
    \abs{\bar x^\pi_t - \bar x^\pi_s}
    &\le \abs{\bar x^\pi_t - x^\pi_i} + \abs{x^\pi_i - x^\pi_j} + \abs{x^\pi_j - \bar x^\pi_s}
    \\
    &\le M_1 (t-t_i)^{\alpha}
    + M_1 (t_i - t_j)^{\alpha} + M_1 (t_j - s)^{\alpha}
    \\
    \overset{\eqref{eq: concave bound}}&\le 3^{1-\alpha}M_1  \abs{t-s}^{\alpha}.
\end{align}
Put together we have using \(M \coloneq \max\set{3^{1-\alpha}M_1, C_{\mathrm{flow}}^R} \ge M_1\)
\[
    \holder{\bar{x}^\pi}_{\alpha}
    = \sup_{s\neq t\in [0,T]} \frac{\abs{\bar x^\pi_t - \bar x^\pi_s}}{\abs{t-s}^{\alpha}}
    \le M
\]
uniformly over \(\pi\). And for the solution of the differential equation \(x\) we have \(\holder{x}_\alpha \le
C_{\text{flow}}^R\le M\) by Theorem \ref{thm: differential equation solution
existence and uniqueness} \ref{it: flow bound}. With the application of Lemma \ref{lem: uniform to holder bound} with \(\epsilon = \alpha - \alpha'\) we thus have
\[
    \holder{\bar x^\pi - x}_{\alpha'}
    \le 2M^{\frac{\alpha'}{\alpha}} \norm{\bar x^\pi - x}_\infty^{1-\frac{\alpha'}{\alpha}}
\]
This implies
\begin{align}
    \norm{\bar x^\pi - x}_{\alpha'}
    &= \norm{\bar x^\pi - x}_\infty + \holder{\bar x^\pi - x}_{\alpha'}
    \\
    &\le \norm{\bar x^\pi - x}_\infty + 2M^{\frac{\alpha'}{\alpha}} \norm{\bar x^\pi - x}_\infty^{1-\frac{\alpha'}{\alpha}}
    \\
    \overset{\text{\eqref{eq: Euler scheme convergence, sup-norm}}}&\le 
        \Bigl(C_{\mathrm{Euler}}^{R, \alpha, 2} \abs{\pi}^{\frac{\alpha'}{\alpha}(\alpha+\beta -1)}
        + 2M^{\frac{\alpha'}{\alpha}} (C_{\mathrm{Euler}}^{R, \alpha, 2})^{1-\frac{\alpha'}{\alpha}}\Bigr)
    \abs{\pi}^{(1-\frac{\alpha'}{\alpha})(\alpha+\beta - 1)}
    \\
    &\le \underbrace{
        \max\set{1, C_{\mathrm{Euler}}^{R, \alpha, 2}} (1+ 2M)
    }_{\eqcolon C_{\mathrm{Euler}}^{R, \alpha}}
    \abs{\pi}^{(1-\frac{\alpha'}{\alpha})(\alpha+\beta - 1)},
\end{align}
where we use \(\alpha' < \alpha\) with \(M \ge M_0 \ge 1\) and \(\abs{\pi} \le \tau \le 1\) in the last inequality.
\qed
\end{steps}

The key ingredient to turn the uniform convergence into convergence in
Hölder space is the following Lemma.

\begin{lemma}[Uniform to Hölder bound]
\label{lem: uniform to holder bound}
Let $\alpha\in (0,1]$. For \(x,y \in C^\alpha([0,T], \banachSpace[X])\) assume there exists \(M>0\) such that
\(\holder{x}_\alpha, \holder{y}_\alpha \le M\); then for all
\(\epsilon\in (0, \alpha)\) we have
\[
    \holder{x - y}_{\alpha-\epsilon} \le 2M^{1-\frac{\epsilon}{\alpha}} \norm{x - y}_\infty^{\frac{\epsilon}{\alpha}}.
\]
\end{lemma}
\begin{proof}
Let \(e_t \coloneq x_t - y_t\) be the error of the approximation. Then
\[
    \abs{e_t - e_s}
    \le \abs{e_t} + \abs{e_s}
    \le 2\norm{e}_\infty
\]
and due to \(\holder{x}_\alpha, \holder{y}_\alpha \le M\)
\[
    \abs{e_t - e_s}
    \le \abs{x_t - x_s} + \abs{y_t - y_s}
    \le 2M \abs{t-s}^\alpha.
\]
This implies the error difference is bounded by the minimum and we may bound this
by any combination of the two bounds. That is
\begin{align}
    \abs{e_t - e_s}
    \le \min\set{2\norm{e}_\infty, 2M \abs{t-s}^\alpha}
    \le 2\norm{e}_\infty^{\frac{\epsilon}{\alpha}} (M\abs{t-s}^{\alpha})^{1-\frac{\epsilon}{\alpha}}
    = 2M^{1-\frac{\epsilon}{\alpha}}\norm{x-y}_\infty^{\frac{\epsilon}{\alpha}} \abs{t-s}^{\alpha-\epsilon}.
\end{align}
But this implies
\(
    \holder{x - y}_{\alpha-\epsilon}
    = \holder{e}_{\alpha-\epsilon}
    \le 2M^{1-\frac{\epsilon}{\alpha}}\norm{x-y}_\infty^{\frac{\epsilon}{\alpha}}
\), which is the claim.
\end{proof}

\paragraph*{Discrete sewing.} To prove the necessary uniform bound on the Hölder semi-norm
required for Lemma \ref{lem: uniform to holder bound} the key ingredient is the discrete
sewing Lemma \ref{lem: discrete sewing}.

For \(z_i \coloneq \sigma(t_i, w_{t_i}, x^\pi_i)\) we define the \emph{integral approximation}
\[
    A_{t_i, t_j} \coloneq A_{i,j}
    \coloneq z_i (g_{t_j} - g_{t_i})
    \Bigl(\approx \int_{t_i}^{t_j} \sigma(s, w_s, x_s) \, dg_s\Bigr).
\]
Since integrals satisfy the addition property \(\int_{t_i}^{t_j} = \int_{t_i}^s + \int_s^{t_j}\) for any \(s\in [t_i, t_j]\) we denote the
\emph{addition defect} by
\[
    \delta A_{i,k,j}
    \coloneq A_{i,j} - A_{i,k} - A_{k,j}.
\]
Using a bound on this addition defect (Lemma \ref{lem: addition defect}) we can
prove the following discrete sewing lemma that is the key ingredient to obtain a
uniform bound on the Hölder semi-norm of \(x^\pi\) on the intervals spanned by
the anchors (\ref{step: discrete Hölder bound} in the proof of Theorem \ref{thm: convergence of euler scheme}).

\begin{lemma}[Discrete sewing]
    \label{lem: discrete sewing}
    Suppose that \(\sigma\) satisfies Assumption \ref{assmpt: sufficiently nice function}
    and \(\norm{w}_\alpha, \holder{g}_\beta \le R\). Let
    \(z_k = \sigma(t_k, w_{t_k}, x^\pi_k)\) for \(x^\pi_k\) and \(t_k\) as defined in Theorem \ref{thm: convergence of euler scheme}. Then, for \(0\le i<j\le n\),
    \[
        \abs[\Bigg]{\sum_{k=i}^{j-1} z_k(g_{t_{k+1}} - g_{t_k}) - z_i (g_{t_j} - g_{t_i})} 
        = \abs[\Bigg]{\sum_{k=i}^{j-1} A_{k, k+1} - A_{i,j}} 
        \le K_{\mathrm{sew}}^R (1+ \norm{x^\pi}_{\alpha,[i:j]}) \abs{t_j - t_i}^{\alpha+\beta}
    \]
    with \(K_{\mathrm{sew}}^R \coloneq \sum_{k=1}^\infty \bigl(\frac{2}{k}\bigr)^{\alpha+\beta}K_{\delta A}^R \).
\end{lemma}
\begin{proof}
    Observe that the sum over \(A_{k, k+1}\) is essentially a better integral
    approximation than \(A_{i,j}\) that corresponds to just the end-points.
    Let \(t_0 < \dots < t_n\) be the discretization \(\pi\) and
    \[
        \pi_m \coloneq \set{t_i, \dots, t_j} = \set{u_0, \dots, u_m},
        \qquad
        \pi_1 \coloneq \set{t_i, t_j}
    \]
    be two sub-partitions of the discretization \(\pi\) with
    \[
        [\pi_m] = \set{[t_k, t_{k+1}]: i \le k < j},
        \qquad\text{and}\qquad
        [\pi_1] = \set{[t_i, t_j]}.
    \]
    We define the notation
    \[
        S_{\pi} A \coloneq \sum_{[t,s]\in [\pi]} A_{t, s},
    \]
    such that the quantity we want to bound is given by
    \[
        S_{\pi_m}A - S_{\pi_1} A
        = \sum_{k=i}^{j-1} A_{k, k+1} - A_{i,j}
        = \sum_{k=i}^{j-1} z_k (g_{t_{k+1}} - g_{t_k}) - z_i (g_{t_j} - g_{t_i}).
    \]
    We will now construct \(\pi_{m-1}, \dots, \pi_{2}\) by successively dropping
    one \(t_l\) from the partition. Specifically, since
    \[
        (m-1) \min_{0<k<m} (u_{k+1} - u_{k-1}) \le \sum_{0<k<m} (u_{k+1} - u_{k-1}) \le 2(u_m - u_0) = 2(t_j - t_i),
    \]
    we have for \(l \coloneq \argmin_{0<k<m} (u_{k+1} - u_{k-1})\) that
    \begin{equation}
        \label{eq: bound on partition gap}    
        u_{l+1} - u_{l-1} \le \tfrac{2(t_j - t_i)}{m-1}.
    \end{equation}
    We then define \(\pi_{m-1} = \pi_m \setminus \set{u_l}\) and thus have
    \[
        S_{\pi_m} A - S_{\pi_{m-1}} A
        = A_{u_{l-1}, u_l} + A_{u_l, u_{l+1}} - A_{u_{l-1}, u_{l+1}}
        = -\delta A_{u_{l-1}, u_l, u_{l+1}}.
    \]
    With Lemma \ref{lem: addition defect} we thus obtain
    \begin{align}
        \abs{S_{\pi_m} A - S_{\pi_{m-1}} A}
        &= \abs{\delta A_{u_{l-1}, u_l, u_{l+1}}}
        \\
        \overset{\text{Lemma \ref{lem: addition defect}}}&\le
        K_{\delta A}^R (1+ \norm{x^\pi}_{\alpha,[i:j]})\abs{u_{l+1} - u_{l-1}}^{\alpha+\beta}
        \\
        \overset{\eqref{eq: bound on partition gap}}&\le
        2^{\alpha+\beta} K_{\delta A}^R (1+ \norm{x^\pi}_{\alpha,[i:j]})\frac{\abs{t_j- t_i}^{\alpha+\beta}}{(m-1)^{\alpha+\beta}}.
    \end{align}
    Iterating this argument we get
    \begin{align}
        \abs{S_{\pi_m} A - S_{\pi_1} A}
        &\le \sum_{k=1}^{m-1} \abs{S_{\pi_{k+1}} A - S_{\pi_k} A}
        \\
        &\le 2^{\alpha+\beta} K_{\delta A}^R (1+ \norm{x^\pi}_{\alpha,[i:j]})\abs{t_j- t_i}^{\alpha+\beta} \sum_{k=1}^{m-1} \frac{1}{k^{\alpha+\beta}}
        \\
        &\le K_{\mathrm{sew}}^R (1+ \norm{x^\pi}_{\alpha,[i:j]})\abs{t_j- t_i}^{\alpha+\beta}
    \end{align}
    with \(K_{\mathrm{sew}}^R \coloneq \sum_{k=1}^\infty \bigl(\frac{2}{k}\bigr)^{\alpha+\beta}K_{\delta A}^R \).
\end{proof}

What is left to prove is the bound on the addition defect.

\begin{lemma}[Addition defect bound]
    \label{lem: addition defect}
    The addition defect is given by
    \[
        \delta A_{i,k,j} = (z_i - z_k)(g_{t_j} - g_{t_k})
    \]
    and if \(\sigma\) satisfies Assumption \ref{assmpt: sufficiently nice function} 
    and \(\norm{w}_\alpha, \holder{g}_\beta \le R\), then
    \[
        \abs{\delta A_{i,k,j}}
        \le K_{\delta A}^R(1+\norm{x^\pi}_{\alpha,[i:j]})\abs{t_j - t_i}^{\alpha+\beta}
        \qquad \forall 0\le i < k < j \le n.
    \]
    with \(\norm{x^\pi}_{\alpha,[i:j]} = \norm{x^\pi}_{\infty, [i:j]} + \holder{x^\pi}_{\alpha,[i:j]}\) where
    \[
        \holder{x^\pi}_{\alpha,[i:j]} = \sup_{i\le k<l \le j} \frac{\abs{x^\pi_l - x^\pi_k}}{\abs{t_l - t_k}^\alpha}
        \qquad \text{and} \qquad \norm{x^\pi}_{\infty, [i:j]} = \max_{i\le k \le j} \abs{x^\pi_k}
    \]
    and \(K_{\delta A}^R \coloneq 2K_c^R(1+R) R\).
\end{lemma}
\begin{proof}
    The first claim follows directly from the definition
    \begin{align}
        \delta A_{i,k,j}
        &= A_{i,j} - A_{i,k} - A_{k,j}
        \\
        &= z_i (g_{t_j} - \cancel{g_{t_i}}) - z_i (g_{t_k} - \cancel{g_{t_i}}) - z_k (g_{t_j} - g_{t_k})
        \\
        &= (z_i - z_k)(g_{t_j} - g_{t_k}).
    \end{align}
    We thus obtain with \(\holder{g}_\beta \le R\) that
    \[
        \abs{\delta A_{i,k,j}}
        \le \holder{z}_{\alpha,[i:j]} \holder{g}_{\beta} \abs{t_k - t_i}^\alpha \abs{t_j - t_k}^\beta
        \le \holder{z}_{\alpha,[i:j]} R \abs{t_j - t_i}^{\alpha+\beta}.
    \]
    Now for \(i\le k < l \le j\) we have
    \begin{align}
        \abs{z_l - z_k}
        &= \abs[\Big]{\sigma(t_l, w_{t_l}, x^\pi_l) - \sigma(t_k, w_{t_k}, x^\pi_k)}
        \\
        &\le \abs{\sigma(t_l, w_{t_l}, x^\pi_l) - \sigma(t_k, w_{t_l}, x^\pi_l)}
        + \abs{\sigma(t_k, w_{t_l}, x^\pi_l) - \sigma(t_k, w_{t_k}, x^\pi_k)}
        \\
        \overset{\text{Assmpt.~\ref{assmpt: sufficiently nice function}}}&\le K_c^R(1+ \abs{x^\pi_l})\abs{t_l - t_k}^\alpha
        + K_c^R
        \Bigl(
            \abs{x^\pi_l - x^\pi_k}
            + (1+\abs{x^\pi_k}+\abs{x^\pi_l})
            \abs{w_{t_l} - w_{t_k}}
        \Bigr)
        \\
        &\le K_c^R\Bigl(\bigl(1+ \abs{x^\pi_l}+\abs{x^\pi_k}\bigr)(1+ \holder{w}_\alpha)+\holder{x^\pi}_{\alpha,[i:j]}\Bigr)\abs{t_l - t_k}^\alpha
        \\
        &\le 2K_c^R(1+R)\bigl(1+ \underbrace{\norm{x^\pi}_{\infty,[i:j]}+\holder{x^\pi}_{\alpha,[i:j]}}_{= \norm{x^\pi}_{\alpha,[i:j]}}\bigr)\abs{t_l - t_k}^\alpha
    \end{align}
    using \(\holder{w}_\alpha \le R\). We thus have
    \(\holder{z}_{\alpha,[i:j]} \le 2K_c^R(1+R)(1+ \norm{x^\pi}_{\alpha,[i:j]})\) and by definition of \(K_{\delta A}^R\) the claim follows.
\end{proof}
 
\subsection*{Acknowledgements}

The experiments presented in this paper were carried
out using the HPC facilities of the University of Luxembourg
~\citep{VCPKVO_HPCCT22} {\small -- see \url{https://hpc.uni.lu}}. Research supported by the Luxembourg National Research Fund
(Grants: O24/18972745/GFRF and O22/17372844/FraMStA).

\bibliographystyle{abbrvnat}
\bibliography{pDOM,zotero-generated}

\end{document}